\documentclass[oneside,reqno]{amsart}
\usepackage{amsmath,amssymb,amsthm}
\usepackage{cite}
\usepackage{a4wide}
\usepackage{color}
\usepackage{txfonts}
\usepackage{enumerate}
\DeclareMathAlphabet{\mathpzc}{OT1}{pzc}{m}{it}

\newtheorem{theorem}{Theorem}[section]
\newtheorem{lem}[theorem]{Lemma}

\newtheorem{prop}[theorem]{Proposition}
\numberwithin{equation}{section}
\theoremstyle{definition}
\newtheorem{defn}[theorem]{Definition}

\usepackage{hyperref}
\allowdisplaybreaks
\begin{document}

\title[Grand variable Herz-Morrey-Hardy spaces] {Grand variable Herz-Morrey-Hardy Spaces  and applications to the boundedness of integral operators}
\author[ ,B. Sultan, A. Hussain  M. Sultan ]{     Babar Sultan$^{1^*}$, Amjad Hussain$^{1}$ and  Mehvish Sultan$^2$ }%
\address{ $^{1^*}$Department of Mathematics, Quaid-I-Azam University, Islamabad  45320, Pakistan.}

\email{  babarsultan40@yahoo.com, a.hussain@qau.edu.pk }%
\address{$^2$Department of Mathematics, Capital University of Science and Technology, Islamabad, Pakistan.}
\email{mehvishsultanbaz@gmail.com}%

\date{\today}
\dedicatory{}%
\begin{abstract}
In this paper, the concept of grand variable Herz-Morrey-Hardy spaces are introduced. We also establish the atomic characterization of these spaces. As an application the authors investigate the continuity of a few singular integral operators. We also obtain the boundedness of commutators of homogeneous fractional integrals and their Lipschitz estimate  on these spaces by applying the characterization. 
\end{abstract}
\maketitle

\section{Introduction and preliminaries }
Beurling \cite{grafakos2} initially pioneered a rudimentary iteration of Herz spaces for examining convolution algebras. Subsequently, Herz \cite{grafakos15} introduced variations of these spaces within a slightly altered framework. Since then, the theory of Herz spaces has undergone substantial advancement, proving highly advantageous in harmonic analysis. Notably, they were employed by Baernstein and Sawyer \cite{grafakos1} for delineating multipliers on the standard Hardy spaces. Considerable development has occurred in the theory of Hardy spaces within the context of Herz spaces. Chen and Lau \cite{grafakos3} introduced Hardy spaces associated with Beurling algebras $A^q$ on the real line with $1 < q < 2$. Garc\'ia-Cuerva \cite{grafakos10} extended the theory of \cite{grafakos3} to higher dimensions and all $q$ with $1 < q < 1$. Lu and Yang \cite{grafakos18,grafakos19},  established the theory of corresponding homogeneous spaces. More recently, García-Cuerva and Herrero \cite{grafakos11} as well as Lu and Yang \cite{grafakos20,grafakos21,grafakos22,grafakos23} independently advanced the real Hardy space theory for Herz spaces.

\medskip

Many classical function spaces, alongside Hardy-type spaces linked with operators, exhibit atomic and molecular decompositions. These decompositions simplify the action of linear operators on these spaces significantly; see \cite{chai20, chai23, chai24, chai28, chai29, chai30}. Over the past three decades, there has been a growing interest in the exploration of variable exponent spaces. These spaces find applications in diverse fields such as electrorheological fluid dynamics \cite{xu7}, differential equations \cite{xu8} and associated literature, as well as image restoration \cite{xu9,xu10,xu11}, among others. Additionally, the atomic, molecular, and wavelet decompositions of variable exponent Besov and Triebel-Lizorkin spaces have been presented in publications such as  \cite{xu13,xu14,xu20,xu21,xu24}. Reference \cite{xu25} contains discussions on the duality and reflexivity of spaces $B_{p(\cdot), q}^{s}$ and $F_{p(\cdot) q}^{s}$. Moreover, \cite{xu22,xu26} described the atomic and molecular decompositions of Hardy spaces with variable exponents and their uses in proving the boundedness of singular integral operator

\medskip

 Herz spaces with variable exponents have emerged as a generalization of Lebesgue spaces with variable exponents. Boundedness of sublinear operators on Herz spaces with variable exponents, $\dot{K}_{p(\cdot)}^{\alpha, q}$ and $K_{p(\cdot)}^{\alpha, q}$, was shown by Izuki in \cite{n17} in 2010. Boundedness results for a wide class of classical operators on Herz spaces were later developed by Almeida and Drihem in 2012 [28].  Grand variable Herz spaces are the generalization of Herz spaces, for boundedness results in these spaces see \cite{ babargh2,  babargh4}. 
 Wang and Liu presented a specific Herz type Hardy spaces with variable exponent in \cite{xu31}, and in \cite{zliu} they established some real-variable characterizations of Herz-type Hardy spaces with variable exponent. In \cite{yanchai}, authors assumed   a nonnegative, self-adjoint differential operator $L$ and  introduced the Herz-type Hardy spaces associated with  the operator $L$. Then,  they proved the atomic and molecular decompositions of the Herz-type Hardy spaces associated with operator $L$. As applications, the boundedness of some singular integral operators on Herz-type Hardy spaces associated with operators is obtained. The idea of grand variable Herz-Morrey spaces were introduced by Sultan et.al. and obtain boundedness results in these spaces \cite{ghm1, ghm2,  ghm4}. In \cite{hmh},  the authors introduced Herz-Morrey-Hardy spaces with variable exponents and established the characterization of these spaces in terms of atom. The authors were able to determine the boundedness of certain singular integral operators on these spaces by applying the characterization. In this paper, we define the idea of grand variable Herz-Morrey-Hardy Spaces  and using the characterization we obtain boundedness results in these spaces.

The paper's structure is outlined as follows: In the subsequent section, we will review pertinent definitions and concepts. Section 2 will introduce the  grand variable Herz-Morrey-Hardy Spaces , denoted as  $H M\dot{K} ^{\alpha(\cdot),u),\theta}_{ \Gamma, p(\cdot)}(\mathbb{R}^n)$ and  $H M\dot{K} ^{\alpha(\cdot),u),\theta}_{ \Gamma, p(\cdot)}(\mathbb{R}^n)$, and provide their atomic characterization. Section 3 will demonstrate the boundedness of certain singular integral operators from grand variable Herz-Morrey-Hardy Spaces  to grand variable Herz-Morrey spaces, leveraging the atomic characterization. In last section we prove the boundedness of the commutaors of homogenous fractional integrals in these spaces by applying characterizations.

To begin, we provide notation and fundamental definitions regarding Lebesgue spaces with variable exponents.
\medskip
Let $E$ be an open set in  $\mathbb{R}^{n}$, consider  a measurable function $p(\cdot):E\rightarrow[1,\infty)$. The conjugate exponent denoted by $p'(\cdot)$, is  defined as $p'(\cdot)=p(\cdot)/(p(\cdot)-1).$

The set $\mathcal{P}(E)$ comprises all functions $p(\cdot):E\rightarrow[1,\infty) $. We suppose that
	\begin{equation}\label{eq1}
		1\leq p^-(E)\leq p(x) \leq p^+(E)<\infty,
	\end{equation}
 
 such that

\begin{align*}
p^{-}&=\operatorname{ess} \inf\bigl\{ p(x):x\in E\bigr\} >1, \\ p^{+}&=\operatorname{ess} \sup\bigl\{ p(x):x\in E\bigr\} < \infty.
\end{align*}

\medskip

We use the notation $L^{p(\cdot)}(E)$ to represent the space of all measurable functions $g$ defined on $E$, such that, for a certain $\eta>0$, $\int_{E} \biggl(\frac{ \vert f(x) \vert }{\eta} \biggr)^{p(y)}\,dy < \infty.$
This is a Banach function space that has the Luxemburg-Nakano norm 

$$\Vert f \Vert _{L^{p(\cdot)}(E)}=\inf \biggl\{ \eta>0: \int_{E} \biggl(\frac{ \vert f(y) \vert }{\eta} \biggr)^{p(y)}\,dy \leq1 \biggr\}.$$

\medskip

The local space $L_{\mathrm{loc}}^{p(\cdot)}(\varOmega)$ is characterized as  $L_{\mathrm{loc}}^{p(\cdot)}(\varOmega):=\{f: f\in L^{p(\cdot)}(E)$ for all compact subsets ${E\subset\varOmega\}}$.

Assume $f\in L_{\mathrm{loc}}^{1}(\mathbb{R}^{n})$, the definition of the Hardy-Littlewood maximal operator is expressed as

$$Mf(x)=\sup_{r>0}\frac{1}{ \vert B_{r}(x) \vert } \int_{B_{r}(x)} \bigl\vert f(y) \bigr\vert \,dy,$$
where $B_{r}(x)=\{y\in\mathbb{R}^{n}:|x-y|< r\}.$

The set $\mathcal{B}(\mathbb{R}^{n})$  is comprised of $p(\cdot)\in\mathcal{P}(\mathbb{R}^{n})$ that fulfill the requirement that $M $ is bounded on  $L^{p(\cdot)}(\mathbb{R}^{n}).$

Now we will define the well known $\log$-condition
\begin{equation}\label{r5}
	{|p(h_1)-p(h_2)|\leq \frac{C(p)}{-\ln|h_1-h_2|}, \;\;\;|h_1-h_2|\leq \frac{1}{2},\; \;\; h_1,h_2\; \in E},
\end{equation}
where $C(p)>0$. And the decay condition: there exists a number $p_\infty \in (1, \infty)$, such that

\begin{equation}\label{r6}
	|p(h)-p_\infty|\leq \frac{C}{\ln(e+|h|)},
\end{equation}

and also decay condition
\begin{equation}\label{r7}
	|p(h)-p_0|\leq \frac{C}{\ln|h|},|h|\leq\frac{1}{2},
\end{equation}

holds for some $p_0 \in (1, \infty)$.
We use these notations in this article:
\begin{enumerate}
	\item[(i)] The set $\mathcal{P}(E)$  {consists} of all measurable functions $p(\cdot)$ satisfying	$	p^-> 1$ and $p^+< \infty.$
	
	\item[(ii)] $\mathcal{P}^{\log} = \mathcal{P}^{\log}(E)$  {consists} of all functions $p \in \mathcal{P}(E)$ satisfying \eqref{eq1}  and \eqref{r5}.
	
	\item[(iii)]  $\mathcal{P}_\infty (E)$ and $\mathcal{P}_{0,\infty} (H)$ are the subsets of $\mathcal{P}(E)$ and  values of these subsets lies in $[1, \infty)$  which  satisfy the  condition \eqref{r6} and both conditions \eqref{r6} and \eqref{r7} respectively. 
	\item[(iv)]  $$\chi _ l =\chi_{F_l},\ F_l=B_l \setminus B_{m_0-1},\ B_l=B(0,2^l)=\{x \in \mathbb{R}^n : |x| < 2^l\}$$ for all $l\in  \mathbb{Z}$.
 \item[(v)] Let $E$ be a measurable subset in $\mathbb{R}^{n}$, then $|E|$ denotes the Lebesgue measure and $\chi_{E}$ the characteristic function.
    \item[(vi)] Let $\beta=(\beta_1,\beta_2,\cdots,\beta_n)$ then $|\beta|$ is defined as $|\beta|=\beta_1,\beta_2+\cdots+\beta_n$.
    \item[(vii)] The symbol $\mathbb{N}_{0}$ denotes the set of all nonnegative integers. For $m \in \mathbb{N}_{0}$, we denote $\tilde{\chi}_{m}:=\chi_{F_{m}}$ if $m \geq 1$ and $\tilde{\chi}_{0}:=\chi_{B_{0}}$. 
    \item[(viii)] $C$ is a positive constant.
    \item[(ix)] By $a \lesssim b$, we mean $a \leq C b$. 
\end{enumerate}

\begin{lem}\cite{n19}

Let $g$ belong to the Lebesgue space $L^{p(\cdot)}(\mathbb{R}^{n})$ and $g$ belong to $L^{p'(\cdot)}(\mathbb{R}^{n})$, where $p(\cdot)\in \mathcal{P}(\mathbb{R}^{n})$. Then, the product function $fg$ is integrable over $\mathbb{R}^{n}$, and the following inequality holds:

$$\int_{\mathbb{R}^{n}} \bigl\vert f(x)g(x) \bigr\vert \,dx \leq r_{p} \Vert f \Vert _{{p(\cdot )}} \Vert g \Vert _{{p'(\cdot)}},$$
where $r_{p}$ is defined as $r_{p}=1+\frac{1}{p^{-}} -\frac{1}{p^{+}}$.
\end{lem}

\begin{lem}\cite{n17}

Let $p(\cdot)$ be a function within the class $\mathcal{B}(\mathbb{R}^{n})$. Then there exists a positive constant $C$ such that the inequality

$$\frac{1}{ \vert B \vert } \Vert \chi_{B} \Vert _{{p(\cdot)}} \Vert \chi_{B} \Vert _{{p'(\cdot)}}\leq C$$
holds.
\end{lem}

\begin{lem}\cite{n17}

Assuming that $p(\cdot)$ is a function in the class $\mathcal{B}(\mathbb{R}^{n}),$ there exists a positive constant $C$ such that, for every ball $B$  and every measurable subset $S$ within $B,$ the following inequalities hold:

$$\frac{ \Vert \chi_{B} \Vert _{{p(\cdot)}}}{ \Vert \chi_{S} \Vert _{{p(\cdot)}}}\leq C \frac{ \vert B \vert }{ \vert S \vert },\;\;\; \frac{ \Vert \chi_{S} \Vert _{{p(\cdot)}}}{ \Vert \chi_{B} \Vert _{{p(\cdot)}}}\leq C \biggl( \frac{ \vert S \vert }{ \vert B \vert } \biggr)^{\delta_{1}}, \;\;\;\frac{ \Vert \chi_{S} \Vert _{{p'(\cdot)}}}{ \Vert \chi_{B} \Vert _{{p'(\cdot)}}}\leq C \biggl(\frac{ \vert S \vert }{ \vert B \vert } \biggr)^{\delta_{2}},$$

where $\delta_{1}$ and $\delta_{2}$ are constants satisfying $0<\delta_{1}, \delta_{2}<1$.
\end{lem}

Next we  give the definition of grand variable Herz-Morrey  spaces (GVHM).

\begin{defn}
	 Let  $p :\mathbb{R}^n \rightarrow [1,\infty)$, $\alpha(\cdot) \in L^\infty(\mathbb{R}^n)$, $ u \in [1, \infty)$,  $\theta >0$, and $0 \leq \Gamma < \infty$. The norm of GVHM is defined as:
	$$M\dot{K} ^{\alpha(\cdot),u),\theta}_{ \Gamma, p(\cdot)}(\mathbb{R}^n)=\left\{ f \in L^{p(\cdot)} _{\mathrm{loc}}(\mathbb{R}^n\setminus\{0\}): \| f \| _ {M\dot{K} ^{\alpha(\cdot),u),\theta}_{ \Gamma, p(\cdot)}(\mathbb{R}^n)} < \infty \right\},$$ 
	where 
	$$\| f \| _ {M\dot{K} ^{\alpha(\cdot),u),\theta}_{ \Gamma, p(\cdot)}(\mathbb{R}^n)}=\sup \limits_{\vartheta>0 } \sup \limits _{ m_0 \in \mathbb{Z}} 2^{-m_0\Gamma } \left( \vartheta   ^\theta  \sum \limits _{{k = -\infty}}^{m_0}  \|2^{k \alpha(\cdot) } f \chi _k \| ^{u(1+\vartheta)}_{{p(\cdot)}}\right)^{\frac{1}{u (1+\vartheta)}}.$$
	For $\Gamma=0$, grand Herz-Morrey  spaces becomes grand Herz spaces.
\end{defn}

If sublinears operators satisfy the size condition
\begin{equation}\label{size condition}
|T
f(z_1) |\leq C \int \limits _{\mathbb{R}^n} |x-y|^{-n}|f(y)|\,d y, \;\;\; x \notin spt f
\end{equation}
for all $f \in L^1(\mathbb{R}^n)$ with compact support then the non-homogeneous and homogeneous Herz spaces have bounds for $T$.  
\begin{theorem}\cite{ghm5}
	Let $  1 \leq u < \infty$, $\alpha(\cdot), p(\cdot) \in \mathcal{P}_{0,\infty}(\mathbb{R}^n) $ such that
	\begin{enumerate}
		\item [(i)] $ \frac{n}{p(0)}< \alpha(0) <    \frac{n}{p'(0)}$
		\item [(ii)] $\frac{n}{p_{\infty}} < \alpha_\infty < \frac{n}{p'_{\infty} }$.
	\end{enumerate}
	 Suppose that  sublinear operators $T$  are bounded on  Lebesgue spaces and satisfying size condition (\ref{size condition}) will be bounded on $M\dot{K} ^{\alpha(\cdot),u),\theta}_{ \Gamma, p(\cdot)}(\mathbb{R}^n)$.
\end{theorem}

The next Proposition is the generalization of variable exponents Herz spaces in \cite{cald15}. We omit the proof of Proposition 2.5 since is essentially similar to the proof given in \cite{cald15} and with slight modification we can obtain following result in grand variable Herz spaces. 
\begin{prop}
	Let $\alpha, u, p$ are as defined in definition $1.4$, then 
	\begin{align*}
	\| f \|_{M\dot{K} ^{\alpha(\cdot),u),\theta}_{ \Gamma, p(\cdot)}(\mathbb{R}^n)}= &  \sup \limits_{\epsilon   >0} \sup \limits _{ m_0 \in \mathbb{Z}} 2^{-m_0\Gamma  } \left( \vartheta   ^\theta \sum_{k =-\infty }^{m_0} \|2^{k \alpha(\cdot) }f \chi _k  \|  ^{u(1+\vartheta)}_{{p(\cdot)}}\right)^{\frac{1}{u(1+\vartheta)}} \\
 \approx & \max \left\{ \sup \limits_{\epsilon   >0} \sup \limits _{ m_0< 0, m_0 \in \mathbb{Z}} 2^{-m_0\Gamma } \left( \vartheta   ^\theta \sum_{k =-\infty }^{-1}2^{k \alpha(0) u(1+\vartheta)} \|f \chi _k  \|  ^{u(1+\vartheta)}_{{p(\cdot)}}\right)^{\frac{1}{u(1+\vartheta)}},\right.\\
	& \left. \sup \limits_{\epsilon   >0} \sup \limits _{ m_0\geq 0, m_0\in  \mathbb{Z}} 2^{-m_0\Gamma } \left( \vartheta   ^\theta \sum_{k =-\infty }^{-1}2^{k \alpha(0) u(1+\vartheta)} \|f \chi _k  \|  ^{u(1+\vartheta)}_{{p(\cdot)}}\right)^{\frac{1}{u(1+\vartheta)}} \right.\\
 & \left. + \sup \limits_{\epsilon   >0} \sup \limits _{ m_0\geq 0, m_0\in  \mathbb{Z}} 2^{-m_0\Gamma } \left( \vartheta   ^\theta \sum_{k =0 }^{m_0}2^{k \alpha_\infty u(1+\vartheta)} \|f \chi _k  \|  ^{u(1+\vartheta)}_{{p(\cdot)}}\right)^{\frac{1}{u(1+\vartheta)}} \right\}
	\end{align*}
\end{prop}

\section{ The Atomic Characterization}
In this section, we will introduce grand variable Herz-Morrey-Hardy Spaces  $H M\dot{K} ^{\alpha(\cdot),u),\theta}_{ \Gamma, p(\cdot)}(\mathbb{R}^n)$. To do this, we need to recall some notations. $\mathcal{S}\left(\mathbb{R}^{n}\right)$ denotes the Schwartz space of all rapidly decreasing infinitely differentiable functions on $\mathbb{R}^{n}$, and $\mathcal{S}^{\prime}\left(\mathbb{R}^{n}\right)$ denotes the dual space of $\mathcal{S}\left(\mathbb{R}^{n}\right)$. Let $G_{N}g $ be the grand maximal function of $g$ defined by

\begin{equation*}
G_{N}g (x):=\sup _{\phi \in \mathcal{A}_{N}}\left|\phi_{\nabla}^{*}(g)(x)\right|, \quad x \in \mathbb{R}^{n} 
\end{equation*}

where $\mathcal{A}_{N}:=\left\{\phi \in \mathcal{S}\left(\mathbb{R}^{n}\right): \sup \limits_{|\alpha|,|\beta| \leq  N, \forall x \in \mathbb{R}^{n}}\left|x^{\alpha} D^{\beta} \phi(x)\right| \leq  1\right\}$ and $N>n+1$ and $\phi_{\nabla}^{*}$ is the nontangential maximal operator defined by

$$\phi_{\nabla}^{*}(g)(x):=\sup \limits _{|y-x|<t}\left|\phi_{t} * g(y)\right|,$$

with $\phi_{t}(x)=t^{-n} \phi\left(\frac{x}{t} \right).$

\begin{defn}
    Let $\alpha(\cdot) \in L^{\infty}\left(\mathbb{R}^{n}\right), 0<q \leq \infty, p(\cdot) \in \mathcal{P}\left(\mathbb{R}^{n}\right)$, $0 \leq \Gamma<\infty$, and $N>n+1$. The grand Herz-Morrey-Hardy space with variable exponents $H M\dot{K} ^{\alpha(\cdot),u),\theta}_{ \Gamma, p(\cdot)}(\mathbb{R}^n)$ is defined  by
$$H M \dot{K} ^{\alpha(\cdot),u),\theta}_{ \Gamma, p(\cdot)}(\mathbb{R}^n) \\
\quad:=\left\{g \in \mathcal{S}^{\prime}\left(\mathbb{R}^{n}\right):\|g\|_{HM\dot{K} ^{\alpha(\cdot),u),\theta}_{ \Gamma, p(\cdot)}(\mathbb{R}^n)}:=\left\|G_{N}g \right\|_{M\dot{K} ^{\alpha(\cdot),u),\theta}_{ \Gamma, p(\cdot)}(\mathbb{R}^n)}<\infty\right\}.$$

\end{defn}

\begin{defn}

Let $p(\cdot) \in \mathcal{P}\left(\mathbb{R}^{n}\right)$ and $\alpha(\cdot) \in L^{\infty}\left(\mathbb{R}^{n}\right)$ be log-H\"older continuous both at the origin and infinity, and nonnegative integer $s \geqslant\left[\alpha_{r}-n \delta_{2}\right]$; here $\alpha_{r}=\alpha(0)$, if $r<1$, and $\alpha_{r}=\alpha_{\infty}$, if $r \geqslant 1, n \delta_{2} \leq  \alpha_{r}<\infty$ and $\delta_{2}$ as in Lemma 1.3.

(i) A function $a$ on $\mathbb{R}^{n}$ is called a central $(\alpha(\cdot), p(\cdot))$ atom, if it satisfies

\begin{enumerate}
    \item [(1)]  supp $a \subset B(0, r)$,
    \item [(2) ] $\|a\|_{{p(\cdot)}} \leq $ $|B(0, r)|^{-\alpha_{r} / n} $,
    \item [(3)] $\int_{\mathbb{R}^{n}} a(x) x^{\beta} \mathrm{d} x=0,|\beta| \leq  s$.
\end{enumerate}

(ii)  A function $a$ on $\mathbb{R}^{n}$ is called a central $(\alpha(\cdot), p(\cdot))$-atom of restricted type, if it satisfies $2,3$ and condition given below

\begin{enumerate}
    \item [(a)]supp $\alpha \subset B(0, r)$, $r \geqslant 1$.
\end{enumerate}
 
\end{defn}

\begin{theorem}

 Let $0<u<\infty, p(\cdot) \in \mathcal{B}\left(\mathbb{R}^{n}\right), 0 \leq \Gamma<\infty$, and $\alpha(\cdot) \in L^{\infty}\left(\mathbb{R}^{n}\right)$ be log-H\"older continuous both at the origin and infinity, $2 \lambda \leq  \alpha(\cdot), n \delta_{2} \leq  \alpha(0), \alpha_{\infty}<\infty$, and $\delta_{2}$ as in Lemma 1.3. Then

 $g \in H M \dot{K} ^{\alpha(\cdot),u),\theta}_{ \Gamma, p(\cdot)}(\mathbb{R}^n)$ iff $g=\sum_{k=-\infty}^{\infty} \lambda_{k} a_{k}$ in the sense of $\mathcal{S}^{\prime}\left(\mathbb{R}^{n}\right)$, where each $a_{k}$ is a central $(\alpha(\cdot), p(\cdot))$-atom with support contained in $B_{k}$ and $\sup\limits_{\vartheta>0 }\sup\limits  _{m_0\in  \mathbb{Z}} 2^{-m_0\Gamma} \epsilon  ^\theta\sum_{k=-\infty}^{m_0}\left|\lambda_{k}\right|^{u(1+\vartheta)}<\infty$. Moreover,

\begin{equation*}
\|g\|_{H M \dot{K} ^{\alpha(\cdot),u),\theta}_{ \Gamma, p(\cdot)}(\mathbb{R}^n)} \approx \inf \left(\sup_{\vartheta>0 }\sup _{m_0\in  \mathbb{Z}} 2^{-m_0\Gamma}\left(\epsilon  ^\theta\sum_{k=-\infty}^{m_0}\left|\lambda_{k}\right|^{u(1+\vartheta)}\right)^{1 / (u(1+\vartheta))} \right).
\end{equation*}

\begin{proof}

Firstly we will prove that $1\Rightarrow 2$. Let $\phi \in$ $C_{0}^{\infty}\left(\mathbb{R}^{n}\right)$ (the space of continuously differentiable complex-valued functions) such that $\phi \geqslant 0, \int_{\mathbb{R}^{n}} \phi(x) \mathrm{d} x=1$, and supp $\phi \subset\{x$ : $|x| \leq  1\}$.

Let $x \in \mathbb{R}^{n}$, $i\in  \mathbb{N}_{0}$, and $$\phi_{(i)}(x):=2^{\text {in }} \phi\left(2^{i} x\right).$$

Let $x \in \mathbb{R}^{n}$, for each $g \in \mathcal{S}^{\prime}\left(\mathbb{R}^{n}\right)$, set $$g^{(i)}(x)=g * \phi_{(i)}(x).$$

It is obvious that $g^{(i)} \in C^{\infty}\left(\mathbb{R}^{n}\right)$ and $\lim _{j \rightarrow \infty} g^{(i)}=g$ in $\mathcal{S}^{\prime}\left(\mathbb{R}^{n}\right)$. Let $\epsilon  $ be a radial smooth function such that supp $\epsilon   \subset\{x$ : $1 / 2-\varepsilon \leq |x| \leq  1+\varepsilon\}$ with $0<\varepsilon<1 / 4$, $\epsilon  (x)=1$ for $1 / 2 \leq |x| \leq  1$. Let $\epsilon  _{k}(x):=\epsilon  \left(2^{-k} x\right)$ for $k \in \mathbb{Z}$ and

\begin{equation*}
\widetilde{H}^k_{ \varepsilon}:=\left\{x: 2^{k-1}-2^{k} \varepsilon \leq |x| \leq  2^{k}+2^{k} \varepsilon\right\} . 
\end{equation*}

It is easy to see that supp $\epsilon  _{k} \subset \widetilde{H}^k_{ \varepsilon}$ and $\epsilon  _{k}(x)=1$ for $x \in A_{k}:=$ $\left\{x: 2^{k-1} \leq |x| \leq  2^{k}\right\}$.

Obviously, $1 \leq  \sum_{k=-\infty}^{\infty} \epsilon  _{k}(x) \leq  2, x \neq 0$. Let

\[
\Phi_{k}(x):= \begin{cases}\frac{\epsilon  _{k}(x)}{\sum_{\ell=-\infty }^{\infty} \epsilon  _{l}(x)}, & x \neq 0  \\ 0, & x=0\end{cases}
\]

then $\sum_{k\in \mathbb{Z}} \Phi_{k}(x)=1$ for $x \neq 0$. For each $m \in \mathbb{N}$, we denote by $\mathcal{P}_{m}$ the class of all the real polynomials with the degree less than $m$.

Let $|\beta| \leq  m=\max \left\{\left[\alpha(0)-n \delta_{2}\right],\left[\alpha_{\infty}-n \delta_{2}\right]\right\}$ and  $P^{(i)}_{k}(x):=P_{\widetilde{H}^k_{ \varepsilon}}\left(g^{(i)} \Phi_{k}\right)(x) \chi_{\widetilde{H}^k_{ \varepsilon}} \in \mathcal{P}_{m}\left(\mathbb{R}^{n}\right)$ be the unique polynomial satisfying

\begin{gather*}
\int_{\widetilde{H}^k_{ \varepsilon}}\left(g^{(i)}(x) \Phi_{k}(x)-P^{(i)}_{k}(x)\right) x^{\beta} \mathrm{d} x=0.
\end{gather*}

Write

\begin{align*}
g^{(i)}(x) & =\sum_{k=-\infty}^{\infty}\left(g^{(i)}(x) \Phi_{k}(x)-P^{(i)}_{k}(x)\right)+\sum_{k=-\infty}^{\infty} P^{(i)}_{k}(x)  :=I_{(i)}+I I_{(i)}. 
\end{align*}

Consider the first term $I_{(i)}=\sum_{k=-\infty}^{\infty}\left(g^{(i)}(x) \Phi_{k}(x)-P^{(i)}_{k}(x)\right).$

Suppose that $a_{k}^{(i)}(x)=\frac{g_{k}^{(i)}(x)} {\lambda_{k}}$, where we define $\lambda_k$ as $\lambda_{k} \quad:=$ $b\left|B_{k+1}\right|^{\alpha_{k+1} / n} \sum_{\ell=k-1}^{k+1}\left\|\left(G_{N}g \right) \chi_\ell \right\|_{{p(\cdot)}}$ and $b$ is a constant which will be chosen later, let $g_{k}^{(i)}(x):=g^{(i)}(x) \Phi_{k}(x)-$ $P^{(i)}_{k}(x)$. It is easy to check that  supp $a_{k}^{(i)} \subset B_{k+1}, I_{(i)}=$ $\sum_{k=-\infty}^{\infty} \lambda_{k} a_{k}^{(i)}(x)$.

Next we will find estimate for 
 $\left\|g_{k}^{(i)}\right\|_{{p(\cdot)}}$.

To do this, let $\left\{\phi_{\gamma}^{k}:|\gamma| \leq  m\right\}$ be the orthogonal polynomials restricted to $\widetilde{H}^k_{ \varepsilon}$ with respect to the weight $\frac{1}{\left|\widetilde{H}^k_{ \varepsilon}\right|}$, which are obtained from $\left\{x^{\beta}:|\beta| \leq  m\right\}$ by the Gram-Schmidt method, which means

\begin{equation*}
\left\langle\phi_{\nu}^{k}, \phi_{\mu}^{k}\right\rangle=\frac{1}{\left|\widetilde{H}^k_{ \varepsilon}\right|} \int_{\widetilde{H}^k_{ \varepsilon}} \phi_{\nu}^{k}(x) \phi_{\mu}^{k}(x) \mathrm{d} x=\delta_{\nu \mu}, 
\end{equation*}

where $\delta_{\nu \mu}=1$ for $\nu=\mu$, otherwise 0 .

If $x \in \widetilde{H}^k_{ \varepsilon}$, then we have
$$P^{(i)}_{k}(x)=\sum_{|\gamma| \leq  m}\left\langle g^{(i)} \Phi_{k}, \phi_{\gamma}^{k}\right\rangle \phi_{\gamma}^{k}(x).$$

On the other hand, by using the fact $$\left(\frac{1} {\left|\widetilde{H}^k_{ \varepsilon}\right|}\right) \int_{\widetilde{H}^k_{ \varepsilon}} \phi_{\nu}^{k}(x) \phi_{\mu}^{k}(x) \mathrm{d} x=\delta_{\nu \mu},$$ we get

\begin{equation*}
\frac{1}{\left|\widetilde{H}_{1, \varepsilon}\right|} \int_{\widetilde{H}_{1, \varepsilon}} \phi_{\nu}^{k}\left(2^{k-1} y\right) \phi_{\mu}^{k}\left(2^{k-1} y\right) \mathrm{d} y=\delta_{\nu \mu}. 
\end{equation*}

Thus, we deduce $\phi_{v}^{k}\left(2^{k-1} y\right)=\phi_{\nu}^{1}(y)$ a.e.

That is, $\phi_{\nu}^{k}(x)=$ $\phi_{\nu}^{1}\left(2^{1-k} x\right)$ almost everywhere for $x \in \widetilde{H}^k_{ \varepsilon}$. Therefore
$\left|\phi_{\nu}^{k}(x)\right| \leq  C$ for $x \in \widetilde{H}^k_{ \varepsilon}$.

By the generalized H\"older inequality, we get

\begin{align*}
\left|P^{(i)}_{k}(x)\right| & \lesssim \frac{1}{\left|\widetilde{H}^k_{ \varepsilon}\right|} \int_{\widetilde{H}^k_{ \varepsilon}}\left|g^{(i)}(x) \Phi_{k}(x)\right| \mathrm{d} x \\
& \lesssim \frac{1}{\left|\widetilde{H}^k_{ \varepsilon}\right|}\left\|g^{(i)} \Phi_{k}\right\|_{{p(\cdot)}}\left\|\chi_{\widetilde{H}^k_{ \varepsilon}}\right\|_{{p^{\prime}(\cdot)}} .
\end{align*}

By Lemma 1.2, we have

\begin{align*}
\left\|g_{k}^{(i)}\right\|_{{p(\cdot)}} \lesssim & \left\|g^{(i)} \Phi_{k}\right\|_{{p(\cdot)}}+\left\|P^{(i)}_{k}\right\|_{{p(\cdot)}} \\
\lesssim & \left\|g^{(i)} \Phi_{k}\right\|_{{p(\cdot)}}+\frac{1}{\left|\widetilde{H}^k_{ \varepsilon}\right|}\left\|g^{(i)} \Phi_{k}\right\|_{{p(\cdot)}} \left\|\chi_{\widetilde{H}^k_{ \varepsilon}}\right\|_{{\left.p^{(} \cdot\right)}}\left\|\chi_{\widetilde{H}^k_{ \varepsilon}}\right\|_{{p(\cdot)}} \\
\lesssim & \left\|g^{(i)} \Phi_{k}\right\|_{{p(\cdot)}}+\left\|g^{(i)} \Phi_{k}\right\|_{{p(\cdot)}} \\
 \lesssim & \left\|\left(g * \phi_{(i)}\right) \Phi_{k}\right\|_{{p(\cdot)}} \\
\leq  & C \sum_{\ell=k-1}^{k+1}\left\|\left(G_{N} g\right) \chi_\ell \right\|_{{p(\cdot)}} .
\end{align*}

Choose $b=C$; then $\left\|a_{k}^{(i)}\right\|_{{p(\cdot)}} \leq \left|B_{k+1}\right|^{\frac{-\alpha_{k+1}}{n}}$ and each $a_{k}^{(i)}$ is a central $(\alpha(\cdot), p(\cdot))$-atom with support contained in $B_{k+1}$.  Thus,

\begin{align*}
& \sup_{\vartheta>0 } \sup _{m_0\in  \mathbb{Z}} 2^{-m_0\Gamma u(1+\vartheta)}\epsilon  ^\theta \sum_{k=-\infty}^{m_0}\left|\lambda_{k}\right|^{u(1+\vartheta)} \\
& \quad \leq  C \sup_{\vartheta>0 } \sup _{m_0\in  \mathbb{Z}} 2^{-m_0\Gamma u(1+\vartheta)}\epsilon  ^\theta \left( \sum_{k=-\infty}^{m_0}\left|B_{k+1}\right|^{ \frac{\alpha_{k+1}} {n}}\sum_{\ell=k-1}^{k+1}\left\|\left(G_{N}g \right) \chi_\ell \right\|_{{p(\cdot)}}\right)^{u(1+\vartheta)} \\
& \quad \leq  C \sup_{\vartheta>0 } \sup _{m_0\in  \mathbb{Z}} 2^{-m_0\Gamma u(1+\vartheta)}\epsilon  ^\theta \sum_{k=-\infty}^{m_0}\left|B_{k+1}\right|^{\frac{u(1+\vartheta) \alpha_{k+1}} { n}}\left\|\left(G_{N}g \right) \chi_{k}\right\|_{{p(\cdot)}}^{u(1+\vartheta)} \\
& \quad:=A . 
\end{align*}

Now we estimate $A$. By the condition of $\alpha(\cdot)$ we consider it in two cases.

Case $1$: When $(m_0<0)$. Consider

\begin{align*}
& \sup_{\vartheta>0 } \sup _{m_0\in  \mathbb{Z}, m_0\leq 0} 2^{-m_0\Gamma u(1+\vartheta)}\epsilon  ^\theta \sum_{k=-\infty}^{m_0}\left|B_{k+1}\right|^{\frac{u(1+\vartheta) \alpha_{k+1}} { n}}\left\|\left(G_{N}g \right) \chi_{k}\right\|_{{p(\cdot)}}^{u(1+\vartheta)} \\
& \quad \leq  C \sup_{\vartheta>0 } \sup _{m_0\in  \mathbb{Z}, m_0\leq 0} 2^{-m_0\Gamma u(1+\vartheta)}\epsilon  ^\theta \sum_{k=-\infty}^{m_0} 2^{k u(1+\vartheta) \alpha(0)}\left\|\left(G_{N}g \right) \chi_{k}\right\|_{{p(\cdot)}}^{u(1+\vartheta)}  \\
& \quad \leq C\left\|G_{N}g \right\|_{M \dot{K} ^{\alpha(\cdot),u),\theta}_{ \Gamma, p(\cdot)}(\mathbb{R}^n)}^{u(1+\vartheta)} .
\end{align*}

Case $1$: When $(m_0\geq 0)$. Consider

\begin{align*}
&\sup_{\vartheta>0 } \sup _{m_0\in  \mathbb{Z}} 2^{-m_0\Gamma u(1+\vartheta)}\epsilon  ^\theta \sum_{k=-\infty}^{m_0}\left|B_{k+1}\right|^{\frac{u(1+\vartheta) \alpha_{k+1}} { n}}\left\|\left(G_{N}g \right) \chi_{k}\right\|_{{p(\cdot)}}^{u(1+\vartheta)} \\
& \leq C \sup_{\vartheta>0 } \sup _{m_0\in  \mathbb{Z}} 2^{-m_0\Gamma u(1+\vartheta)}\epsilon  ^\theta \sum_{k=-\infty}^{0}\left|B_{k+1}\right|^{\frac{u(1+\vartheta) \alpha_{k+1}} { n}}\left\|\left(G_{N}g \right) \chi_{k}\right\|_{{p(\cdot)}}^{u(1+\vartheta)} \\
& +\sup_{\vartheta>0 } \sup _{m_0\in  \mathbb{Z}} 2^{-m_0\Gamma u(1+\vartheta)}\epsilon  ^\theta \sum_{k=-0}^{m_0}\left|B_{k+1}\right|^{\frac{u(1+\vartheta) \alpha_{k+1}} { n}}\left\|\left(G_{N}g \right) \chi_{k}\right\|_{{p(\cdot)}}^{u(1+\vartheta)} \\
& \leq C \sup_{\vartheta>0 } \sup _{m_0\in  \mathbb{Z}} 2^{-m_0\Gamma u(1+\vartheta)}\epsilon  ^\theta \sum_{k=-\infty}^{0}2^{k u(1+\vartheta) \alpha(0)}\left\|\left(G_{N}g \right) \chi_{k}\right\|_{{p(\cdot)}}^{u(1+\vartheta)} \\
& +C\sup_{\vartheta>0 } \sup _{m_0\in  \mathbb{Z}} 2^{-m_0\Gamma u(1+\vartheta)}\epsilon  ^\theta \sum_{k=-0}^{m_0}2^{k u(1+\vartheta) \alpha_\infty}\left\|\left(G_{N}g \right) \chi_{k}\right\|_{{p(\cdot)}}^{u(1+\vartheta)} \\
 & \leq C\left\|G_{N}g \right\|_{M \dot{K} ^{\alpha(\cdot),u),\theta}_{ \Gamma, p(\cdot)}(\mathbb{R}^n)}^{u(1+\vartheta)} .
\end{align*}

Hence,

\begin{equation*}
\sup_{\vartheta>0 } \sup _{m_0\in  \mathbb{Z}} 2^{-m_0\Gamma u(1+\vartheta)}\epsilon  ^\theta \sum_{k=-\infty}^{m_0}\left|\lambda_{k}\right|^{u(1+\vartheta)} \lesssim\left\|G_{N}g \right\|_{M \dot{K} ^{\alpha(\cdot),u),\theta}_{ \Gamma, p(\cdot)}(\mathbb{R}^n)}^{u(1+\vartheta)} .
\end{equation*}

Next we find the estimate for $I I_{(i)}$.

Let $\left\{\epsilon  _{d}^{k}:|\gamma| \leq  m\right\}$ be the dual basis of $\left\{x^{\beta}:|\beta| \leq  m\right\}$ with respect to the weight $\frac{1} {\left|\widetilde{H}^k_{ \varepsilon}\right|}$ on $\widetilde{H}^k_{ \varepsilon}$, that is,

\begin{equation*}
\left\langle\epsilon  _{\gamma}^{k}, x^{\beta}\right\rangle=\frac{1}{\left|\widetilde{H}^k_{ \varepsilon}\right|} \int_{\widetilde{H}^k_{ \varepsilon}} x^{\beta} \epsilon  _{\gamma}^{k}(x) \mathrm{d} x=\delta_{\beta \gamma} . 
\end{equation*}

By applying similar methods of \cite{48},
let

\begin{align*}
h_{k, \gamma}^{(i)}(x):=  \sum_{\ell=-\infty }^{k}\left(\frac{\epsilon  _{\gamma}^{k}(x) \chi_{\widetilde{H}^k_{ \varepsilon}}(x)}{\left|\widetilde{H}^k_{ \varepsilon}\right|}-\frac{\epsilon  _{\gamma}^{k+1}(x) \chi_{\widetilde{H}^{k+1}_{ \varepsilon}}(x)}{\left|\widetilde{H}^{k+1}_{ \varepsilon}\right|}\right)  \int_{\mathbb{R}^{n}} g^{(i)}(y) \Phi_\ell (y) y^{\gamma} \mathrm{d} y . 
\end{align*}

It is easy to check that

\begin{align*}
\int_{\mathbb{R}^{n}} \sum_{\ell=-\infty }^{k}\left|\Phi_\ell (x) x^{\gamma}\right| \mathrm{d} x & =\sum_{\ell=-\infty }^{k} \int_{\widetilde{H}^k_{ \varepsilon}}\left|\Phi_\ell (x) x^{\gamma}\right| \mathrm{d} x  \\
& \lesssim 2^{k(n+|\gamma|)} .
\end{align*}

Therefore, we get

\begin{align*}
I I_{(i)}= & \sum_{k=-\infty}^{\infty} \sum_{|\gamma| \leq  m}\left\langle g^{(i)} \Phi_{k}, x^{\gamma}\right\rangle \epsilon  _{\gamma}^{k}(x) \chi_{\widetilde{H}^k_{ \varepsilon}}(x) \\
= & \sum_{|\gamma| \leq  m} \sum_{k=-\infty}^{\infty}\left(\int_{\mathbb{R}^{n}} g^{(i)} \Phi_{k} x^{\gamma} \mathrm{d} x\right) \frac{\epsilon  _{\gamma}^{k}(x) \chi_{\widetilde{H}^k_{ \varepsilon}}(x)}{\left|\widetilde{H}^k_{ \varepsilon}\right|} \\
= & \sum_{|\gamma| \leq  m} \sum_{k=-\infty}^{\infty}\left(\sum_{\ell=-\infty }^{k} \int_{\mathbb{R}^{n}} g^{(i)}(x) \Phi_\ell (x) x^{\gamma} \mathrm{d} x\right)\left(\frac{\epsilon  _{\gamma}^{k}(x) \chi_{\widetilde{H}^k_{ \varepsilon}}(x)}{\left|\widetilde{H}^k_{ \varepsilon}\right|}-\frac{\epsilon  _{\gamma}^{k+1}(x) \chi_{\widetilde{H}^{k+1}_{ \varepsilon}}(x)}{\left|\widetilde{H}^{k+1}_{ \varepsilon}\right|}\right) \\
= & \sum_{|\gamma| \leq  m} \sum_{k=-\infty}^{\infty} \frac{\alpha_{k, \gamma} h_{k, \gamma}^{(i)}(x)}{\alpha_{k, \gamma}} \\
= & \sum_{|\gamma| \leq  m} \sum_{k=-\infty}^{\infty} \alpha_{k, \gamma} a_{k, \gamma}^{(i)}(x),
\end{align*}

where

\begin{equation*}
\alpha_{k, \gamma}:=\widetilde{b} \sum_{\ell=k-1}^{k+1}\left\|\left(G_{N}g \right) \chi_\ell \right\|_{{p(\cdot)}}\left|B_{k+2}\right|^{\alpha_{k+2} / n} 
\end{equation*}

and the constant $\widetilde{b}$ will be selected at a later time.

Let $x \in B_{k+2}$, then by a computation we have

\[
\begin{array}{r}
\left|\int_{\mathbb{R}^{n}} g^{(i)}(y) \sum_{\ell=-\infty }^{k} \Phi_\ell (y) y^{\gamma} \mathrm{d} y\right| \lesssim 2^{k(n+|\gamma|)} G_{N}g (x) .
\end{array}
\]

Since

\begin{equation*}
\left|\frac{\epsilon  _{\gamma}^{k}(x) \chi_{\widetilde{H}^k_{ \varepsilon}}(x)}{\left|\widetilde{H}^k_{ \varepsilon}\right|}-\frac{\epsilon  _{\gamma}^{k+1}(x) \chi_{\widetilde{H}^{k+1}_{ \varepsilon}}(x)}{\left|\widetilde{H}^{k+1}_{ \varepsilon}\right|}\right| \lesssim 2^{-k(n+|\gamma|)} \sum_{\ell=k-1}^{k+1} \chi_\ell (x) . 
\end{equation*}

It follows that

\begin{equation*}
\left\|h_{k, \gamma}^{(i)}\right\|_{{p(\cdot)}} \leq  C \sum_{\ell=k-1}^{k+1}\left\|\left(G_{N}g \right) \chi_\ell \right\|_{{p(\cdot)}} . 
\end{equation*}

Take $\widetilde{b}=C$. It is easy to show that each $a_{k, \gamma}^{(i)}$ is a central $(\alpha(\cdot)$, $p(\cdot))$-atom with support contained in $\widetilde{H}^k_{ \varepsilon} \cup \widetilde{H}^{k+1}_{ \varepsilon} \subset B_{k+2}$, and

\begin{equation*}
\alpha_{k, \gamma}=C \sum_{\ell=k-1}^{k+1}\left\|\left(G_{N}g \right) \chi_\ell \right\|_{{p(\cdot)}}\left|B_{k+2}\right|^{\frac{\alpha_{k+2}}{n}} 
\end{equation*}

where $C$ is a constant independent of $j, f, k$, and $\gamma$. Consequently, we get

\begin{align*}
& \sup_{\vartheta>0 } \sup _{m_0\in  \mathbb{Z}} 2^{-m_0\Gamma u(1+\vartheta)}\epsilon  ^\theta \sum_{k=-\infty}^{m_0} \sum_{|\gamma|\leq m}\left|\alpha_{k, \gamma}\right|^{u(1+\vartheta)} \\
& \quad C\leq \sup_{\vartheta>0 } \sup _{m_0\in  \mathbb{Z}} 2^{-m_0\Gamma u(1+\vartheta)}\epsilon  ^\theta\sum_{k=-\infty}^{m_0}\left|B_{k+2}\right|^{\frac{u(1+\vartheta) \alpha_{k+2}}  {n}}\left(\sum_{\ell=k-1}^{k+1}\left\|\left(G_{N}g \right) \chi_\ell \right\|_{{p(\cdot)}}\right)^{u(1+\vartheta)} \\
& \quad C\leq  \sup_{\vartheta>0 } \sup _{m_0\in  \mathbb{Z}} 2^{-m_0\Gamma u(1+\vartheta)}\epsilon  ^\theta \sum_{k=-\infty}^{m_0}\left|B_{k+2}\right|^{\frac{u(1+\vartheta) \alpha_{k+2}}  {n}}\left\|\left(G_{N}g \right) \chi_{k}\right\|_{{p(\cdot)}}^{u(1+\vartheta)} \\
& :=B . 
\end{align*}

Using the same argument as before for $A$, we obtain

\begin{equation*}
B \leq C\left\|G_{N}g \right\|_{M \dot{K} ^{\alpha(\cdot),u),\theta}_{ \Gamma, p(\cdot)}(\mathbb{R}^n)}^{u(1+\vartheta)} .
\end{equation*}

Therefore,

\begin{equation*}
\sup_{\vartheta>0 } \sup _{m_0\in  \mathbb{Z}} 2^{-m_0\Gamma u(1+\vartheta)}\epsilon  ^\theta \sum_{k=-\infty}^{m_0} \sum_{|\gamma|\leq m}\left|\alpha_{k, \gamma}\right|^{u(1+\vartheta)} \lesssim\left\|G_{N}g \right\|_{M \dot{K} ^{\alpha(\cdot),u),\theta}_{ \Gamma, p(\cdot)}(\mathbb{R}^n)}^{u(1+\vartheta)} . 
\end{equation*}

Thus, we obtain that

\begin{equation*}
g^{(i)}(x)=\sum_{d=-\infty}^{\infty} \lambda_{d} a_{d}^{(i)}(x),
\end{equation*}

where each $a_{d}^{(i)}$ is a central $(\alpha(\cdot), p(\cdot))$-atom with support contained in $\widetilde{H}^d_{ \varepsilon} \cup \widetilde{H}^{d+1}_{ \varepsilon} \subset B_{d+2}, \lambda_{d}$ is independent of $j$ and

\begin{equation*}
\sup_{\vartheta>0 }\sup _{m_0\in  \mathbb{Z}} 2^{-L \lambda}\left(\sum_{d=-\infty}^{m_0}\left|\lambda_{d}\right|^{q}\right)^{1 / q} \leq  C\left\|G_{N}g \right\|_{{M \dot{K} ^{\alpha(\cdot),u),\theta}_{ \Gamma, p(\cdot)}(\mathbb{R}^n)}}<\infty, 
\end{equation*}

where $C$ is independent of $j$ and $g$.

As we know that

\begin{equation*}
\sup _{i\in  \mathbb{N}_{0}}\left\|a_{0}^{(i)}\right\|_{{p(\cdot)}} \leq \left|B_{2}\right|^{\frac{-\alpha_{2}}{n}} . 
\end{equation*}

By using the Banach-Alaoglu theorem,  we have  a subsequence $\left\{a_{0}^{\left(i_{n_{0}}\right)}\right\}$ of $\left\{a_{0}^{(i)}\right\}$ converging in the weak$^*$ topology of $L^{p(\cdot)}$ to some $a_{0} \in L^{p(\cdot)}$. It is easy to verify that $a_{0}$ is a central $(\alpha(\cdot)$, $p(\cdot))$-atom supported on $B_{2}$. Next, since

\begin{equation*}
\sup _{i_{n_{0}} \in \mathbb{N}_{0}}\left\|a_{1}^{\left(i_{n_{0}}\right)}\right\|_{{p(\cdot)}} \leq \left|B_{3}\right|^{\frac{-\alpha_{3}} {n}} .
\end{equation*}

By the  application of the Banach-Alaoglu theorem yields a subsequence $\left\{a_{1}^{\left(i_{n_{1}}\right)}\right\}$ of $\left\{a_{1}^{\left(i_{n_{0}}\right)}\right\}$ which converges in the weak$^*$ topology in $L^{p(\cdot)}$ to a central $(\alpha(\cdot), p(\cdot))$-atom $a_{1}$ with support in $B_{3}$. Moreover,

\begin{equation*}
\sup _{i_{n_{1}} \in \mathbb{N}_{0}}\left\|a_{-1}^{\left(i_{n_{1}}\right)}\right\|_{{p(\cdot)}} \leq \left|B_{1}\right|^{\frac{-\alpha_{1}} { n}} .
\end{equation*}

Similarly, there exists a subsequence $\left\{a_{-1}^{\left(i_{n_{-1}}\right)}\right\}$ of $\left\{a_{-1}^{\left(i_{n_{1}}\right)}\right\}$ which converges weak$^*$ in $L^{p(\cdot)}$ to some $a_{-1} \in L^{p(\cdot)}$, and $a_{-1}$ is a central $(\alpha(\cdot), p(\cdot))$-atom supported on $B_{1}$. Repeating the above procedure for each $d \in \mathbb{Z}$, we can find a subsequence $\left\{a_{d}^{\left(i_{n_{d}}\right)}\right\}$ of $\left\{a_{d}^{(i)}\right\}$ converging weak$^*$ in $L^{p(\cdot)}$ to some $a_{d} \in L^{p{(\cdot)}}$ which is a central $(\alpha(\cdot), p(\cdot))$-atom supported on $B_{d+2}$. Consider a subsequence $\left\{i_{\nu}\right\}$ of $\mathbb{N}_{0}$ such that, for each $d \in \mathbb{Z}, \lim _{v \rightarrow \infty} a_{d}^{\left(i_{v}\right)}=a_{d}$ in the weak$^*$ topology of $L^{p(\cdot)}$ and therefore in $\mathcal{S}^{\prime}\left(\mathbb{R}^{n}\right)$.

Now we only need to prove that $g=\sum_{d=-\infty}^{\infty} \lambda_{d} a_{d}$ in the sense of $\mathcal{S}^{\prime}\left(\mathbb{R}^{n}\right)$. For each $\varphi \in \mathcal{S}\left(\mathbb{R}^{n}\right)$, note that supp $a_{d}^{\left(i_{v}\right)} \subset$ $\left(\widetilde{H}^d_{ \varepsilon} \cup \widetilde{H}^{d+1}_{ \varepsilon}\right) \subset\left(A_{d-1} \cup A_{d} \cup A_{d+1} \cup A_{d+2}\right)$. Using the same argument in \cite{48}, we have

\begin{equation*}
\langle g, \varphi\rangle=\lim _{v \rightarrow \infty} \sum_{d=-\infty}^{\infty} \lambda_{d} \int_{\mathbb{R}^{n}} a_{d}^{\left(i_{v}\right)}(x) \varphi(x) \mathrm{d} x . 
\end{equation*}

Recall that $m=\max \left\{\left[\alpha(0)-n \delta_{2}\right],\left[\alpha_{\infty}-n \delta_{2}\right]\right\}$. If $d \leq  0$, then by Lemmas 1.1 and 1.3 we have

\begin{align*}
& \left|\int_{\mathbb{R}^{n}} a_{d}^{\left(i_{v}\right)}(x) \varphi(x) \mathrm{d} x\right| \\
& =\left|\int_{\mathbb{R}^{n}} a_{d}^{\left(i_{\nu}\right)}(x)\left(\varphi(x)-\sum_{|\beta| \leq  m} \frac{D^{\beta} \varphi(0)}{\beta !} x^{\beta}\right) \mathrm{d} x\right| \\
& \leq C\int_{\mathbb{R}^{n}}\left|a_{d}^{\left(i_{v}\right)}(x)\right| \cdot|x|^{m+1} \mathrm{~d} x \\
& \leq C 2^{d(m+1)} \int_{\mathbb{R}^{n}}\left|a_{d}^{\left(i_{v}\right)}(x)\right| \mathrm{d} x \\
& \leq C 2^{d\left(m+1-\alpha_{d+2}\right)}\left\|\chi_{B_{d+2}}\right\|_{{p^{\prime}(\cdot)}} \\
& \leq C 2^{d\left(m+1-\alpha_{d+2}\right)}\left(\frac{\left|B_{d+2}\right|}{\left|B_{2}\right|}\right)^{\delta_{2}}\left\|\chi_{B_{2}}\right\|_{{p^{\prime}(\cdot)}} \\
& \leq C 2^{d\left(m+1-\alpha_{d+2}+n \delta_{2}\right)} \frac{\left|B_{2}\right|}{\left|B_{0}\right|}\left\|\chi_{B_{0}}\right\|_{{p^{\prime} (\cdot)}} \\
& \leq C 2^{d\left(m+1-\alpha_{d+2}+n \delta_{2}\right)} .
\end{align*}

If $d>0$, let $k_{0} \in \mathbb{N}_{0}$ such that $\min \left\{k_{0}+\alpha(0)-n, k_{0}+\alpha_{\infty}-\right.$ $n\}>0$; then again using Lemmas 1.1 and 1.3 we have

\begin{align*}
\left|\int_{\mathbb{R}^{n}} a_{d}^{\left(i_{v}\right)}(x) \varphi(x) \mathrm{d} x\right| & \lesssim \int_{\mathbb{R}^{n}}\left|a_{d}^{\left(i_{v}\right)}(x)\right||x|^{-k_{0}} \mathrm{~d} x \\
& \lesssim 2^{-d\left(k_{0}+\alpha_{d+2}\right)}\left\|\chi_{B_{d+2}}\right\|_{{p^{\prime}(\cdot)}} \\
& \lesssim 2^{-d\left(k_{0}+\alpha_{d+2}\right)} \frac{\left|B_{d+2}\right|}{\left|B_{0}\right|}\left\|\chi_{B_{0}}\right\|_{{p^{\prime}(\cdot)}} \\
& \lesssim 2^{-d\left(k_{0}+\alpha_{d+2}-n\right)} .
\end{align*}

Let

\[
\mu_{d}= \begin{cases}\left|\lambda_{d}\right| 2^{d\left(m+1-\alpha_{d+2}+n \delta_{2}\right)}, & d \leq  0  \\ \left|\lambda_{d}\right| 2^{-d\left(k_{0}+\alpha_{d+2}-n\right)}, & d>0 .\end{cases}
\]

Then

\begin{align*}
& \sup _{m_0\in  \mathbb{Z}} 2^{-L \lambda} \sum_{d=-\infty}^{m_0}\left|\mu_{d}\right| \lesssim \sup_{\vartheta>0 }\left(\sup _{m_0\in  \mathbb{Z}} 2^{-L \lambda u(1+\vartheta)}\epsilon  ^\theta \sum_{d=-\infty}^{m_0}\left|\lambda_{d}\right|^{u(1+\vartheta)}\right)^{\frac{1}{u(1+\vartheta)}} \\
& \leq C\left\|G_{N}g \right\|_{{M \dot{K} ^{\alpha(\cdot),u),\theta}_{ \Gamma, p(\cdot)}(\mathbb{R}^n)}}<\infty  \\
&\left|\lambda_{d}\right|\left|\int_{\mathbb{R}^{n}} a_{d}^{\left(i_{v}\right)}(x) \varphi(x) \mathrm{d} x\right| \lesssim\left|\mu_{d}\right|
\end{align*}

which implies that

\begin{align*}
\langle g, \varphi\rangle & =\sum_{d=-\infty}^{\infty} \lim _{\nu \rightarrow \infty} \lambda_{d} \int_{\mathbb{R}^{n}} a_{d}^{\left(i_{\nu}\right)}(x) \varphi(x) \mathrm{d} x  =\sum_{d=-\infty}^{\infty} \lambda_{d} \int_{\mathbb{R}^{n}} a_{d}(x) \varphi(x) \mathrm{d} x . 
\end{align*}

This creates the desired identity.\\
Next we prove that $2 \Rightarrow 1$. 
In order to demonstrate the sufficiency, we denote $$\sup_{\vartheta>0 }\sup \limits _{m_0\in  \mathbb{Z}} 2^{-m_0\Gamma u(1+\vartheta) } \epsilon  ^\theta \sum_{k=-\infty}^{m_0}\left|\lambda_{k}\right|^{u(1+\vartheta)}=\Lambda.$$ Firstly we have

\begin{align*}
& \left\|G_{N}g \right\|_{{M \dot{K} ^{\alpha(\cdot),u),\theta}_{ \Gamma, p(\cdot)}(\mathbb{R}^n)}}^{u(1+\vartheta)} \\
& \approx \max \left\{\sup_{\vartheta>0 }\sup _{m_0<0  , m_0\in  \mathbb{Z}} 2^{-m_0\Gamma u(1+\vartheta)} \epsilon  ^\theta \sum_{k=-\infty}^{m_0} 2^{  \alpha(0) k u(1+\vartheta)}\left\|\left(G_{N}g \right) \chi_{k}\right\|_{{p(\cdot)}}^{u(1+\vartheta)},\right. \\
&  \sup_{\vartheta>0 }\sup _{m_0\geq 0, m_0\in  \mathbb{Z}} 2^{-m_0\Gamma u(1+\vartheta)} \epsilon  ^\theta\left(\sum_{k=-\infty}^{-1} 2^{  \alpha(0) k u(1+\vartheta)}\left\|\left(G_{N}g \right) \chi_{k}\right\|_{{p(\cdot)}}^{u(1+\vartheta)}\right. \\
& \left.\left.+\sum_{k=0}^{m_0} 2^{\alpha_\infty k u(1+\vartheta)}\left\|\left(G_{N}g \right) \chi_{k}\right\|_{{p(\cdot)}}^{u(1+\vartheta)}\right)\right\}\\
&\leq  \max \{I, I I+I I I\}.
\end{align*}

Now we have

\begin{align*}
I= & \sup_{\vartheta>0 }\sup _{m_0<0  , m_0\in  \mathbb{Z}} 2^{-m_0\Gamma u(1+\vartheta)} \epsilon  ^\theta \sum_{k=-\infty}^{m_0} 2^{  \alpha(0) k u(1+\vartheta)}\left\|\left(G_{N}g \right) \chi_{k}\right\|_{{p(\cdot)}}^{u(1+\vartheta)}, \\
\leq & \sup_{\vartheta>0 }\sup _{m_0<0  , m_0\in  \mathbb{Z}} 2^{-m_0\Gamma u(1+\vartheta)} \epsilon  ^\theta \sum_{k=-\infty}^{m_0} 2^{  \alpha(0) k u(1+\vartheta)}\left(\sum_{\ell=k}^{\infty}\left|\lambda_\ell \right|\left\|a_\ell \right\|_{{p(\cdot)}}\right)^{u(1+\vartheta)} \\
 +&\sup_{\vartheta>0 }\sup _{m_0<0  , m_0\in  \mathbb{Z}} 2^{-m_0\Gamma u(1+\vartheta)} \epsilon  ^\theta \sum_{k=-\infty}^{m_0} 2^{  \alpha(0) k u(1+\vartheta)}\left(\sum_{\ell=-\infty }^{k-1}\left|\lambda_\ell \right|\left\|\left(G_{N} a_\ell \right) \chi_{k}\right\|_{{p(\cdot)}}\right)^{u(1+\vartheta)} \\
:= & I_{1}+I_{2} . \\
I I= &  \sup_{\vartheta>0 }\sup _{m_0\geq 0, m_0\in  \mathbb{Z}} 2^{-m_0\Gamma u(1+\vartheta)} \epsilon  ^\theta\sum_{k=-\infty}^{-1} 2^{  \alpha(0) k u(1+\vartheta)}\left\|\left(G_{N}g \right) \chi_{k}\right\|_{{p(\cdot)}}^{u(1+\vartheta)} \\
\leq & \sup_{\vartheta>0 }\sup _{m_0\geq 0, m_0\in  \mathbb{Z}} 2^{-m_0\Gamma u(1+\vartheta)} \epsilon  ^\theta\sum_{k=-\infty}^{-1} 2^{  \alpha(0) k u(1+\vartheta)}\left(\sum_{\ell=k}^{\infty}\left|\lambda_\ell \right|\left\|a_\ell \right\|_{{p(\cdot)}}\right)^{u(1+\vartheta)} \\
 +&\sup_{\vartheta>0 }\sup _{m_0\geq 0, m_0\in  \mathbb{Z}} 2^{-m_0\Gamma u(1+\vartheta)} \epsilon  ^\theta\sum_{k=-\infty}^{-1} 2^{  \alpha(0) k u(1+\vartheta)}\left(\sum_{\ell=-\infty }^{k-1}\left|\lambda_\ell \right|\left\|\left(G_{N} a_\ell \right) \chi_{k}\right\|_{{p(\cdot)}}\right)^{u(1+\vartheta)} \\
:= & I I_{1}+I I_{2} .\\
I I I= & \sup_{\vartheta>0 }\sup _{m_0\geq 0, m_0\in  \mathbb{Z}} 2^{-m_0\Gamma u(1+\vartheta)} \epsilon  ^\theta \sum_{k=0}^{m_0} 2^{k  \alpha_{\infty}u(1+\vartheta)}\left\|\left(G_{N}g \right) \chi_{k}\right\|_{{p(\cdot)}}^{u(1+\vartheta)} \\
\leq  & \sup_{\vartheta>0 }\sup _{m_0\geq 0, m_0\in  \mathbb{Z}} 2^{-m_0\Gamma u(1+\vartheta)} \epsilon  ^\theta \sum_{k=0}^{m_0} 2^{ \alpha_{\infty} k u(1+\vartheta)}\left(\sum_{\ell=k}^{m_0}\left|\lambda_\ell \right|\left\|a_\ell \right\|_{{p(\cdot)}}\right)^{u(1+\vartheta)}\\
+&\sup_{\vartheta>0 }\sup _{m_0\geq 0, m_0\in  \mathbb{Z}} 2^{-m_0\Gamma u(1+\vartheta)} \epsilon  ^\theta\sum_{k=0}^{m_0} 2^{k  \alpha_{\infty}u(1+\vartheta)}\left(\sum_{\ell=-\infty }^{k-1}\left|\lambda_\ell \right|\left\|\left(G_{N} a_\ell \right) \chi_{k}\right\|_{{p(\cdot)}}\right)^{u(1+\vartheta)} \\
:= & I I I_{1}+I I I_{2} .
\end{align*}

Next for $I, I I$, and $I I I$ we need a pointwise estimate for $G_{N} a_\ell (x)$ on $F_{k}$, where $k \geqslant \ell+2$. Let $\phi \in \mathcal{A}_{N}, m \in \mathbb{N}$ such that $\alpha_{k}-n \delta_{2}<m+1$. Denote by $P_{m}$ the $m$-th order Taylor series expansion of $\phi$ at $\frac{y}{t}$. If $|x-y|<t$, then $a_\ell$'s vanishing moment condition gives us

\begin{align*}
& \left|a_\ell  * \phi_{t}(y)\right| \\
& \quad=t^{-n}\left|\int_{\mathbb{R}^{n}} a_\ell (z)\left(\phi\left(\frac{y-z}{t}\right)-P_{m}\left(-\frac{z}{t}\right)\right) \mathrm{d} z\right| \\
& \quad \leq C t^{-n} \int_{\mathbb{R}^{n}}\left|a_\ell (z)\right|\left|\frac{z}{t}\right|^{m+1}\left(1+\left|\frac{(y-\theta z)}{t}\right|\right)^{-(n+m+1)} \mathrm{d} z \\
& \quad \leq C \int_{\mathbb{R}^{n}}\left|a_\ell (z)\right||z|^{m+1}(t+|y-\theta z|)^{-(n+m+1)} \mathrm{d} z,
\end{align*}

where $0<\theta<1$. Since $x \in D_{k}$ for $k \in \mathbb{Z}$, we have $|x| \geqslant 2^{k-1}$. From $|x-y|<t$ and $|z|<2^{\ell}$, we have

\begin{equation*}
t+|y-\theta z| \geqslant|x-y|+|y-\theta z| \geqslant|x|-|z| \geqslant \frac{|x|}{2}
\end{equation*}

Thus,

\begin{align*}
& \left|a_\ell  * \phi_{t}(y)\right| \\
& \quad \lesssim \int_{\mathbb{R}^{n}}\left|a_\ell (z)\right||z|^{m+1}(|x-y|+|y-\theta z|)^{-(n+m+1)} \mathrm{d} z \\
& \quad \lesssim 2^{\ell(m+1)}|x|^{-(n+m+1)} \int_{\mathbb{R}^{n}}|a(z)| \mathrm{d} z  \\
& \quad \lesssim 2^{\ell(m+1)} 2^{-k(n+m+1)}\left|B_\ell \right|^{-\alpha_\ell  / n}\left\|\chi_{B_\ell }\right\|_{{p^{\prime}(\cdot)}} .
\end{align*}

Therefore for $x \in F_{k}$ and $k\geq \ell+2$, we have

\[
\begin{array}{r}
G_{N} a_\ell (x) \lesssim 2^{\ell(m+1)} 2^{-k(n+m+1)}\left|B_\ell \right|^{\frac{-\alpha_\ell  } {n}}\left\|\chi_{B_\ell }\right\|_{{p^{\prime}(\cdot)}}.
\end{array}
\]

To continue, we divide them into two cased when $0<u(1+\vartheta) \leq  1$ and $1<u(1+\vartheta)<\infty$.

If $0<u(1+\vartheta) \leq  1$.
First we will find the estimate for the term $I_1$. 
\begin{align*}
I_{1}&=  \sup_{\vartheta>0 }\sup _{m_0<0  , m_0\in  \mathbb{Z}} 2^{-m_0\Gamma u(1+\vartheta)} \epsilon  ^\theta\sum_{k=-\infty}^{-1} 2^{  \alpha(0) k u(1+\vartheta)}\left(\sum_{\ell=k}^{\infty}\left|\lambda_\ell \right|\left\|a_\ell \right\|_{{p(\cdot)}}\right)^{u(1+\vartheta)}\\
& \lesssim \sup_{\vartheta>0 }\sup _{m_0<0  , m_0\in  \mathbb{Z}} 2^{-m_0\Gamma u(1+\vartheta)} \epsilon  ^\theta \sum_{k=-\infty}^{m_0} 2^{\alpha(0)k u(1+\vartheta)}\left(\sum_{\ell=k}^{\infty}\left|\lambda_\ell \right| 2^{-\alpha_\ell  \ell }\right)^{u(1+\vartheta)} \\
&\lesssim  \sup_{\vartheta>0 }\sup _{m_0\leq 0, m_0\in  \mathbb{Z}} 2^{-m_0\Gamma u(1+\vartheta)} \epsilon  ^\theta \sum_{k=-\infty}^{m_0} 2^{\alpha(0)k u(1+\vartheta)}  \\
& \times\left(\sum_{\ell=k}^{-1}\left|\lambda_\ell \right|^{u(1+\vartheta)} 2^{-\alpha(0) \ell u(1+\vartheta)}+\sum_{\ell=0}^{\infty}\left|\lambda_\ell \right|^{u(1+\vartheta)} 2^{-\alpha_{\infty} \ell u(1+\vartheta)}\right)\\
& \lesssim \sup_{\vartheta>0 }\sup _{m_0<0  , m_0\in  \mathbb{Z}} 2^{-m_0\Gamma u(1+\vartheta)} \epsilon  ^\theta \sum_{k=-\infty}^{m_0} \sum_{\ell=k}^{-1}\left|\lambda_\ell \right|^{u(1+\vartheta)} 2^{\alpha(0)(k-\ell) u(1+\vartheta)} \\
& + \sup_{\vartheta>0 }\sup _{m_0<0  , m_0\in  \mathbb{Z}} 2^{-m_0\Gamma u(1+\vartheta)} \epsilon  ^\theta \sum_{k=-\infty}^{m_0} 2^{\alpha(0) k  u(1+\vartheta)} \sum_{\ell=0}^{\infty}\left|\lambda_\ell \right|^{u(1+\vartheta)} 2^{-\alpha_{\infty} \ell u(1+\vartheta)} \\
&\lesssim \sup_{\vartheta>0 }\sup _{m_0<0  , m_0\in  \mathbb{Z}} 2^{-m_0\Gamma u(1+\vartheta)} \epsilon  ^\theta \sum_{\ell=-\infty }^{-1}\left|\lambda_\ell \right|^{u(1+\vartheta)} \sum_{k=-\infty}^{m_0} 2^{\alpha(0)(k-\ell) u(1+\vartheta)} \\
& +\sup_{\vartheta>0 }\sup _{m_0<0  , m_0\in  \mathbb{Z}}  \epsilon  ^\theta \sum_{\ell=0}^{\infty} 2^{-m_0\Gamma u(1+\vartheta)}\left|\lambda_\ell \right|^{u(1+\vartheta)} 2^{\left(\Gamma-\alpha_{\infty}\right) \ell u(1+\vartheta)} 2^{-m_0\Gamma u(1+\vartheta)} \sum_{k=-\infty}^{m_0} 2^{\alpha(0) k u(1+\vartheta)} \\
& \lesssim \left( \Lambda+  \sup_{\vartheta>0 }\sup _{m_0<0  , m_0\in  \mathbb{Z}} 2^{-m_0\Gamma u(1+\vartheta)} \epsilon  ^\theta \sum_{\ell=m_0}^{-1}\left|\lambda_\ell \right|^{u(1+\vartheta)} \sum_{k=-\infty}^{m_0} 2^{\alpha(0)(k-\ell) u(1+\vartheta)}\right.\\
&\left. +\Lambda \sup_{\vartheta>0 }\sup _{m_0<0  , m_0\in  \mathbb{Z}}  \epsilon  ^\theta  \sum_{\ell=0}^{\infty} 2^{\left(\Gamma-\alpha_{\infty}\right) \ell u(1+\vartheta)} \sum_{k=-\infty}^{m_0} 2^{(\alpha(0) k -m_0\Gamma)u(1+\vartheta) }\right) \\
&\lesssim\left( \Lambda+\sup_{\vartheta>0 }\sup _{m_0<0  , m_0\in  \mathbb{Z}}  \epsilon  ^\theta \sum_{\ell=m_0}^{-1} 2^{-m_0\Gamma u(1+\vartheta)}\left|\lambda_\ell \right|^{u(1+\vartheta)} 2^{(\ell-m_0) \Gamma u(1+\vartheta)} \sum_{k=-\infty}^{m_0} 2^{\alpha(0)(k-\ell) u(1+\vartheta)}+\Lambda \right)\\
&\lesssim\left(\Lambda+\Lambda \sup_{\vartheta>0 }\sup _{m_0<0  , m_0\in  \mathbb{Z}}  \epsilon  ^\theta  \sum_{\ell=m_0}^{-1} 2^{(\ell-m_0) \Gamma u(1+\vartheta)} \sum_{k=-\infty}^{m_0} 2^{\alpha(0)(k-\ell) u(1+\vartheta)}\right)\\
& \lesssim \Lambda.
\end{align*}

Now we will find the estimate for $I_2$. 

\begin{align*}
I_{2}&=  \sup_{\vartheta>0 }\sup _{m_0<0  , m_0\in  \mathbb{Z}} 2^{-m_0\Gamma u(1+\vartheta)} \epsilon  ^\theta \sum_{k=-\infty}^{m_0} 2^{  \alpha(0) k u(1+\vartheta)}\left(\sum_{\ell=-\infty }^{k-1}\left|\lambda_\ell \right|\left\|\left(G_{N} a_\ell \right) \chi_{k}\right\|_{{p(\cdot)}}\right)^{u(1+\vartheta)} \\
&\lesssim  \sup_{\vartheta>0 }\sup _{m_0<0  , m_0\in  \mathbb{Z}} 2^{-m_0\Gamma u(1+\vartheta)} \epsilon  ^\theta \sum_{k=-\infty}^{m_0} 2^{  \alpha(0) k u(1+\vartheta)}  \left(\sum_{\ell=-\infty }^{k-1}\left|\lambda_\ell \right| 2^{\ell(m+1)-k(n+m+1)}\left|B_\ell \right|^{\frac{-\alpha_\ell } {n}}\left\|\chi_{B_\ell }\right\|_{{p^{\prime}(\cdot)}}\left\|\chi_{B_{k}}\right\|_{{p(\cdot)}}\right)^{u(1+\vartheta)} \\
& \lesssim \sup_{\vartheta>0 }\sup _{m_0<0  , m_0\in  \mathbb{Z}} 2^{-m_0\Gamma u(1+\vartheta)} \epsilon  ^\theta \sum_{k=-\infty}^{m_0}\left(\sum_{\ell=-\infty }^{k-1}\left|\lambda_\ell \right| 2^{(\ell-k)\left(m+1+n \delta_{2}\right)-(\ell-k) \alpha(0)}\right)^{u(1+\vartheta)} \\
& \lesssim \sup_{\vartheta>0 }\sup _{m_0<0  , m_0\in  \mathbb{Z}} 2^{-m_0\Gamma u(1+\vartheta)} \epsilon  ^\theta \sum_{k=-\infty}^{m_0}\left(\sum_{\ell=-\infty }^{k-1}\left|\lambda_\ell \right| 2^{(\ell-k)\left(m+1+n \delta_{2}-\alpha(0)\right)}\right)^{u(1+\vartheta)} \\
& = \sup_{\vartheta>0 }\sup _{m_0<0  , m_0\in  \mathbb{Z}} 2^{-m_0\Gamma u(1+\vartheta)} \epsilon  ^\theta \sum_{\ell=-\infty }^{m_0-1} \sum_{k=\ell+1}^{\ell }\left|\lambda_\ell \right|^{u(1+\vartheta)} 2^{(\ell-k)\left(m+1+n \delta_{2}-\alpha(0)\right) u(1+\vartheta)} \\
& \lesssim \sup_{\vartheta>0 }\sup _{m_0<0  , m_0\in  \mathbb{Z}} 2^{-m_0\Gamma u(1+\vartheta)} \epsilon  ^\theta \sum_{\ell=-\infty }^{\ell }\left|\lambda_\ell \right|^{u(1+\vartheta)}\\
&\lesssim \Lambda .
\end{align*}

Then we turn to estimate $II_1$.

\begin{align*}
I I_{1}= & \sup_{\vartheta>0 }\sup _{m_0\geq 0, m_0\in  \mathbb{Z}} 2^{-m_0\Gamma u(1+\vartheta)} \epsilon  ^\theta\sum_{k=-\infty}^{-1} 2^{  \alpha(0) k u(1+\vartheta)}\left(\sum_{\ell=k}^{\infty}\left|\lambda_\ell \right|\left\|a_\ell \right\|_{{p(\cdot)}}\right)^{u(1+\vartheta)} \\
\lesssim &  \sup_{\vartheta>0 }\sup _{m_0\geq 0, m_0\in  \mathbb{Z}} 2^{-m_0\Gamma u(1+\vartheta)} \epsilon  ^\theta\sum_{k=-\infty}^{-1} 2^{  \alpha(0) k u(1+\vartheta)}\left(\sum_{\ell=k}^{\infty}\left|\lambda_\ell \right||B_l|^{\frac{-\alpha_l}{n}}\right)^{u(1+\vartheta)} \\
\lesssim &  \sup_{\vartheta>0 }\sup _{m_0\geq 0, m_0\in  \mathbb{Z}} 2^{-m_0\Gamma u(1+\vartheta)} \epsilon  ^\theta \sum_{k=-\infty}^{-1} 2^{\alpha(0) k u(1+\vartheta)}\left(\sum_{\ell=k}^{-1}\left|\lambda_\ell \right|^{u(1+\vartheta)} 2^{-\alpha(0) \ell u(1+\vartheta)}+\sum_{\ell=0}^{\infty}\left|\lambda_\ell \right|^{u(1+\vartheta)} 2^{-\alpha_{\infty} \ell u(1+\vartheta)}\right) \\
\lesssim  &   \sup_{\vartheta>0 }\sup _{m_0\geq 0, m_0\in  \mathbb{Z}}  \epsilon  ^\theta\sum_{k=-\infty}^{-1} \sum_{\ell=k}^{-1}\left|\lambda_\ell \right|^{u(1+\vartheta)} 2^{(\alpha(0)(k-\ell) -m_0\Gamma ) u(1+\vartheta)} \\
+&   \sup_{\vartheta>0 }\sup _{m_0\geq 0, m_0\in  \mathbb{Z}}  \epsilon  ^\theta\sum_{k=-\infty}^{-1} 2^{(\alpha(0) k-m_0\Gamma ) u(1+\vartheta)} \sum_{\ell=0}^{\infty}\left|\lambda_\ell \right|^{u(1+\vartheta)} 2^{-\alpha_{\infty} \ell u(1+\vartheta)} \\
\lesssim  &  \sup_{\vartheta>0 }\sup _{m_0\geq 0, m_0\in  \mathbb{Z}}   \epsilon  ^\theta 2^{-m_0\Gamma u(1+\vartheta)} \sum_{\ell=-\infty }^{-1}\left|\lambda_\ell \right|^{u(1+\vartheta)} \sum_{k=-\infty}^{m_0} 2^{(\alpha(0)(k-\ell) ) u(1+\vartheta)}  \\
\lesssim&   \sup_{\vartheta>0 }\sup _{m_0\geq 0, m_0\in  \mathbb{Z}}  \epsilon  ^\theta 2^{-m_0\Gamma u(1+\vartheta)} \sum_{\ell=0}^{\infty}\left|\lambda_\ell \right|^{u(1+\vartheta)} 2^{-\alpha_{\infty} \ell u(1+\vartheta)} \sum_{k=-\infty}^{-1} 2^{(\alpha(0) k) u(1+\vartheta)} \\
\lesssim &  \Lambda + C \Lambda \sup_{\vartheta>0 }\sup _{m_0\geq 0, m_0\in  \mathbb{Z}}  \epsilon  ^\theta 2^{-m_0\Gamma u(1+\vartheta)}\sum_{i=-\infty}^l \left|\lambda_{i}\right|^{u(1+\vartheta)} \sum_{\ell=0}^{\infty} 2^{-\alpha_{\infty} \ell u(1+\vartheta)} \sum_{k=-\infty}^{-1} 2^{(\alpha(0) k) u(1+\vartheta)} \\
\lesssim & \Lambda. 
\end{align*}

For the second term $II_2$, we have

\begin{align*}
I I_{2} & =\sup_{\vartheta>0 }\sup _{m_0\geq 0, m_0\in  \mathbb{Z}} 2^{-m_0\Gamma u(1+\vartheta)} \epsilon  ^\theta\sum_{k=-\infty}^{-1} 2^{  \alpha(0) k u(1+\vartheta)}\left(\sum_{\ell=-\infty }^{k-1}\left|\lambda_\ell \right|\left\|\left(G_{N} a_\ell \right) \chi_{k}\right\|_{{p(\cdot)}}\right)^{u(1+\vartheta)}  \\
& \lesssim \sup_{\vartheta>0 }\sup _{m_0\geq 0, m_0\in  \mathbb{Z}} 2^{-m_0\Gamma u(1+\vartheta)} \epsilon  ^\theta \sum_{k=-\infty}^{-1} 2^{k  \alpha(0) u(1+\vartheta)}\left(\sum_{\ell=-\infty }^{k-1}\left|\lambda_\ell \right| 2^{\ell(m+1)-k(n+m+1)}\left|B_\ell \right|^{-\alpha_\ell  / n}\left\|\chi_{B_\ell }\right\|_{{p^{\prime}(\cdot)}}\left\|\chi_{B_{k}}\right\|_{{p(\cdot)}}\right)^{u(1+\vartheta)} \\
& \lesssim \sup_{\vartheta>0 }\sup _{m_0\geq 0, m_0\in  \mathbb{Z}} 2^{-m_0\Gamma u(1+\vartheta)} \epsilon  ^\theta  \sum_{k=-\infty}^{-1}\left(\sum_{\ell=-\infty }^{k-1}\left|\lambda_\ell \right| 2^{(\ell-k)\left(m+1+n \delta_{2}\right)-(\ell-k) \alpha(0)}\right)^{u(1+\vartheta)} \\
& \lesssim \sup_{\vartheta>0 }\sup _{m_0\geq 0, m_0\in  \mathbb{Z}} 2^{-m_0\Gamma u(1+\vartheta)} \epsilon  ^\theta \sum_{k=-\infty}^{-1}\left(\sum_{\ell=-\infty }^{k-1}\left|\lambda_\ell \right| 2^{(\ell-k)\left(m+1+n \delta_{2}-\alpha(0)\right)}\right)^{u(1+\vartheta)} \\
& \lesssim \sup_{\vartheta>0 }\sup _{m_0\geq 0, m_0\in  \mathbb{Z}} 2^{-m_0\Gamma u(1+\vartheta)} \epsilon  ^\theta \sum_{k=-\infty}^{-1}\left(\sum_{\ell=-\infty }^{k-1}\left|\lambda_\ell \right|^{u(1+\vartheta)} 2^{(\ell-k)\left(m+1+n \delta_{2}-\alpha(0)\right) u(1+\vartheta)}\right) \\
& = \sup_{\vartheta>0 }\sup _{m_0\geq 0, m_0\in  \mathbb{Z}} 2^{-m_0\Gamma u(1+\vartheta)} \epsilon  ^\theta\sum_{\ell=-\infty }^{-2} \sum_{k=\ell+1}^{-1}\left|\lambda_\ell \right|^{u(1+\vartheta)} 2^{(\ell-k)\left(m+1+n \delta_{2}-\alpha(0)\right) u(1+\vartheta)} \\
& \lesssim \sup_{\vartheta>0 }\sup _{m_0\geq 0, m_0\in  \mathbb{Z}} 2^{-m_0\Gamma u(1+\vartheta)} \epsilon  ^\theta \sum_{\ell=-\infty }^{-1}\left|\lambda_\ell \right|^{u(1+\vartheta)}\\
&  \lesssim \Lambda .
\end{align*}

Thirdly, we will find the  estimate for $III$:

\begin{align*}
I I I_{1}&=\sup_{\vartheta>0 }\sup _{m_0\geq 0, m_0\in  \mathbb{Z}} 2^{-m_0\Gamma u(1+\vartheta)} \epsilon  ^\theta \sum_{k=0}^{m_0} 2^{ \alpha_{\infty} k u(1+\vartheta)}\left(\sum_{\ell=k}^{m_0}\left|\lambda_\ell \right|\left\|a_\ell \right\|_{{p(\cdot)}}\right)^{u(1+\vartheta)} \\
& \lesssim\sup_{\vartheta>0 }\sup _{m_0\geq 0, m_0\in  \mathbb{Z}} 2^{-m_0\Gamma u(1+\vartheta)} \epsilon  ^\theta \sum_{k=0}^{m_0} 2^{k  \alpha_{\infty}u(1+\vartheta)}\left|B_{k}\right|^{\alpha_{\infty} u(1+\vartheta) / n}\left(\sum_{\ell=k}^{\infty}\left|\lambda_\ell \right|\left|B_\ell \right|^{-\alpha_\ell  / n}\right)^{u(1+\vartheta)} \\
& \lesssim \sup_{\vartheta>0 }\sup _{m_0\geq 0, m_0\in  \mathbb{Z}} 2^{-m_0\Gamma u(1+\vartheta)} \epsilon  ^\theta  \sum_{k=0}^{m_0} \sum_{\ell=k}^{\infty}\left|B_{k}\right|^{\frac{\alpha_{\infty} u(1+\vartheta)} {n}}\left|\lambda_\ell \right|^{u(1+\vartheta)}\left|B_\ell \right|^{\frac{-\alpha_{\infty}u(1+\vartheta) }{ n}} \\
& \lesssim \sup_{\vartheta>0 }\sup _{m_0\geq 0, m_0\in  \mathbb{Z}} 2^{-m_0\Gamma u(1+\vartheta)} \epsilon  ^\theta \sum_{k=0}^{m_0} \sum_{\ell=k}^{\infty}\left|\lambda_\ell \right|^{u(1+\vartheta)} 2^{\alpha_{\infty}(k-\ell)  u(1+\vartheta)} \\
& =\sup_{\vartheta>0 }\sup _{m_0\geq 0, m_0\in  \mathbb{Z}} 2^{-m_0\Gamma u(1+\vartheta)} \epsilon  ^\theta \\
& \times\left[\sum_{\ell=0}^{\ell }\left|\lambda_\ell \right|^{u(1+\vartheta)} \sum_{k=0}^{m_0} 2^{(k-\ell) \alpha_{\infty} u(1+\vartheta)}+\sum_{\ell=m_0}^{\infty}\left|\lambda_\ell \right|^{u(1+\vartheta)} \sum_{k=0}^{m_0} 2^{(k-\ell) \alpha_{\infty} u(1+\vartheta)}\right] \\
& \lesssim\sup_{\vartheta>0 }\sup _{m_0\geq 0, m_0\in  \mathbb{Z}} 2^{-m_0\Gamma u(1+\vartheta)} \epsilon  ^\theta \sum_{\ell=0}^{\ell }\left|\lambda_\ell \right|^{u(1+\vartheta)} \\
& +\sup _{m_0>0, m_0\in  \mathbb{Z} } \sum_{\ell=m_0}^{\infty} 2^{\Gamma(\ell-m_0)u(1+\vartheta)} 2^{-m_0\Gamma u(1+\vartheta)} \sum_{i=-\infty}^{\ell }\left|\lambda_{i}\right|^{u(1+\vartheta)} \sum_{k=0}^{m_0} 2^{\frac{(k-\ell) \alpha_{\infty} u(1+\vartheta)} {2}} \\
& \lesssim \Lambda+\Lambda \sup _{m_0>0, m_0\in  \mathbb{Z}} \sum_{\ell=m_0}^{\infty} 2^{\Gamma(\ell-m_0)u(1+\vartheta)} 2^{\frac{(\ell-m_0) \alpha_{\infty} u(1+\vartheta)} {2}} \\
& \lesssim \Lambda+\Lambda \sup _{m_0>0, m_0\in  \mathbb{Z}} \sum_{\ell=m_0}^{\infty} 2^{\Gamma(\ell-m_0)u(1+\vartheta)} 2^{\frac{(\ell-m_0) \alpha_{\infty} u(1+\vartheta)} {2}} \\
& \lesssim  \Lambda . 
\end{align*}
\begin{align*}
     I I I_{2}&=\sup_{\vartheta>0 }\sup _{m_0\geq 0, m_0\in  \mathbb{Z}} 2^{-m_0\Gamma u(1+\vartheta)} \epsilon  ^\theta\sum_{k=0}^{m_0} 2^{k  \alpha_{\infty}u(1+\vartheta)}\left(\sum_{\ell=-\infty }^{k-1}\left|\lambda_\ell \right|\left\|\left(G_{N} a_\ell \right) \chi_{k}\right\|_{{p(\cdot)}}\right)^{u(1+\vartheta)} \\
&  \lesssim \sup_{\vartheta>0 }\sup _{m_0\geq 0, m_0\in  \mathbb{Z}} 2^{-m_0\Gamma u(1+\vartheta)} \epsilon  ^\theta \\
& \times \sum_{k=0}^{m_0} 2^{k  \alpha_{\infty}u(1+\vartheta)}\left(\sum_{\ell=-\infty }^{k-1}\left|\lambda_\ell \right| 2^{\ell(m+1)-k(n+m+1)}\left|B_\ell \right|^{\frac{-\alpha_\ell } {n}}\left\|\chi_{B_\ell }\right\|_{{p^{\prime}(\cdot)}}\left\|\chi_{B_{k}}\right\|_{{p(\cdot)}}\right)^{u(1+\vartheta)} \\
&  \lesssim \sup_{\vartheta>0 }\sup _{m_0\geq 0, m_0\in  \mathbb{Z}} 2^{-m_0\Gamma u(1+\vartheta)} \epsilon  ^\theta \sum_{k=0}^{m_0}\left(\sum_{\ell=-\infty }^{k-1}\left|\lambda_\ell \right| 2^{(\ell-k)\left(m+1+n \delta_{2}\right)-\ell \alpha_\ell + \alpha_{\infty} k }\right)^{u(1+\vartheta)} \\
&  \lesssim \sup_{\vartheta>0 }\sup _{m_0\geq 0, m_0\in  \mathbb{Z}} 2^{-m_0\Gamma u(1+\vartheta)} \epsilon  ^\theta \sum_{k=0}^{m_0}\left(\sum_{\ell=-\infty }^{-1}\left|\lambda_\ell \right| 2^{(\ell-k)\left(m+1+n \delta_{2}\right)-\ell \alpha(0)+ \alpha_{\infty} k }\right)^{u(1+\vartheta)} \\
&  + \sup_{\vartheta>0 }\sup _{m_0\geq 0, m_0\in  \mathbb{Z}} 2^{-m_0\Gamma u(1+\vartheta)} \epsilon  ^\theta \sum_{k=0}^{m_0}\left(\sum_{\ell=0}^{k-1}\left|\lambda_\ell \right| 2^{(\ell-k)\left(m+1+n \delta_{2}-\alpha_\infty\right) }\right)^{u(1+\vartheta)} \\
&\lesssim  \sup_{\vartheta>0 }\sup _{m_0\geq 0, m_0\in  \mathbb{Z}} 2^{-m_0\Gamma u(1+\vartheta)} \epsilon  ^\theta  \sum_{k=0}^{m_0} 2^{u(1+\vartheta) k\left[\alpha_{\infty}-\left(m+1+n \delta_{2}\right)\right]} \left(\sum_{\ell=-\infty }^{-1}\left|\lambda_\ell \right| 2^{\ell \left(m+1+n \delta_{2}-\alpha(0)\right)}\right)^{u(1+\vartheta)} \\
& + \sup_{\vartheta>0 }\sup _{m_0\geq 0, m_0\in  \mathbb{Z}} 2^{-m_0\Gamma u(1+\vartheta)} \epsilon  ^\theta \sum_{k=0}^{m_0} \sum_{\ell=0}^{k-1}\left|\lambda_\ell \right|^{u(1+\vartheta)} 2^{(\ell-k)\left(m+1+n \delta_{2}-\alpha_{\infty}\right) u(1+\vartheta)} \\
&\lesssim  \sup_{\vartheta>0 }\sup _{m_0\geq 0, m_0\in  \mathbb{Z}} 2^{-m_0\Gamma u(1+\vartheta)} \epsilon  ^\theta \sum_{\ell=-\infty }^{-1}\left|\lambda_\ell \right|^{u(1+\vartheta)} 2^{\ell \left(m+1+n \delta_{2}-\alpha(0)\right) u(1+\vartheta)} \\
& +\sup_{\vartheta>0 }\sup _{m_0\geq 0, m_0\in  \mathbb{Z}} 2^{-m_0\Gamma u(1+\vartheta)} \epsilon  ^\theta\sum_{\ell=0}^{m_0-1} \sum_{k=\ell+1}^{\ell }\left|\lambda_\ell \right|^{u(1+\vartheta)} 2^{(\ell-k)\left(m+1+n \delta_{2}-\alpha_{\infty}\right) u(1+\vartheta)} \\
&\lesssim   \sup_{\vartheta>0 }\sup _{m_0\geq 0, m_0\in  \mathbb{Z}} 2^{-m_0\Gamma u(1+\vartheta)} \epsilon  ^\theta \sum_{\ell=-\infty }^{-1}\left|\lambda_\ell \right|^{u(1+\vartheta)}+\sup_{\vartheta>0 }\sup _{m_0\geq 0, m_0\in  \mathbb{Z}} 2^{-m_0\Gamma u(1+\vartheta)} \epsilon  ^\theta\sum_{\ell=0}^{\ell }\left|\lambda_\ell \right|^{u(1+\vartheta)} \\
& \lesssim \Lambda .
\end{align*}

Now for the second case if $1<u(1+\vartheta)<\infty$. To find the estimate of $I_1$, we have
\begin{align*}
     I_{1}&= \sup_{\vartheta>0 }\sup _{m_0<0  , m_0\in  \mathbb{Z}} 2^{-m_0\Gamma u(1+\vartheta)} \epsilon  ^\theta\sum_{k=-\infty}^{-1} 2^{  \alpha(0) k u(1+\vartheta)}\left(\sum_{\ell=k}^{\infty}\left|\lambda_\ell \right|\left\|a_\ell \right\|_{{p(\cdot)}}\right)^{u(1+\vartheta)}\\
& \lesssim \sup_{\vartheta>0 }\sup _{m_0<0  , m_0\in  \mathbb{Z}} 2^{-m_0\Gamma u(1+\vartheta)} \epsilon  ^\theta \sum_{k=-\infty}^{m_0} 2^{  \alpha(0) k u(1+\vartheta)}\left(\sum_{\ell=k}^{\infty}\left|\lambda_\ell \right|\left|B_\ell \right|^{-\alpha_\ell  / n}\right)^{u(1+\vartheta)} \\
& \lesssim\sup_{\vartheta>0 }\sup _{m_0<0  , m_0\in  \mathbb{Z}} 2^{-m_0\Gamma u(1+\vartheta)} \epsilon  ^\theta \sum_{k=-\infty}^{m_0}\left(\sum_{\ell=k}^{-1}\left|\lambda_\ell \right| 2^{\alpha(0)(k-\ell)}\right)^{u(1+\vartheta)} \\
& +\sup_{\vartheta>0 }\sup _{m_0<0  , m_0\in  \mathbb{Z}} 2^{-m_0\Gamma u(1+\vartheta)} \epsilon  ^\theta \sum_{k=-\infty}^{m_0} 2^{\alpha(0) k u(1+\vartheta)}\left(\sum_{\ell=0}^{\infty}\left|\lambda_\ell \right| 2^{-\alpha_{\infty} \ell }\right)^{u(1+\vartheta)} \\
& \lesssim\sup_{\vartheta>0 }\sup _{m_0<0  , m_0\in  \mathbb{Z}} 2^{-m_0\Gamma u(1+\vartheta)} \epsilon  ^\theta \sum_{k=-\infty}^{m_0}\left(\sum_{\ell=k}^{-1}\left|\lambda_\ell \right|^{u(1+\vartheta)} 2^{\alpha(0)(k-\ell) u(1+\vartheta) / 2}\right) 
\left(\sum_{\ell=k}^{-1} 2^{\alpha(0)(k-\ell) (u(1+\vartheta))^{\prime} / 2}\right)^{u(1+\vartheta) / (u(1+\vartheta))^{\prime}} \\
& +\sup_{\vartheta>0 }\sup _{m_0<0  , m_0\in  \mathbb{Z}} 2^{-m_0\Gamma u(1+\vartheta)} \epsilon  ^\theta \sum_{k=-\infty}^{m_0} 2^{\alpha(0) k u(1+\vartheta) }\left(\sum_{\ell=0}^{\infty}\left|\lambda_\ell \right|^{u(1+\vartheta)} 2^{-\alpha_{\infty} \ell u(1+\vartheta) / 2}\right) 
\left(\sum_{\ell=0}^{\infty} 2^{-\alpha_{\infty} \ell  (u(1+\vartheta))^{\prime} / 2}\right)^{u(1+\vartheta) / (u(1+\vartheta))^{\prime}} \\
& \lesssim\sup_{\vartheta>0 }\sup _{m_0<0  , m_0\in  \mathbb{Z}} 2^{-m_0\Gamma u(1+\vartheta)} \epsilon  ^\theta \sum_{k=-\infty}^{m_0}\sum_{\ell=k}^{-1}\left|\lambda_\ell \right|^{u(1+\vartheta)} 2^{\alpha(0)(k-\ell) u(1+\vartheta) / 2} \\
& + \sup_{\vartheta>0 }\sup _{m_0<0  , m_0\in  \mathbb{Z}} 2^{-m_0\Gamma u(1+\vartheta)} \epsilon  ^\theta \sum_{k=-\infty}^{m_0} 2^{\alpha(0) k u(1+\vartheta) }\sum_{\ell=0}^{\infty}\left|\lambda_\ell \right|^{u(1+\vartheta)} 2^{-\alpha_{\infty} \ell u(1+\vartheta) / 2} \\
& \lesssim\sup_{\vartheta>0 }\sup _{m_0<0  , m_0\in  \mathbb{Z}} 2^{-m_0\Gamma u(1+\vartheta)} \epsilon  ^\theta \sum_{\ell=-\infty }^{-1}\left|\lambda_\ell \right|^{u(1+\vartheta)} \sum_{k=-\infty}^{m_0} 2^{\alpha(0)(k-\ell) u(1+\vartheta) / 2} \\
& +\sup_{\vartheta>0 }\sup _{m_0<0  , m_0\in  \mathbb{Z}}  \epsilon  ^\theta \sum_{\ell=0}^{\infty}2^{-m_0\Gamma u(1+\vartheta)}\left|\lambda_\ell \right|^{u(1+\vartheta)} 2^{\left(\Gamma-\alpha_{\infty} / 2\right) \ell u(1+\vartheta)} 2^{-m_0\Gamma u(1+\vartheta)} \sum_{k=-\infty}^{m_0} 2^{\alpha(0) k u(1+\vartheta)} \\
& \lesssim\sup_{\vartheta>0 }\sup _{m_0<0  , m_0\in  \mathbb{Z}} 2^{-m_0\Gamma u(1+\vartheta)} \epsilon  ^\theta \sum_{\ell=-\infty }^{-1}\left|\lambda_\ell \right|^{u(1+\vartheta)}
 +\sup_{\vartheta>0 }\sup _{m_0<0  , m_0\in  \mathbb{Z}} 2^{-m_0\Gamma u(1+\vartheta)} \epsilon  ^\theta \sum_{\ell=m_0}^{-1}\left|\lambda_\ell \right|^{u(1+\vartheta)} \sum_{k=-\infty}^{m_0} 2^{\alpha(0)(k-\ell) u(1+\vartheta) / 2} \\
& +\Lambda \sup_{\vartheta>0 } \sup _{m_0<0  , m_0\in  \mathbb{Z}}\epsilon  ^\theta \sum_{\ell=0}^{\infty} 2^{\left(\lambda-\alpha_{\infty} / 2\right) l  u(1+\vartheta)} \sum_{k=-\infty}^{m_0} 2^{(\alpha(0) k -m_0\Gamma)  u(1+\vartheta)} \\
& \lesssim \Lambda+ \sup_{\vartheta>0 } \sup _{m_0<0  , m_0\in  \mathbb{Z}}\epsilon  ^\theta \sum_{\ell=m_0}^{-1} 2^{-l \lambda  u(1+\vartheta)}\left|\lambda_\ell \right|^{ u(1+\vartheta)} 2^{(\ell-m_0) \lambda  u(1+\vartheta)}   \sum_{k=-\infty}^{m_0} 2^{\alpha(0)(k-\ell)  u(1+\vartheta) / 2}+\Lambda \\
& \lesssim \Lambda+ \Lambda \sup_{\vartheta>0 } \sup _{m_0<0  , m_0\in  \mathbb{Z}}\epsilon  ^\theta \sum_{\ell=m_0}^{-1}  2^{(\ell-m_0) \lambda  u(1+\vartheta)}   \sum_{k=-\infty}^{m_0} 2^{\alpha(0)(k-\ell)  u(1+\vartheta) / 2}+\Lambda \\
& \lesssim \Lambda .
\end{align*}

\begin{align*}
    I_{2}&=  \sup_{\vartheta>0 }\sup _{m_0<0  , m_0\in  \mathbb{Z}} 2^{-m_0\Gamma u(1+\vartheta)} \epsilon  ^\theta \sum_{k=-\infty}^{m_0} 2^{  \alpha(0) k u(1+\vartheta)}\left(\sum_{\ell=-\infty }^{k-1}\left|\lambda_\ell \right|\left\|\left(G_{N} a_\ell \right) \chi_{k}\right\|_{{p(\cdot)}}\right)^{u(1+\vartheta)} \\
    & \lesssim  \sup_{\vartheta>0 }\sup _{m_0<0  , m_0\in  \mathbb{Z}} 2^{-m_0\Gamma u(1+\vartheta)} \epsilon  ^\theta \sum_{k=-\infty}^{m_0} 2^{  \alpha(0) k u(1+\vartheta)} \left(\sum_{\ell=-\infty }^{k-1}\left|\lambda_\ell \right| 2^{\ell(m+1)-k(n+m+1)}\left|B_\ell \right|^{\frac{-\alpha_\ell }  {n}}\left\|\chi_{B_\ell }\right\|_{{p^{\prime}(\cdot)}}\left\|\chi_{B_{k}}\right\|_{{p(\cdot)}}\right)^{u(1+\vartheta)}\\
    & \lesssim  \sup_{\vartheta>0 }\sup _{m_0<0  , m_0\in  \mathbb{Z}} 2^{-m_0\Gamma u(1+\vartheta)} \epsilon  ^\theta \sum_{k=-\infty}^{m_0} 2^{  \alpha(0) k u(1+\vartheta)}  \left(\sum_{\ell=-\infty }^{k-1}\left|\lambda_\ell \right| 2^{(\ell-k)\left(m+1+n \delta_{2}\right)-(\ell-k) \alpha(0)}\right)^{u(1+\vartheta)}\\
     & \lesssim  \sup_{\vartheta>0 }\sup _{m_0<0  , m_0\in  \mathbb{Z}} 2^{-m_0\Gamma u(1+\vartheta)} \epsilon  ^\theta \sum_{k=-\infty}^{m_0}   \left(\sum_{\ell=-\infty }^{k-1}\left|\lambda_\ell \right| 2^{(\ell-k)\left(m+1+n \delta_{2}- \alpha(0)\right)}\right)^{u(1+\vartheta)}\\
     & \lesssim  \sup_{\vartheta>0 }\sup _{m_0<0  , m_0\in  \mathbb{Z}} 2^{-m_0\Gamma u(1+\vartheta)} \epsilon  ^\theta \sum_{k=-\infty}^{m_0}   \left(\sum_{\ell=-\infty }^{k-1}\left|\lambda_\ell \right|^{u(1+\vartheta)} 2^{(\ell-k)\left(m+1+n \delta_{2}- \alpha(0)\right)u(1+\vartheta)/2}\right)\left(\sum_{\ell=-\infty }^{k-1} 2^{(\ell-k)\left(m+1+n \delta_{2}- \alpha(0)\right)(u(1+\vartheta))^\prime/2}\right)^{u(1+\vartheta)/(u(1+\vartheta))^\prime}\\
         & \lesssim  \sup_{\vartheta>0 }\sup _{m_0<0  , m_0\in  \mathbb{Z}} 2^{-m_0\Gamma u(1+\vartheta)} \epsilon  ^\theta \sum_{k=-\infty}^{m_0}   \left(\sum_{\ell=-\infty }^{k-1}\left|\lambda_\ell \right|^{u(1+\vartheta)} 2^{(\ell-k)\left(m+1+n \delta_{2}- \alpha(0)\right)u(1+\vartheta)/2}\right)\\
           & \lesssim  \sup_{\vartheta>0 }\sup _{m_0<0  , m_0\in  \mathbb{Z}} 2^{-m_0\Gamma u(1+\vartheta)} \epsilon  ^\theta \sum_{\ell=-\infty }^{m_0-1} \left|\lambda_\ell \right|^{u(1+\vartheta)}  \sum_{k=\ell+1}^{\ell } 2^{(\ell-k)\left(m+1+n \delta_{2}- \alpha(0)\right)u(1+\vartheta)/2}\\
            & \lesssim  \sup_{\vartheta>0 }\sup _{m_0<0  , m_0\in  \mathbb{Z}} 2^{-m_0\Gamma u(1+\vartheta)} \epsilon  ^\theta \sum_{\ell=-\infty }^{m_0-1} \left|\lambda_\ell \right|^{u(1+\vartheta)}\\
            & \lesssim \Lambda .
\end{align*}

Secondly, we find  estimate of the $I I$. By using  the same argument, we obtain that

\begin{align*}
    I I_{1}=  & \sup_{\vartheta>0 }\sup _{m_0\geq 0, m_0\in  \mathbb{Z}} 2^{-m_0\Gamma u(1+\vartheta)} \epsilon  ^\theta\sum_{k=-\infty}^{-1} 2^{  \alpha(0) k u(1+\vartheta)}\left(\sum_{\ell=k}^{\infty}\left|\lambda_\ell \right|\left\|a_\ell \right\|_{{p(\cdot)}}\right)^{u(1+\vartheta)} \\
\lesssim & \sup_{\vartheta>0 }\sup _{m_0\geq 0, m_0\in  \mathbb{Z}} 2^{-m_0\Gamma u(1+\vartheta)} \epsilon  ^\theta\sum_{k=-\infty}^{-1} 2^{  \alpha(0) k u(1+\vartheta)}\left(\sum_{\ell=k}^{\infty}\left|\lambda_\ell \right|\left|B_\ell \right|^{\frac{-\alpha_l}{n}}\right)^{u(1+\vartheta)} \\
\lesssim & \sup_{\vartheta>0 }\sup _{m_0\geq 0, m_0\in  \mathbb{Z}} 2^{-m_0\Gamma u(1+\vartheta)} \epsilon  ^\theta \left[ \sum_{k=-\infty}^{-1}\left(\sum_{\ell=k}^{-1}\left|\lambda_\ell \right| 2^{\alpha(0)(k-\ell)}\right)^{u(1+\vartheta)}+\sum_{k=-\infty}^{-1} 2^{\alpha(0) k u(1+\vartheta)}\left(\sum_{\ell=0}^{\infty}\left|\lambda_\ell \right|2^{-\alpha_{\infty} \ell }\right)^{u(1+\vartheta)}\right] \\
\lesssim & \sup_{\vartheta>0 }\sup _{m_0\geq 0, m_0\in  \mathbb{Z}} 2^{-m_0\Gamma u(1+\vartheta)} \epsilon  ^\theta \left[ \sum_{k=-\infty}^{-1}\left(\sum_{\ell=k}^{-1}\left|\lambda_\ell \right|^{u(1+\vartheta)} 2^{\alpha(0)(k-\ell)(u(1+\vartheta)/2)}\right)\left(\sum_{\ell=k}^{-1} 2^{\alpha(0)(k-\ell)(u(1+\vartheta)^\prime/2)}\right)^{u(1+\vartheta)/(u(1+\vartheta))^\prime}\right.\\
& \left.+\sum_{k=-\infty}^{-1} 2^{\alpha(0) k u(1+\vartheta)}\left(\sum_{\ell=0}^{\infty}\left|\lambda_\ell \right|^{u(1+\vartheta)}2^{-\alpha_{\infty} \ell (u(1+\vartheta))/2}\right)\left(\sum_{\ell=0}^{\infty}2^{-\alpha_{\infty} \ell (u(1+\vartheta))^\prime/2}\right)^{u(1+\vartheta)/(u(1+\vartheta))^\prime}\right] \\
\lesssim & \sup_{\vartheta>0 }\sup _{m_0\geq 0, m_0\in  \mathbb{Z}} 2^{-m_0\Gamma u(1+\vartheta)} \epsilon  ^\theta \left[ \sum_{k=-\infty}^{-1}\left(\sum_{\ell=k}^{-1}\left|\lambda_\ell \right|^{u(1+\vartheta)} 2^{\alpha(0)(k-\ell)(u(1+\vartheta)/2)}\right)+\sum_{k=-\infty}^{-1} 2^{\alpha(0) ku(1+\vartheta)}\left(\sum_{\ell=0}^{\infty}\left|\lambda_\ell \right|^{u(1+\vartheta)}2^{-\alpha_{\infty} \ell (u(1+\vartheta))/2}\right)\right] \\
\lesssim & \sup_{\vartheta>0 }\sup _{m_0\geq 0, m_0\in  \mathbb{Z}} 2^{-m_0\Gamma u(1+\vartheta)} \epsilon  ^\theta \left[ \sum_{k=-\infty}^{-1} \left|\lambda_\ell \right|^{u(1+\vartheta)}\left(\sum_{k=-\infty}^{-1} 2^{\alpha(0)(k-\ell)(u(1+\vartheta)/2)}\right)+\sum_{\ell=0}^{\infty}\left|\lambda_\ell \right|^{u(1+\vartheta)}2^{-\alpha_{\infty} \ell (u(1+\vartheta))/2}\sum_{k=-\infty}^{-1} 2^{\alpha(0) ku(1+\vartheta)}\right] \\
\lesssim & \sup_{\vartheta>0 }\sup _{m_0\geq 0, m_0\in  \mathbb{Z}} 2^{-m_0\Gamma u(1+\vartheta)} \epsilon  ^\theta \left[ \sum_{k=-\infty}^{-1} \left|\lambda_\ell \right|^{u(1+\vartheta)}+\sum_{\ell=0}^{\infty}2^{(\Gamma-\alpha_\infty/2)lu(1+\vartheta)}2^{-m_0\Gamma  u(1+\vartheta)}\sum_{i=-\infty}^l \left|\lambda_{i}\right|^{u(1+\vartheta)}\sum_{k=-\infty}^{-1} 2^{\alpha(0) ku(1+\vartheta)}\right] \\
\lesssim & \Lambda + \Lambda \sum_{\ell=0}^{\infty}2^{(\Gamma-\alpha_\infty/2)lu(1+\vartheta)}\sum_{k=-\infty}^{-1} 2^{\alpha(0) ku(1+\vartheta)}\\
\lesssim & \Lambda .
\end{align*}

\begin{align*}
I I_{2} & =\sup_{\vartheta>0 }\sup _{m_0\geq 0, m_0\in  \mathbb{Z}} 2^{-m_0\Gamma u(1+\vartheta)} \epsilon  ^\theta\sum_{k=-\infty}^{-1} 2^{  \alpha(0) k u(1+\vartheta)}\left(\sum_{\ell=-\infty }^{k-1}\left|\lambda_\ell \right|\left\|\left(G_{N} a_\ell \right) \chi_{k}\right\|_{{p(\cdot)}}\right)^{u(1+\vartheta)}  \\
& \lesssim \sup_{\vartheta>0 }\sup _{m_0\geq 0, m_0\in  \mathbb{Z}} 2^{-m_0\Gamma u(1+\vartheta)} \epsilon  ^\theta \sum_{k=-\infty}^{-1} 2^{k q \alpha(0)}\left(\sum_{\ell=-\infty }^{k-1}\left|\lambda_\ell \right| 2^{\ell(m+1)-k(n+m+1)}\left|B_\ell \right|^{\frac{-\alpha_\ell } {n}}\left\|\chi_{B_\ell }\right\|_{{p^{\prime}(\cdot)}}\left\|\chi_{B_{k}}\right\|_{{p (\cdot)}}\right)^{u(1+\vartheta)} \\
& \lesssim \sup_{\vartheta>0 }\sup _{m_0\geq 0, m_0\in  \mathbb{Z}} 2^{-m_0\Gamma u(1+\vartheta)} \epsilon  ^\theta \sum_{k=-\infty}^{-1}\left(\sum_{\ell=-\infty }^{k-1}\left|\lambda_\ell \right| 2^{(\ell-k)\left(m+1+n \delta_{2}\right)-(\ell-k) \alpha(0)}\right)^{u(1+\vartheta)} \\
& \lesssim \sup_{\vartheta>0 }\sup _{m_0\geq 0, m_0\in  \mathbb{Z}} 2^{-m_0\Gamma u(1+\vartheta)} \epsilon  ^\theta \sum_{k=-\infty}^{-1}\left(\sum_{\ell=-\infty }^{k-1}\left|\lambda_\ell \right| 2^{(\ell-k)\left(m+1+n \delta_{2}- \alpha(0)\right)}\right)^{u(1+\vartheta)} \\
& \lesssim \sup_{\vartheta>0 }\sup _{m_0\geq 0, m_0\in  \mathbb{Z}} 2^{-m_0\Gamma u(1+\vartheta)} \epsilon  ^\theta \sum_{k=-\infty}^{-1}\left(\sum_{\ell=-\infty }^{k-1}\left|\lambda_\ell \right|^{u(1+\vartheta)} 2^{(\ell-k)\left(m+1+n \delta_{2}- \alpha(0)\right)(u(1+\vartheta))/2}\right) \left(\sum_{\ell=-\infty }^{k-1} 2^{(\ell-k)\left(m+1+n \delta_{2}- \alpha(0)\right)(u(1+\vartheta))^\prime/2}\right)^{{u(1+\vartheta)}/{(u(1+\vartheta))^\prime}} \\
& \lesssim \sup_{\vartheta>0 }\sup _{m_0\geq 0, m_0\in  \mathbb{Z}} 2^{-m_0\Gamma u(1+\vartheta)} \epsilon  ^\theta \sum_{k=-\infty}^{-1}\left(\sum_{\ell=-\infty }^{k-1}\left|\lambda_\ell \right|^{u(1+\vartheta)} 2^{(\ell-k)\left(m+1+n \delta_{2}- \alpha(0)\right)(u(1+\vartheta))/2}\right)  \\
& \lesssim \sup_{\vartheta>0 }\sup _{m_0\geq 0, m_0\in  \mathbb{Z}} 2^{-m_0\Gamma u(1+\vartheta)} \epsilon  ^\theta \sum_{\ell=-\infty }^{-1}\sum_{k=\ell+1}^{-1}\left|\lambda_\ell \right|^{u(1+\vartheta)} 2^{(\ell-k)\left(m+1+n \delta_{2}- \alpha(0)\right)(u(1+\vartheta))/2}  \\
& \lesssim \Lambda .
\end{align*}

Next we find the  estimate of $ III$. We have

\begin{align*}
    I I I_{1}&=\sup_{\vartheta>0 }\sup _{m_0\geq 0, m_0\in  \mathbb{Z}} 2^{-m_0\Gamma u(1+\vartheta)} \epsilon  ^\theta \sum_{k=0}^{m_0} 2^{ \alpha_{\infty} k u(1+\vartheta)}\left(\sum_{\ell=k}^{m_0}\left|\lambda_\ell \right|\left\|a_\ell \right\|_{{p(\cdot)}}\right)^{u(1+\vartheta)} \\
& \lesssim \sup_{\vartheta>0 }\sup _{m_0\geq 0, m_0\in  \mathbb{Z}} 2^{-m_0\Gamma u(1+\vartheta)} \epsilon  ^\theta \sum_{k=0}^{m_0} 2^{k  \alpha_{\infty}u(1+\vartheta)}\left|B_{k}\right|^{\alpha_{\infty} u(1+\vartheta) / n}\left(\sum_{\ell=k}^{\infty}\left|\lambda_\ell \right|\left|B_\ell \right|^{-\alpha_\ell  / n}\right)^{u(1+\vartheta)} \\
& \lesssim \sup_{\vartheta>0 }\sup _{m_0\geq 0, m_0\in  \mathbb{Z}} 2^{-m_0\Gamma u(1+\vartheta)} \epsilon  ^\theta \sum_{k=0}^{m_0} 2^{k  \alpha_{\infty} u(1+\vartheta)}\left|B_{k}\right|^{\alpha_{\infty}  u(1+\vartheta) / n}\left(\sum_{\ell=k}^{\infty}\left|\lambda_\ell \right|^{ u(1+\vartheta)}\left|B_\ell \right|^{-\alpha_\ell   u(1+\vartheta) / 2n}\right) \left(\sum_{\ell=k}^{\infty}\left|B_\ell \right|^{-\alpha_\ell  ( u(1+\vartheta))^{\prime} / 2n}\right)^{  u(1+\vartheta)/  (u(1+\vartheta))^{\prime}} \\
& \lesssim \sup_{\vartheta>0 }\sup _{m_0\geq 0, m_0\in  \mathbb{Z}} 2^{-m_0\Gamma u(1+\vartheta)} \epsilon  ^\theta \sum_{k=0}^{m_0} 2^{k  \alpha_{\infty} u(1+\vartheta)}\left(\sum_{\ell=k}^{\infty}\left|\lambda_\ell \right|^{ u(1+\vartheta)}\left|B_\ell \right|^{-\alpha_\ell   u(1+\vartheta) / 2n}\left|B_{k}\right|^{\alpha_{\infty}  u(1+\vartheta) / 2n}\right) \\
& \times \left(\sum_{\ell=k}^{\infty}\left|B_\ell \right|^{-\alpha_\ell  ( u(1+\vartheta))^{\prime} / 2n}\left|B_{k}\right|^{\alpha_{\infty}  (u(1+\vartheta))^\prime / 2n}\right)^{  u(1+\vartheta)/  (u(1+\vartheta))^{\prime}} \\
& \lesssim \sup_{\vartheta>0 }\sup _{m_0\geq 0, m_0\in  \mathbb{Z}} 2^{-m_0\Gamma u(1+\vartheta)} \epsilon  ^\theta \sum_{k=0}^{m_0} \sum_{\ell=k}^{\infty}\left|\lambda_\ell \right|^{ u(1+\vartheta)}\left|B_\ell \right|^{-\alpha_\ell   u(1+\vartheta) / 2n}\left|B_{k}\right|^{\alpha_{\infty}  u(1+\vartheta) / 2n} \\
& \lesssim \sup_{\vartheta>0 }\sup _{m_0\geq 0, m_0\in  \mathbb{Z}} 2^{-m_0\Gamma u(1+\vartheta)} \epsilon  ^\theta \sum_{k=0}^{m_0} \sum_{\ell=k}^{\infty}\left|\lambda_\ell \right|^{ u(1+\vartheta)}2^{(k-\ell)\alpha_\infty u(1+\vartheta)/2}\\
& \lesssim \sup_{\vartheta>0 }\sup _{m_0\geq 0, m_0\in  \mathbb{Z}} 2^{-m_0\Gamma u(1+\vartheta)} \epsilon  ^\theta\left[\sum_{\ell=0}^{\ell }\left|\lambda_\ell \right|^{u(1+\vartheta)} \sum_{k=0}^{m_0} 2^{(k-\ell) \alpha_{\infty} u(1+\vartheta) / 2}+\sum_{\ell=m_0}^{\infty}\left|\lambda_\ell \right|^{u(1+\vartheta)} \sum_{k=0}^{m_0} 2^{(k-\ell) \alpha_{\infty} u(1+\vartheta) / 2}\right] \\
& \lesssim \sup_{\vartheta>0 }\sup _{m_0\geq 0, m_0\in  \mathbb{Z}} 2^{-m_0\Gamma u(1+\vartheta)} \epsilon  ^\theta \sum_{\ell=0}^{\ell }\left|\lambda_\ell \right|^{u(1+\vartheta)}  +\sup_{\vartheta>0 }\sup _{m_0\geq 0, m_0\in  \mathbb{Z}}  \epsilon  ^\theta \sum_{\ell=m_0}^{\infty} 2^{\Gamma(l -L) u(1+\vartheta)} 2^{-m_0\Gamma u(1+\vartheta)} \sum_{\ell=-\infty }^{\ell }\left|\lambda_\ell \right|^{u(1+\vartheta)} \sum_{k=0}^{m_0} 2^{(k-\ell) \alpha_{\infty} u(1+\vartheta) / 2} \\
& \lesssim \Lambda+\Lambda \sup_{\vartheta>0 }\sup _{m_0\geq 0, m_0\in  \mathbb{Z}}  \epsilon  ^\theta \sum_{\ell=m_0}^{\infty} 2^{\Gamma(\ell-m_0) u(1+\vartheta)} 2^{(\ell-m_0) \alpha_{\infty} u(1+\vartheta) / 2} \\
& \lesssim \Lambda .
\end{align*}

\begin{align*}
      I I I_{2}&=\sup_{\vartheta>0 }\sup _{m_0\geq 0, m_0\in  \mathbb{Z}} 2^{-m_0\Gamma u(1+\vartheta)} \epsilon  ^\theta\sum_{k=0}^{m_0} 2^{k  \alpha_{\infty}u(1+\vartheta)}\left(\sum_{\ell=-\infty }^{k-1}\left|\lambda_\ell \right|\left\|\left(G_{N} a_\ell \right) \chi_{k}\right\|_{{p(\cdot)}}\right)^{u(1+\vartheta)} \\
      & \lesssim \sup_{\vartheta>0 }\sup _{m_0\geq 0, m_0\in  \mathbb{Z}} 2^{-m_0\Gamma u(1+\vartheta)} \epsilon  ^\theta\sum_{k=0}^{m_0} 2^{k  \alpha_{\infty}u(1+\vartheta)} \left(\sum_{\ell=-\infty }^{k-1}\left|\lambda_\ell \right| 2^{\ell(m+1)-k(n+m+1)}\left|B_\ell \right|^{-\alpha_\ell  / n}\left\|\chi_{B_\ell }\right\|_{{p^{\prime}(\cdot)}}\left\|\chi_{B_{k}}\right\|_{{p(\cdot)}}\right)^{u(1+\vartheta)} \\
       & \lesssim \sup_{\vartheta>0 }\sup _{m_0\geq 0, m_0\in  \mathbb{Z}} 2^{-m_0\Gamma u(1+\vartheta)} \epsilon  ^\theta  \sum_{k=0}^{m_0}\left(\sum_{\ell=-\infty }^{k-1}\left|\lambda_\ell \right| 2^{(\ell-k)\left(m+1+n \delta_{2}\right)-\ell \alpha_\ell + \alpha_{\infty} k }\right)^{u(1+\vartheta)} \\
 & \lesssim \sup_{\vartheta>0 }\sup _{m_0\geq 0, m_0\in  \mathbb{Z}} 2^{-m_0\Gamma u(1+\vartheta)} \epsilon  ^\theta  \sum_{k=0}^{m_0}\left(\sum_{\ell=-\infty }^{-1}\left|\lambda_\ell \right| 2^{(\ell-k)\left(m+1+n \delta_{2}\right)-\ell \alpha(0)+ \alpha_{\infty} k }\right)^{u(1+\vartheta)} \\
 & + \sup_{\vartheta>0 }\sup _{m_0\geq 0, m_0\in  \mathbb{Z}} 2^{-m_0\Gamma u(1+\vartheta)} \epsilon  ^\theta  \sum_{k=0}^{m_0}\left(\sum_{\ell=0}^{k-1}\left|\lambda_\ell \right| 2^{(\ell-k)\left(m+1+n \delta_{2}  -\alpha_{\infty}\right)}\right)^{u(1+\vartheta)} \\ 
 & \lesssim \sup_{\vartheta>0 }\sup _{m_0\geq 0, m_0\in  \mathbb{Z}} 2^{-m_0\Gamma u(1+\vartheta)} \epsilon  ^\theta  \sum_{k=0}^{m_0}2^{ku(1+\vartheta)\left[\alpha_\infty-\left(m+1+n \delta_{2}\right)\right]}\left(\sum_{\ell=-\infty }^{-1}|\lambda_l|2^{\ell \left[\left(m+1+n \delta_{2}-\alpha(0)\right)\right]} \right)^{u(1+\vartheta)}\\
 & + \sup_{\vartheta>0 }\sup _{m_0\geq 0, m_0\in  \mathbb{Z}} 2^{-m_0\Gamma u(1+\vartheta)} \epsilon  ^\theta \sum_{k=0}^{m_0}\left(\sum_{\ell=0}^{k-1}\left|\lambda_\ell \right|^{u(1+\vartheta)} 2^{(\ell-k)\left(m+1+n \delta_{2}-\alpha_{\infty}\right) u(1+\vartheta) / 2}\right) \left(\sum_{\ell=0}^{k-1} 2^{(\ell-k)\left(m+1+n \delta_{2}-\alpha_{\infty}\right) (u(1+\vartheta))^{\prime} / 2}\right)^{u(1+\vartheta) / (u(1+\vartheta))^{\prime}} \\
 & \lesssim \sup_{\vartheta>0 }\sup _{m_0\geq 0, m_0\in  \mathbb{Z}} 2^{-m_0\Gamma u(1+\vartheta)} \epsilon  ^\theta\left(\sum_{\ell=-\infty }^{-1}\left|\lambda_\ell \right|^{u(1+\vartheta)} 2^{\ell \left(m+1+n \delta_{2}-\alpha(0)\right) u(1+\vartheta) / 2}\right) \left(\sum_{\ell=-\infty }^{-1} 2^{\ell \left(m+1+n \delta_{2}-\alpha(0)\right) (u(1+\vartheta))^{\prime} / 2}\right)^{u(1+\vartheta) / (u(1+\vartheta))^{\prime}} \\
 & +\sup_{\vartheta>0 }\sup _{m_0\geq 0, m_0\in  \mathbb{Z}} 2^{-m_0\Gamma u(1+\vartheta)} \epsilon  ^\theta \sum_{k=0}^{m_0} \sum_{\ell=0}^{k-1}\left|\lambda_\ell \right|^{u(1+\vartheta)} 2^{(\ell-k)\left(m+1+n \delta_{2}-\alpha_{\infty}\right) u(1+\vartheta) / 2} \\
  & \lesssim \sup_{\vartheta>0 }\sup _{m_0\geq 0, m_0\in  \mathbb{Z}} 2^{-m_0\Gamma u(1+\vartheta)} \epsilon  ^\theta\left(\sum_{\ell=-\infty }^{-1}\left|\lambda_\ell \right|^{u(1+\vartheta)} 2^{\ell \left(m+1+n \delta_{2}-\alpha(0)\right) u(1+\vartheta) / 2}\right)\\
  & +\sup_{\vartheta>0 }\sup _{m_0\geq 0, m_0\in  \mathbb{Z}} 2^{-m_0\Gamma u(1+\vartheta)} \epsilon  ^\theta \sum_{k=0}^{m_0} \sum_{\ell=0}^{k-1}\left|\lambda_\ell \right|^{u(1+\vartheta)} 2^{(\ell-k)\left(m+1+n \delta_{2}-\alpha_{\infty}\right) u(1+\vartheta) / 2} \\
    & \lesssim \sup_{\vartheta>0 }\sup _{m_0\geq 0, m_0\in  \mathbb{Z}} 2^{-m_0\Gamma u(1+\vartheta)} \epsilon  ^\theta \sum_{\ell=-\infty }^{-1}\left|\lambda_\ell \right|^{u(1+\vartheta)}+\lesssim \sup_{\vartheta>0 }\sup _{m_0\geq 0, m_0\in  \mathbb{Z}} 2^{-m_0\Gamma u(1+\vartheta)} \epsilon  ^\theta \sum_{\ell=0}^{\ell }\left|\lambda_\ell \right|^{u(1+\vartheta)}\\
    & \lesssim \Lambda .
\end{align*}

Thus, we completes our proof.
\end{proof}
\end{theorem}
\section{ Applications}
We are going to demonstrate the following outcome as an application of the atomic decompositions.

\begin{theorem}

 Let $0 \leq \Gamma<\infty$,  $0<u<\infty,$ , $p(\cdot) \in \mathcal{B}\left(\mathbb{R}^{n}\right)$, and $\alpha(\cdot) \in L^{\infty}\left(\mathbb{R}^{n}\right)$ be log-H\"older continuous both at the origin and infinity.  Let  $\alpha(\cdot)\geq 2\lambda$, with $n \delta_{2} \leq  \alpha(0), \alpha_{\infty}<\infty$. Let the integer $s=\max \left\{\left[\alpha(0)-n \delta_{2}\right],\left[\alpha_{\infty}-n \delta_{2}\right]\right\}$, and $\delta_{2}$ as in Lemma 1.3. Let a sublinear operator $T$ satisfying the following conditions:

(i) $T$ is bounded on $L^{p(\cdot)}$.

(ii) There exists a constant $\delta>0$ such that $s+\delta>$ $\max \left\{\alpha(0)-n \delta_{2}, \alpha_{\infty}-n \delta_{2}\right\}$, and for any compact support function $g$ with

\begin{equation*}
\int_{\mathbb{R}^{n}} f(x) x^{\beta} \mathrm{d} x=0, \quad|\beta| \leq  s. 
\end{equation*}

$Tf$ satisfies the size condition

\[
\begin{array}{r}
|T f(x)| \leq  C(\operatorname{diam}(\operatorname{supp} f))^{s+\delta}|x|^{-(n+s+\delta)}\|f\|_{1}  , 
\end{array}
\]
\text { when } $\operatorname{dist}(x, \text { supp } f) \geqslant \frac{|x|}{2}$. 
Then there exists a constant $C$ such that

\begin{align*}
\|T f\|_{M \dot{K} ^{\alpha(\cdot),u),\theta}_{ \Gamma, p(\cdot)}(\mathbb{R}^n)} & \leq  C\|f\|_{H M \dot{K} ^{\alpha(\cdot),u),\theta}_{ \Gamma, p(\cdot)}(\mathbb{R}^n)} ,
\end{align*}

for $f \in H {M \dot{K} ^{\alpha(\cdot),u),\theta}_{ \Gamma, p(\cdot)}(\mathbb{R}^n)}$ .

\end{theorem}

\begin{proof}

 Suppose that  $f \in  H {M \dot{K} ^{\alpha(\cdot),u),\theta}_{ \Gamma, p(\cdot)}(\mathbb{R}^n)}$. By using Theorem 2.3, $f=\sum_{i=-\infty}^{\infty} \lambda_{i} b_{i}$ converges in $\mathcal{S}^{\prime}\left(\mathbb{R}^{n}\right)$, where each $b_{i}$ is a central $(\alpha(\cdot), p(\cdot))$-atom with support contained in $B_{i}$ and

\begin{equation*}
\|f\|_ {H {M \dot{K} ^{\alpha(\cdot),u),\theta}_{ \Gamma, p(\cdot)}(\mathbb{R}^n)}} \approx \inf \left(\sup_{\vartheta>0 }\sup _{m_0\in  \mathbb{Z}} 2^{-m_0\Gamma}\left(\epsilon  ^\theta\sum_{i=-\infty}^{\ell }\left|\lambda_{i}\right|^{u(1+\vartheta)}\right)^{\frac{1}{u(1+\vartheta)}}\right) . \tag{73}
\end{equation*}

To keep things simple, we denote 
 $\Lambda=\sup_{\vartheta>0 }\sup _{m_0\in  \mathbb{Z}} 2^{-m_0\Gamma u(1+\vartheta)}\epsilon  ^\theta \sum_{i=-\infty}^{\ell }\left|\lambda_{i}\right|^{u(1+\vartheta)}$. By Proposition $1.6$, we have

\begin{align*}
& \|T f\|_{M \dot{K} ^{\alpha(\cdot),u),\theta}_{ \Gamma, p(\cdot)}(\mathbb{R}^n)}^{u(1+\vartheta)} \\
& \approx \max \left\{\sup_{\vartheta>0 }\sup _{m_0<0, m_0\in  \mathbb{Z}} 2^{-m_0\Gamma u(1+\vartheta)} \epsilon  ^\theta \left(\sum_{k=-\infty}^{m_0} 2^{k  \alpha(0) u(1+\vartheta)}\left\|(T f) \chi_{k}\right\|_{{p(\cdot)}}^{u(1+\vartheta)}\right),\right. \\
& \sup_{\vartheta>0 }\sup _{m_0\geq 0, m_0\in  \mathbb{Z}} 2^{-L \lambda u(1+\vartheta)} \epsilon  ^\theta \left(\sum_{k=-\infty}^{-1} 2^{k  \alpha(0) u(1+\vartheta)}\left\|(T f) \chi_{k}\right\|_{{p(\cdot)}}^{u(1+\vartheta)}\right. \\
& \left.\left.\quad+\sum_{k=0}^{m_0} 2^{k \alpha_\infty u(1+\vartheta)}\left\|(T f) \chi_{k}\right\|_{{p(\cdot)}}^{u(1+\vartheta)}\right)\right\} \\
& \lesssim \max \{I, I I+I I I\}. 
\end{align*}

We will find the estimated for $I$ and $III$ and estimate of $II$ can be obtained similarly.

We just need to demonstrate that there is a positive constant $C$ such that $I, I I, I I I \leq C \Lambda$ in order to finish our proof.

Firstly, we will find the  estimate of $I$ :

\begin{align*}
I= & \sup_{\vartheta>0 }\sup _{m_0<0, m_0\in  \mathbb{Z}} 2^{-m_0\Gamma u(1+\vartheta)} \epsilon  ^\theta \left(\sum_{k=-\infty}^{m_0} 2^{k  \alpha(0) u(1+\vartheta)}\left\|(T f) \chi_{k}\right\|_{{p(\cdot)}}^{u(1+\vartheta)}\right) \\
\lesssim & \sup_{\vartheta>0 }\sup _{m_0<0, m_0\in  \mathbb{Z}} 2^{-m_0\Gamma u(1+\vartheta)} \epsilon  ^\theta\sum_{k=-\infty}^{m_0} 2^{k \alpha(0) u(1+\vartheta)}\left(\sum_{i=k}^{\infty}\left|\lambda_{i}\right|\left\|\left(T b_{i}\right) \chi_{k}\right\|_{{p(\cdot)}}\right)^{u(1+\vartheta)} \\
+& \sup_{\vartheta>0 }\sup _{m_0<0, m_0\in  \mathbb{Z}} 2^{-m_0\Gamma u(1+\vartheta)} \epsilon  ^\theta \sum_{k=-\infty}^{m_0} 2^{k \alpha(0) u(1+\vartheta)}\left(\sum_{i=-\infty}^{k-1}\left|\lambda_{i}\right|\left\|\left(T b_{i}\right) \chi_{k}\right\|_{{p(\cdot)}}\right)^{u(1+\vartheta)} \\
:= & I_{1}+I_{2} . 
\end{align*}

By the boundedness of $T$ in ${p(\cdot)}$, we have

\begin{equation*}
\left\|\left(T b_{i}\right) \chi_{k}\right\|_{{p(\cdot)}} \leq \left\|b_{i}\right\|_{{p(\cdot)}} \leq \left|B_{i}\right|^{-\alpha_{i} / n}=2^{-\alpha_{i} i} . 
\end{equation*}

Therefore, when $0<u(1+\vartheta) \leq  1$, we get

\begin{align*}
I_{1}= & \sup_{\vartheta>0 }\sup _{m_0<0, m_0\in  \mathbb{Z}} 2^{-m_0\Gamma u(1+\vartheta)} \epsilon  ^\theta\sum_{k=-\infty}^{m_0} 2^{k \alpha(0) u(1+\vartheta)}\left(\sum_{i=k}^{\infty}\left|\lambda_{i}\right|\left\|\left(T b_{i}\right) \chi_{k}\right\|_{{p(\cdot)}}\right)^{u(1+\vartheta)}\\
\lesssim  & \sup_{\vartheta>0 }\sup _{m_0<0, m_0\in  \mathbb{Z}} 2^{-m_0\Gamma u(1+\vartheta)} \epsilon  ^\theta \sum_{k=-\infty}^{m_0} 2^{\alpha(0)k u(1+\vartheta)  }\left(\sum_{i=k}^{\infty}\left|\lambda_{i}\right| 2^{-\alpha_{i} i}\right)^{u(1+\vartheta)} \\
\lesssim  & \sup_{\vartheta>0 }\sup _{m_0<0, m_0\in  \mathbb{Z}} 2^{-m_0\Gamma u(1+\vartheta)} \epsilon  ^\theta \\
& \times \sum_{k=-\infty}^{m_0} 2^{k \alpha(0) u(1+\vartheta) }\left(\sum_{i=k}^{-1}\left|\lambda_{i}\right|^{u(1+\vartheta)} 2^{-\alpha(0) iu (1+\vartheta)}+\sum_{i=0}^{\infty}\left|\lambda_{i}\right|^{u(1+\vartheta)} 2^{-\alpha_{\infty} iu (1+\vartheta)}\right) \\
\lesssim  & \sup_{\vartheta>0 }\sup _{m_0<0, m_0\in  \mathbb{Z}} 2^{-m_0\Gamma u(1+\vartheta)} \epsilon  ^\theta \sum_{k=-\infty}^{m_0} \sum_{i=k}^{-1}\left|\lambda_{i}\right|^{u(1+\vartheta)} 2^{\alpha(0)(k-i) u(1+\vartheta)} \\
+& \sup_{\vartheta>0 }\sup _{m_0<0, m_0\in  \mathbb{Z}} 2^{-m_0\Gamma u(1+\vartheta)} \epsilon  ^\theta \sum_{k=-\infty}^{m_0} 2^{\alpha(0)k  u(1+\vartheta)} \sum_{i=0}^{\infty}\left|\lambda_{i}\right|^{u(1+\vartheta)} 2^{-\alpha_{\infty} iu (1+\vartheta)} \\
\lesssim  & \sup_{\vartheta>0 }\sup _{m_0<0, m_0\in  \mathbb{Z}} 2^{-m_0\Gamma u(1+\vartheta)} \epsilon  ^\theta \sum_{i=-\infty}^{-1}\left|\lambda_{i}\right|^{u(1+\vartheta)} \sum_{k=-\infty}^{i} 2^{\alpha(0)(k-i) u(1+\vartheta)} \\
+& \sup_{\vartheta>0 }\sup _{m_0<0  , m_0\in  \mathbb{Z}} \epsilon  ^\theta \sum_{i=0}^{\infty} 2^{-i\Gamma  u(1+\vartheta)}\left|\lambda_{i}\right|^{u(1+\vartheta)} 2^{\left(\Gamma-\alpha_{\infty}\right) iu (1+\vartheta)} 2^{-m_0\Gamma u(1+\vartheta)} \sum_{k=-\infty}^{m_0} 2^{\alpha(0) k u(1+\vartheta)}\\
\lesssim  & \sup_{\vartheta>0 }\sup _{m_0<0, m_0\in  \mathbb{Z}} 2^{-m_0\Gamma u(1+\vartheta)} \epsilon  ^\theta \sum_{i=-\infty}^{\ell }\left|\lambda_{i}\right|^{u(1+\vartheta)}+\sup_{\vartheta>0 }\sup _{m_0<0, m_0\in  \mathbb{Z}} 2^{-m_0\Gamma u(1+\vartheta)} \epsilon  ^\theta \sum_{i=m_0}^{-1}\left|\lambda_{i}\right|^{u(1+\vartheta)} \sum_{k=-\infty}^{i} 2^{\alpha(0)(k-i) u(1+\vartheta)}   \\
+ & \Lambda  \sup_{\vartheta>0 }\sup _{m_0<0  , m_0\in  \mathbb{Z}} \epsilon  ^\theta \sum_{i=0}^{\infty} 2^{\left(\Gamma-\alpha_{\infty}\right) iu (1+\vartheta)}  \sum_{k=-\infty}^{m_0} 2^{\alpha(0) k u(1+\vartheta)-m_0\Gamma u(1+\vartheta)}\\
\lesssim & \Lambda +\sup_{\vartheta>0 }\sup _{m_0<0, m_0\in  \mathbb{Z}}  \epsilon  ^\theta \sum_{i=m_0}^{-1}2^{-j\Gamma u(1+\vartheta)}\left|\lambda_{i}\right|^{u(1+\vartheta)} 2^{(i-m_0) \Gamma u(1+\vartheta)}\sum_{k=-\infty}^{i} 2^{\alpha(0)(k-i) u(1+\vartheta)}  + \Lambda \\
\lesssim & \Lambda + \Lambda \sup_{\vartheta>0 }\sup _{m_0<0, m_0\in  \mathbb{Z}}  \epsilon  ^\theta \sum_{i=m_0}^{-1} 2^{(i-m_0) \Gamma u(1+\vartheta)}\sum_{k=-\infty}^{i} 2^{\alpha(0)(k-i) u(1+\vartheta)}\\
\lesssim & \Lambda.
\end{align*}

Now we will the estimate for the second case when $1<u(1+\vartheta)<\infty$. Let $\frac{1}{u(1+\vartheta)}+ \frac{1}{(u(1+\vartheta))^\prime}=1$,  we obtain

\begin{align*}
    I_{1}&=  \sup_{\vartheta>0 }\sup _{m_0<0, m_0\in  \mathbb{Z}} 2^{-m_0\Gamma u(1+\vartheta)} \epsilon  ^\theta\sum_{k=-\infty}^{m_0} 2^{k \alpha(0) u(1+\vartheta)}\left(\sum_{i=k}^{\infty}\left|\lambda_{i}\right|\left\|\left(T b_{i}\right) \chi_{k}\right\|_{{p(\cdot)}}\right)^{u(1+\vartheta)}\\
    & \lesssim  \sup_{\vartheta>0 }\sup _{m_0<0, m_0\in  \mathbb{Z}} 2^{-m_0\Gamma u(1+\vartheta)} \epsilon  ^\theta \sum_{k=-\infty}^{m_0} 2^{\alpha(0)k  u(1+\vartheta)}\left(\sum_{i=k}^{\infty}\left|\lambda_{i}\right| 2^{-\alpha_{i} i}\right)^{u(1+\vartheta)} \\
    & \lesssim \sup_{\vartheta>0 }\sup _{m_0<0, m_0\in  \mathbb{Z}} 2^{-m_0\Gamma u(1+\vartheta)} \epsilon  ^\theta \sum_{k=-\infty}^{m_0}\left(\sum_{i=k}^{-1}\left|\lambda_{i}\right| 2^{\alpha(0)(k-i)}\right)^{u(1+\vartheta)} \\
    &+ \sup_{\vartheta>0 }\sup _{m_0<0, m_0\in  \mathbb{Z}} 2^{-m_0\Gamma u(1+\vartheta)} \epsilon  ^\theta  \sum_{k=-\infty}^{m_0} 2^{\alpha(0)k u(1+\vartheta)}\left(\sum_{i=0}^{\infty}\left|\lambda_{i}\right| 2^{-i \alpha_{\infty} u(1+\vartheta)}\right)^{u(1+\vartheta)} \\
    & \lesssim   \sup_{\vartheta>0 }\sup _{m_0<0, m_0\in  \mathbb{Z}} 2^{-m_0\Gamma u(1+\vartheta)} \epsilon  ^\theta \sum_{k=-\infty}^{m_0}\left(\sum_{i=k}^{-1}\left|\lambda_{i}\right|^{u(1+\vartheta)} 2^{\alpha(0)(k-i) u(1+\vartheta) / 2}\right)  \times\left(\sum_{i=k}^{-1} 2^{\alpha(0)(k-i) (u(1+\vartheta))^{\prime} / 2}\right)^{u(1+\vartheta) / (u(1+\vartheta))^{\prime}} \\
    & +\sup_{\vartheta>0 }\sup _{m_0<0, m_0\in  \mathbb{Z}} 2^{-m_0\Gamma u(1+\vartheta)} \epsilon  ^\theta \sum_{k=-\infty}^{m_0} 2^{\alpha(0)k u(1+\vartheta) }\left(\sum_{i=0}^{\infty}\left|\lambda_{i}\right|^{u(1+\vartheta)} 2^{-\alpha_{\infty} iu (1+\vartheta) / 2}\right)  \left(\sum_{i=0}^{\infty} 2^{-\alpha_{\infty} i( u(1+\vartheta))^{\prime} / 2}\right)^{u(1+\vartheta) / (u(1+\vartheta))^{\prime}} \\
    & \lesssim   \sup_{\vartheta>0 }\sup _{m_0<0, m_0\in  \mathbb{Z}} 2^{-m_0\Gamma u(1+\vartheta)} \epsilon  ^\theta \sum_{k=-\infty}^{m_0}\left(\sum_{i=k}^{-1}\left|\lambda_{i}\right|^{u(1+\vartheta)} 2^{\alpha(0)(k-i) u(1+\vartheta) / 2}\right)   \\
    & +\sup_{\vartheta>0 }\sup _{m_0<0, m_0\in  \mathbb{Z}} 2^{-m_0\Gamma u(1+\vartheta)} \epsilon  ^\theta \sum_{k=-\infty}^{m_0} 2^{\alpha(0)k u(1+\vartheta) }\left(\sum_{i=0}^{\infty}\left|\lambda_{i}\right|^{u(1+\vartheta)} 2^{-\alpha_{\infty} iu (1+\vartheta) / 2}\right)  \\
    & \lesssim  \sup_{\vartheta>0 }\sup _{m_0<0, m_0\in  \mathbb{Z}} 2^{-m_0\Gamma u(1+\vartheta)} \epsilon  ^\theta \sum_{i=-\infty}^{\ell }\left|\lambda_{i}\right|^{u(1+\vartheta)} +\sup_{\vartheta>0 }\sup _{m_0<0, m_0\in  \mathbb{Z}} 2^{-m_0\Gamma u(1+\vartheta)} \epsilon  ^\theta \sum_{i=m_0}^{-1}\left|\lambda_{i}\right|^{u(1+\vartheta)} \sum_{k=-\infty}^{i} 2^{\alpha(0)(k-i) u(1+\vartheta) / 2} \\
&+\Lambda \sup_{\vartheta>0 }\sup _{m_0<0, m_0\in  \mathbb{Z}}  \epsilon  ^\theta \sum_{i=0}^{\infty} 2^{\left(\Gamma-\alpha_{\infty} / 2\right) iu (1+\vartheta)} \sum_{k=-\infty}^{m_0} 2^{(\alpha(0)k -m_0\Gamma) u(1+\vartheta)} \\
& \lesssim  \Lambda+\sup_{\vartheta>0 }\sup _{m_0<0, m_0\in  \mathbb{Z}}  \epsilon  ^\theta\sum_{i=m_0}^{-1} 2^{-i\Gamma  u(1+\vartheta)}\left|\lambda_{i}\right|^{u(1+\vartheta)} 2^{(i-m_0) \Gamma u(1+\vartheta)} \sum_{k=-\infty}^{i} 2^{\alpha(0)(k-i) u(1+\vartheta) / 2}+\Lambda \\
& \lesssim \Lambda+\Lambda \sup_{\vartheta>0 }\sup _{m_0<0, m_0\in  \mathbb{Z}}  \epsilon  ^\theta \sum_{i=m_0}^{-1} 2^{(i-m_0) \Gamma u(1+\vartheta)} \sum_{k=-\infty}^{i} 2^{\alpha(0)(k-i) u(1+\vartheta) / 2} \\
& \lesssim \Lambda .
\end{align*}

Second, we estimate $I_{2}$. By using the H\"older's inequality and Lemma 1.3, we get

\begin{align*}
\left|T b_{i}(x)\right| & \lesssim|x|^{-(n+s+\delta)} 2^{i((s+\delta)} \int_{B_{i}}\left|b_{i}(y)\right| \mathrm{d} y \\
& \lesssim 2^{-k(n+s+\delta)} 2^{i((s+\delta)}\left\|b_{i}\right\|_{{p(\cdot)}}\left\|\chi_{B_{i}}\right\|_{{p^{\prime}(\cdot)}}  \\
& \lesssim 2^{i\left(s+\delta-\alpha_{i}\right)-k(s+\delta+n)}\left\|\chi_{B_{i}}\right\|_{{p^{\prime}(\cdot)}} .
\end{align*}

By using these estimates, we have

\begin{align*}
\|( & \left.T b_{i}\right) \chi_{k} \|_{{p(\cdot)}} \\
& \lesssim 2^{i\left(s+\delta-\alpha_{i}\right)-k(s+\delta+n)}\left\|\chi_{B_{i}}\right\|_{{p^{\prime}(\cdot)}}\left\|\chi_{B_{k}}\right\|_{{p(\cdot)}} \\
& \lesssim 2^{i\left(s+\delta-\alpha_{i}\right)-k(s+\delta)} 2^{-k n}\left(\left|B_{k}\right|\left\|\chi_{B_{k}}\right\|_{{p^{\prime}(\cdot)}}^{-1}\right)\left\|\chi_{B_{i}}\right\|_{{p^{\prime}(\cdot)}}  \\
& \lesssim 2^{i\left(s+\delta-\alpha_{i}\right)-k(s+\delta)} \frac{\left\|\chi_{B_{i}}\right\|_{{p^{\prime}(\cdot)}}}{\left\|\chi_{B_{k}}\right\|_{{p^{\prime}(\cdot)}}} \\
& \lesssim 2^{\left(s+\delta+n \delta_{2}\right)(i-k)-i \alpha_{i}} .
\end{align*}

Therefore, when $0<u(1+\vartheta) \leq  1$, by $n \delta_{2} \leq  \alpha(0)<s+\delta+n \delta_{2}$ we get

\begin{align*}
 I_{1}= &  \sup_{\vartheta>0 }\sup _{m_0<0, m_0\in  \mathbb{Z}} 2^{-m_0\Gamma u(1+\vartheta)} \epsilon  ^\theta\sum_{k=-\infty}^{m_0} 2^{k \alpha(0) u(1+\vartheta)}\left(\sum_{i=k}^{\infty}\left|\lambda_{i}\right|\left\|\left(T b_{i}\right) \chi_{k}\right\|_{{p(\cdot)}}\right)^{u(1+\vartheta)}\\
\lesssim &  \sup_{\vartheta>0 }\sup _{m_0<0, m_0\in  \mathbb{Z}} 2^{-m_0\Gamma u(1+\vartheta)} \epsilon  ^\theta \\
& \times \sum_{k=-\infty}^{m_0} 2^{k \alpha(0) u(1+\vartheta)}\left(\sum_{i=-\infty}^{k-1}\left|\lambda_{i}\right|^{u(1+\vartheta)} 2^{\left[\left(s+\delta+n \delta_{2}\right)(i-k)-i \alpha(0)\right] u(1+\vartheta)}\right)\\
=& \sup_{\vartheta>0 }\sup _{m_0<0, m_0\in  \mathbb{Z}} 2^{-m_0\Gamma u(1+\vartheta)} \epsilon  ^\theta \sum_{i=-\infty}^{\ell }\left|\lambda_{i}\right|^{u(1+\vartheta)} \sum_{k=j+1}^{-1} 2^{(i-k)\left(s+\delta+n \delta_{2}-\alpha(0)\right) u(1+\vartheta)} \\
 \lesssim & \Lambda  .
\end{align*}

When $1<u(1+\vartheta)<\infty$, let $\frac{1}{u(1+\vartheta)}+ \frac{1}{(u(1+\vartheta))\prime}=1$. Since $n \delta_{2} \leq  \alpha(0)<$ $s+\delta+n \delta_{2}$, by H\"older's inequality, we have

\begin{align*}
I_{2}= & \sup_{\vartheta>0 }\sup _{m_0<0, m_0\in  \mathbb{Z}} 2^{-m_0\Gamma u(1+\vartheta)} \epsilon  ^\theta \sum_{k=-\infty}^{m_0} 2^{k \alpha(0) u(1+\vartheta)}\left(\sum_{i=-\infty}^{k-1}\left|\lambda_{i}\right|\left\|\left(T b_{i}\right) \chi_{k}\right\|_{{p(\cdot)}}\right)^{u(1+\vartheta)} \\
\lesssim  & \sup_{\vartheta>0 }\sup _{m_0<0, m_0\in  \mathbb{Z}} 2^{-m_0\Gamma u(1+\vartheta)} \epsilon  ^\theta \sum_{k=-\infty}^{m_0} 2^{\alpha(0) k u(1+\vartheta)}\left(\sum_{i=-\infty}^{k-1}\left|\lambda_{i}\right| 2^{\left(s+\delta+n \delta_{2}\right)(i-k)-i \alpha(0)}\right)^{u(1+\vartheta)} \\
\lesssim  & \sup_{\vartheta>0 }\sup _{m_0<0, m_0\in  \mathbb{Z}} 2^{-m_0\Gamma u(1+\vartheta)} \epsilon  ^\theta  \sum_{k=-\infty}^{m_0} 2^{\alpha(0) k u(1+\vartheta)}\left(\sum_{i=-\infty}^{k-1}\left|\lambda_{i}\right|^{u(1+\vartheta)} 2^{(i-k)\left(s+\delta+n \delta_{2}-\alpha(0)\right) u(1+\vartheta) / 2}\right) \\
& \times\left(\sum_{i=-\infty}^{k-1} 2^{(i-k)\left(s+\delta+n \delta_{2}-\alpha(0)\right) (u(1+\vartheta))^{\prime} / 2}\right)^{u(1+\vartheta) / (u(1+\vartheta))^{\prime}} \\
\lesssim  &  \sup_{\vartheta>0 }\sup _{m_0<0, m_0\in  \mathbb{Z}} 2^{-m_0\Gamma u(1+\vartheta)} \epsilon  ^\theta  \sum_{k=-\infty}^{m_0} 2^{\alpha(0) k u(1+\vartheta)}\left(\sum_{i=-\infty}^{k-1}\left|\lambda_{i}\right|^{u(1+\vartheta)} 2^{(i-k)\left(s+\delta+n \delta_{2}-\alpha(0)\right) u(1+\vartheta) / 2}\right) \\
= & \sup_{\vartheta>0 }\sup _{m_0<0, m_0\in  \mathbb{Z}} 2^{-m_0\Gamma u(1+\vartheta)} \epsilon  ^\theta \sum_{i=-\infty}^{\ell }\left|\lambda_{i}\right|^{u(1+\vartheta)} \sum_{k=j+1}^{-1} 2^{(i-k)\left(s+\delta+n \delta_{2}-\alpha(0)\right) u(1+\vartheta) / 2}\\
\lesssim & \Lambda .
\end{align*}

Finally, we estimate $I I I$ :

\begin{align*}
I I I&=  \sup_{\vartheta>0 }\sup _{m_0<0, m_0\in  \mathbb{Z}} 2^{-m_0\Gamma u(1+\vartheta)} \epsilon  ^\theta \sum_{k=0}^{m_0} 2^{k  \alpha_{\infty}u(1+\vartheta)}\left\|(T f) \chi_{k}\right\|_{{p(\cdot)}}^{u(1+\vartheta)} \\
& \lesssim \sup_{\vartheta>0 }\sup _{m_0<0, m_0\in  \mathbb{Z}} 2^{-m_0\Gamma u(1+\vartheta)} \epsilon  ^\theta \sum_{k=0}^{m_0} 2^{k  \alpha_{\infty}u(1+\vartheta)}\left(\sum_{i=k}^{\infty}\left|\lambda_{i}\right|\left\|\left(T b_{i}\right) \chi_{k}\right\|_{{p(\cdot)}}\right)^{u(1+\vartheta)} \\
& +\sup_{\vartheta>0 }\sup _{m_0<0, m_0\in  \mathbb{Z}} 2^{-m_0\Gamma u(1+\vartheta)} \epsilon  ^\theta \sum_{k=0}^{m_0} 2^{k  \alpha_{\infty}u(1+\vartheta)}\left(\sum_{i=-\infty}^{k-1}\left|\lambda_{i}\right|\left\|\left(T b_{i}\right) \chi_{k}\right\|_{{p(\cdot)}}\right)^{u(1+\vartheta)} \\
& :=I I I_{1}+I I I_{2} . \tag{89}
\end{align*}

When $0<u(1+\vartheta) \leq  1$, by the boundedness of $T$ in $L^{p(\cdot)}$, we have

\begin{align*}
I I I_{1}= & \sup_{\vartheta>0 }\sup _{m_0<0, m_0\in  \mathbb{Z}} 2^{-m_0\Gamma u(1+\vartheta)} \epsilon  ^\theta \sum_{k=0}^{m_0} 2^{k  \alpha_{\infty}u(1+\vartheta)}\left(\sum_{i=k}^{\infty}\left|\lambda_{i}\right|\left\|\left(T b_{i}\right) \chi_{k}\right\|_{{p(\cdot)}}\right)^{u(1+\vartheta)} \\
\lesssim & \sup_{\vartheta>0 }\sup _{m_0<0, m_0\in  \mathbb{Z}} 2^{-m_0\Gamma u(1+\vartheta)} \epsilon  ^\theta \sum_{k=0}^{m_0} 2^{k\alpha_{\infty} u(1+\vartheta) } \sum_{i=k}^{\infty}\left|\lambda_{i}\right|^{u(1+\vartheta)}\left\|\left(T b_{i}\right) \chi_{k}\right\|_{{p(\cdot)}}^{u(1+\vartheta)} \\
\lesssim & \sup_{\vartheta>0 }\sup _{m_0<0, m_0\in  \mathbb{Z}} 2^{-m_0\Gamma u(1+\vartheta)} \epsilon  ^\theta \sum_{k=0}^{m_0} 2^{k\alpha_{\infty}u(1+\vartheta)  } \sum_{i=k}^{\infty}\left|\lambda_{i}\right|^{u(1+\vartheta)} 2^{-\alpha_{i}j  u(1+\vartheta)} \\
\lesssim & \sup_{\vartheta>0 }\sup _{m_0<0, m_0\in  \mathbb{Z}} 2^{-m_0\Gamma u(1+\vartheta)} \epsilon  ^\theta \sum_{k=0}^{m_0} 2^{\alpha_{\infty} k u(1+\vartheta)} \sum_{i=k}^{\infty}\left|\lambda_{i}\right|^{u(1+\vartheta)} 2^{-\alpha_{\infty} iu (1+\vartheta)} \\
= & \sup_{\vartheta>0 }\sup _{m_0<0, m_0\in  \mathbb{Z}} 2^{-m_0\Gamma u(1+\vartheta)} \epsilon  ^\theta \sum_{i=0}^{m_0}\left|\lambda_{i}\right|^{u(1+\vartheta)} \sum_{k=0}^{i} 2^{(k-i) \alpha_{\infty} u(1+\vartheta)} \\
+ & \sup_{\vartheta>0 }\sup _{m_0<0, m_0\in  \mathbb{Z}} 2^{-m_0\Gamma u(1+\vartheta)} \epsilon  ^\theta \sum_{i=m_0}^{\infty}\left|\lambda_{i}\right|^{u(1+\vartheta)} \sum_{k=0}^{m_0} 2^{(k-i) \alpha_{\infty} u(1+\vartheta)} \\
\lesssim & \sup_{\vartheta>0 }\sup _{m_0<0, m_0\in  \mathbb{Z}} 2^{-m_0\Gamma u(1+\vartheta)} \epsilon  ^\theta \sum_{i=0}^{m_0}\left|\lambda_{i}\right|^{u(1+\vartheta)} \\
+ &\sup_{\vartheta>0 }\sup _{m_0<0, m_0\in  \mathbb{Z}} \epsilon  ^\theta \sum_{i=m_0}^{\infty} 2^{(i\Gamma  u(1+\vartheta)-m_0\Gamma u(1+\vartheta))} 2^{-i\Gamma  u(1+\vartheta)} \sum_{i=-\infty}^{i}\left|\lambda_{i}\right|^{u(1+\vartheta)} \sum_{k=0}^{m_0} 2^{(k-i) \alpha_{\infty} u(1+\vartheta)} \\
\lesssim & \Lambda+\Lambda \sup_{\vartheta>0 }\sup _{m_0<0, m_0\in  \mathbb{Z}} \epsilon  ^\theta \sum_{i=m_0}^{\infty} 2^{(i-m_0) \Gamma u(1+\vartheta)} 2^{(L-j) \alpha_{\infty} u(1+\vartheta)} \\
\lesssim & \Lambda+\Lambda \sup_{\vartheta>0 }\sup _{m_0<0, m_0\in  \mathbb{Z}} \epsilon  ^\theta \sum_{i=m_0}^{\infty} 2^{(i-m_0) u(1+\vartheta)\left(\Gamma-\alpha_{\infty}\right)} \\
\lesssim & \Lambda .
\end{align*}

Due to H\"older's inequality and the boundedness of $T$ in $L^{p(\cdot)}$, for $1<u(1+\vartheta)<\infty$, we obtain

\begin{align*}
I I I_{1}&=  \sup_{\vartheta>0 }\sup _{m_0<0, m_0\in  \mathbb{Z}} 2^{-m_0\Gamma u(1+\vartheta)} \epsilon  ^\theta \sum_{k=0}^{m_0} 2^{k  \alpha_{\infty}u(1+\vartheta)}\left(\sum_{i=k}^{\infty}\left|\lambda_{i}\right|\left\|\left(T b_{i}\right) \chi_{k}\right\|_{{p(\cdot)}}\right)^{u(1+\vartheta)} \\
& \lesssim \sup_{\vartheta>0 }\sup _{m_0<0, m_0\in  \mathbb{Z}} 2^{-m_0\Gamma u(1+\vartheta)} \epsilon  ^\theta \sum_{k=0}^{m_0} 2^{\alpha_{\infty} k u(1+\vartheta)}\left(\sum_{i=k}^{\infty}\left|\lambda_{i}\right|^{u(1+\vartheta)}\left\|\left(T b_{i}\right) \chi_{k}\right\|_{{p(\cdot)}}^{u(1+\vartheta) / 2}\right) \\
& \times\left(\sum_{i=k}^{\infty}\left\|\left(T b_{i}\right) \chi_{k}\right\|_{{p(\cdot)}}^{(u(1+\vartheta))^{\prime} / 2}\right)^{u(1+\vartheta) / (u(1+\vartheta))^{\prime}} \\
& \lesssim  \sup_{\vartheta>0 }\sup _{m_0<0, m_0\in  \mathbb{Z}} 2^{-m_0\Gamma u(1+\vartheta)} \epsilon  ^\theta \sum_{k=0}^{m_0} 2^{\alpha_{\infty} k u(1+\vartheta)}\left(\sum_{i=k}^{\infty}\left|\lambda_{i}\right|^{u(1+\vartheta)}\left\|b_{i}\right\|_{{p(\cdot)}}^{u(1+\vartheta) / 2}\right) \left(\sum_{i=k}^{\infty}\left\|b_{i}\right\|_{L^{p}}^{(u(1+\vartheta))^{\prime} / 2}\right)^{u(1+\vartheta) / (u(1+\vartheta))^{\prime}}\\
& \lesssim  \sup_{\vartheta>0 }\sup _{m_0<0, m_0\in  \mathbb{Z}} 2^{-m_0\Gamma u(1+\vartheta)} \epsilon  ^\theta \sum_{k=0}^{m_0} 2^{\alpha_{\infty} k u(1+\vartheta)}\left(\sum_{i=k}^{\infty}\left|\lambda_{i}\right|^{u(1+\vartheta)}\left|B_{i}\right|^{-\alpha_{i} u(1+\vartheta) /(2 n)}\right) \left(\sum_{i=k}^{\infty}\left|B_{i}\right|^{-\alpha_{i} (u(1+\vartheta))^{\prime} /(2 n)}\right)^{u(1+\vartheta) / (u(1+\vartheta))^{\prime}} \\
& \lesssim  \sup_{\vartheta>0 }\sup _{m_0<0, m_0\in  \mathbb{Z}} 2^{-m_0\Gamma u(1+\vartheta)} \epsilon  ^\theta \sum_{k=0}^{m_0} 2^{\alpha_{\infty} k u(1+\vartheta)}\left(\sum_{i=k}^{\infty}\left|\lambda_{i}\right|^{u(1+\vartheta)}\left|B_{i}\right|^{-\alpha_{i} u(1+\vartheta) /(2 n)}\right)\\
& = \sup_{\vartheta>0 }\sup _{m_0<0, m_0\in  \mathbb{Z}} 2^{-m_0\Gamma u(1+\vartheta)} \epsilon  ^\theta \sum_{i=0}^{m_0}\left|\lambda_{i}\right|^{u(1+\vartheta)} \sum_{k=0}^{i} 2^{(k-i) \alpha_{\infty} u(1+\vartheta) / 2}  +\sup_{\vartheta>0 }\sup _{m_0<0, m_0\in  \mathbb{Z}} 2^{-m_0\Gamma u(1+\vartheta)} \epsilon  ^\theta \sum_{i=m_0}^{\infty}\left|\lambda_{i}\right|^{u(1+\vartheta)} \sum_{k=0}^{m_0} 2^{(k-i) \alpha_{\infty} u(1+\vartheta) / 2} \\
& \lesssim  \sup_{\vartheta>0 }\sup _{m_0<0, m_0\in  \mathbb{Z}} 2^{-m_0\Gamma u(1+\vartheta)} \epsilon  ^\theta \sum_{i=0}^{m_0}\left|\lambda_{i}\right|^{u(1+\vartheta)}  +\sup_{\vartheta>0 }\sup _{m_0<0, m_0\in  \mathbb{Z}}  \epsilon  ^\theta \sum_{i=m_0}^{\infty} 2^{(i\Gamma  -m_0\Gamma) u(1+\vartheta)} 2^{-i\Gamma  u(1+\vartheta)} \sum_{i=-\infty}^{i}\left|\lambda_{i}\right|^{u(1+\vartheta)} \sum_{k=0}^{m_0} 2^{(k-i) \alpha_{\infty} u(1+\vartheta) / 2} \\
& \lesssim \Lambda+\Lambda \sup_{\vartheta>0 }\sup _{m_0<0, m_0\in  \mathbb{Z}}  \epsilon  ^\theta \sum_{i=m_0}^{\infty} 2^{(i-m_0) \Gamma u(1+\vartheta)} 2^{(L-j) \alpha_{\infty} u(1+\vartheta) / 2} \\
& \lesssim \Lambda+\Lambda \sup _{m_0<0, m_0\in  \mathbb{Z}}  \epsilon  ^\theta \sum_{i=m_0}^{\infty} 2^{(i-m_0) u(1+\vartheta)\left(\lambda-\alpha_{\infty} / 2\right)} \\
& \lesssim \Lambda .
\end{align*}

When $0<u(1+\vartheta) \leq  1$, by $n \delta_{2} \leq  \alpha(0), \alpha_{\infty}<s+\delta+n \delta_{2}$ we get

\begin{align*}
I I I_{2}= & \sup_{\vartheta>0 }\sup _{m_0<0, m_0\in  \mathbb{Z}} 2^{-m_0\Gamma u(1+\vartheta)} \epsilon  ^\theta \sum_{k=0}^{m_0} 2^{\alpha_{\infty} k u(1+\vartheta)}\left(\sum_{i=-\infty}^{k-1}\left|\lambda_{i}\right|\left\|\left(T b_{i}\right) \chi_{k}\right\|_{{p(\cdot)}}\right)^{u(1+\vartheta)} \\
\lesssim & \sup_{\vartheta>0 }\sup _{m_0<0, m_0\in  \mathbb{Z}} 2^{-m_0\Gamma u(1+\vartheta)} \epsilon  ^\theta  \sum_{k=0}^{m_0} 2^{\alpha_{\infty} k u(1+\vartheta)}\left(\sum_{i=-\infty}^{k-1}\left|\lambda_{i}\right|^{u(1+\vartheta)} 2^{\left[\left(s+\delta+n \delta_{2}\right)(i-k)-i \alpha_{i}\right] u(1+\vartheta)}\right) \\
= & \sup_{\vartheta>0 }\sup _{m_0<0, m_0\in  \mathbb{Z}} 2^{-m_0\Gamma u(1+\vartheta)} \epsilon  ^\theta\\
& \times \sum_{k=0}^{m_0} 2^{\alpha_{\infty} k u(1+\vartheta)}\left(\sum_{i=-\infty}^{-1}\left|\lambda_{i}\right|^{u(1+\vartheta)} 2^{\left[\left(s+\delta+n \delta_{2}\right)(i-k)-i \alpha(0)\right] u(1+\vartheta)}\right) \\
+& \sup_{\vartheta>0 }\sup _{m_0<0, m_0\in  \mathbb{Z}} 2^{-m_0\Gamma u(1+\vartheta)} \epsilon  ^\theta \\
& \times \sum_{k=0}^{m_0} 2^{\alpha_{\infty} k u(1+\vartheta)}\left(\sum_{i=0}^{k-1}\left|\lambda_{i}\right|^{u(1+\vartheta)} 2^{\left[\left(s+\delta+n \delta_{2}\right)(i-k)-i \alpha_{\infty}\right] u(1+\vartheta)}\right)\\
\lesssim & \sup_{\vartheta>0 }\sup _{m_0<0, m_0\in  \mathbb{Z}} 2^{-m_0\Gamma u(1+\vartheta)} \epsilon  ^\theta \sum_{k=0}^{m_0} 2^{\left[\alpha_{\infty}-\left(s+\delta+n \delta_{2}\right)\right] k u(1+\vartheta)} \\
& \times \sum_{i=-\infty}^{-1}\left|\lambda_{i}\right|^{u(1+\vartheta)} 2^{\left(s+\delta+n \delta_{2}-\alpha(0)\right) iu (1+\vartheta)} \\
+& \sup_{\vartheta>0 }\sup _{m_0<0, m_0\in  \mathbb{Z}} 2^{-m_0\Gamma u(1+\vartheta)} \epsilon  ^\theta \sum_{k=0}^{m_0}\left|\lambda_{i}\right|^{u(1+\vartheta)} \sum_{k=j+1}^{\infty} 2^{(i-k)\left(s+\delta+n \delta_{2}-\alpha_{\infty}\right) u(1+\vartheta)} \\
\lesssim & \sup_{\vartheta>0 }\sup _{m_0<0, m_0\in  \mathbb{Z}} 2^{-m_0\Gamma u(1+\vartheta)} \epsilon  ^\theta \sum_{i=-\infty}^{-1}\left|\lambda_{i}\right|^{u(1+\vartheta)}+\sup_{\vartheta>0 }\sup _{m_0<0, m_0\in  \mathbb{Z}} 2^{-m_0\Gamma u(1+\vartheta)} \epsilon  ^\theta \sum_{i=0}^{m_0-1}\left|\lambda_{i}\right|^{u(1+\vartheta)} \\
\lesssim & \Lambda .
\end{align*}

When $1<u(1+\vartheta)<\infty$ and since $n \delta_{2} \leq  \alpha(0)$, $\alpha_{\infty}<s+\delta+n \delta_{2}$, by H\"older's inequality, we have

\begin{align*}
I I I_{2}=& \sup_{\vartheta>0 }\sup _{m_0<0, m_0\in  \mathbb{Z}} 2^{-m_0\Gamma u(1+\vartheta)} \epsilon  ^\theta \sum_{k=0}^{m_0} 2^{\alpha_{\infty} k u(1+\vartheta)}\left(\sum_{i=-\infty}^{k-1}\left|\lambda_{i}\right|\left\|\left(T b_{i}\right) \chi_{k}\right\|_{{p(\cdot)}}\right)^{u(1+\vartheta)} \\
 \lesssim & \sup_{\vartheta>0 }\sup _{m_0<0, m_0\in  \mathbb{Z}} 2^{-m_0\Gamma u(1+\vartheta)} \epsilon  ^\theta \sum_{k=0}^{m_0} 2^{\alpha_{\infty} k u(1+\vartheta)}\left(\sum_{i=-\infty}^{k-1}\left|\lambda_{i}\right| 2^{\left[\left(s+\delta+n \delta_{2}\right)(i-k)-i \alpha_{i}\right]}\right)^{u(1+\vartheta)} \\
\lesssim & \sup_{\vartheta>0 }\sup _{m_0<0, m_0\in  \mathbb{Z}} 2^{-m_0\Gamma u(1+\vartheta)} \epsilon  ^\theta \sum_{k=0}^{m_0} 2^{\alpha_{\infty} k u(1+\vartheta)}\left(\sum_{i=-\infty}^{-1}\left|\lambda_{i}\right| 2^{\left[\left(s+\delta+n \delta_{2}\right)(i-k)-i \alpha(0)\right]}\right)^{u(1+\vartheta)} \\
 +&\sup_{\vartheta>0 }\sup _{m_0<0, m_0\in  \mathbb{Z}} 2^{-m_0\Gamma u(1+\vartheta)} \epsilon  ^\theta \sum_{k=0}^{m_0} 2^{\alpha_{\infty} k u(1+\vartheta)}\left(\sum_{i=0}^{k-1}\left|\lambda_{i}\right| 2^{\left[\left(s+\delta+n \delta_{2}\right)(i-k)-i \alpha_{\infty}\right]}\right)^{u(1+\vartheta)} \\
\lesssim & \sup_{\vartheta>0 }\sup _{m_0<0, m_0\in  \mathbb{Z}} 2^{-m_0\Gamma u(1+\vartheta)} \epsilon  ^\theta \sum_{k=0}^{m_0} 2^{\left[\alpha_{\infty}-\left(s+\delta+n \delta_{2}\right)\right] k u(1+\vartheta)} \\
 &\times \left(\sum_{i=-\infty}^{-1}\left|\lambda_{i}\right| 2^{\left(s+\delta+n \delta_{2}-\alpha(0)\right) j}\right)^{u(1+\vartheta)} \\
+& \sup_{\vartheta>0 }\sup _{m_0<0, m_0\in  \mathbb{Z}} 2^{-m_0\Gamma u(1+\vartheta)} \epsilon  ^\theta \sum_{k=0}^{m_0}\left(\sum_{i=0}^{k-1}\left|\lambda_{i}\right| 2^{(i-k)\left(s+\delta+n \delta_{2}-\alpha_{\infty}\right)}\right)^{u(1+\vartheta)} \\
\lesssim & \left(\sup_{\vartheta>0 }\sup _{m_0<0, m_0\in  \mathbb{Z}} 2^{-m_0\Gamma u(1+\vartheta)} \epsilon  ^\theta\sum_{i=-\infty}^{-1}\left|\lambda_{i}\right|^{u(1+\vartheta)} 2^{\left(s+\delta+n \delta_{2}-\alpha(0)\right) iu (1+\vartheta) / 2}\right) \\
 &\times\left(\sum_{i=-\infty}^{-1} 2^{\left(s+\delta+n \delta_{2}-\alpha(0)\right) i( u(1+\vartheta))^{\prime} / 2}\right)^{u(1+\vartheta) / (u(1+\vartheta))^{\prime}} \\
\lesssim & \sup_{\vartheta>0 }\sup _{m_0<0, m_0\in  \mathbb{Z}} 2^{-m_0\Gamma u(1+\vartheta)} \epsilon  ^\theta \sum_{k=0}^{m_0}\left(\sum_{i=0}^{k-1}\left|\lambda_{i}\right|^{u(1+\vartheta)} 2^{(i-k)\left(s+\delta+n \delta_{2}-\alpha_{\infty}\right) u(1+\vartheta) / 2}\right) \\
& \times\left(\sum_{i=0}^{k-1} 2^{(i-k)\left(s+\delta+n \delta_{2}-\alpha_{\infty}\right) (u(1+\vartheta))^{\prime} / 2}\right)^{u(1+\vartheta) / (u(1+\vartheta))^{\prime}}\\
\lesssim & \sup_{\vartheta>0 }\sup _{m_0<0, m_0\in  \mathbb{Z}} 2^{-m_0\Gamma u(1+\vartheta)} \epsilon  ^\theta \sum_{i=-\infty}^{-1}\left|\lambda_{i}\right|^{u(1+\vartheta)} 2^{\left(s+\delta+n \delta_{2}-\alpha(0)\right) iu (1+\vartheta) / 2} \\
& +\sup_{\vartheta>0 }\sup _{m_0<0, m_0\in  \mathbb{Z}} 2^{-m_0\Gamma u(1+\vartheta)} \epsilon  ^\theta \sum_{k=0}^{m_0} \sum_{i=0}^{k-1}\left|\lambda_{i}\right|^{u(1+\vartheta)} 2^{(i-k)\left(s+\delta+n \delta_{2}-\alpha_{\infty}\right) u(1+\vartheta) / 2} \\
\lesssim & \sup_{\vartheta>0 }\sup _{m_0<0, m_0\in  \mathbb{Z}} 2^{-m_0\Gamma u(1+\vartheta)} \epsilon  ^\theta \sum_{i=-\infty}^{-1}\left|\lambda_{i}\right|^{u(1+\vartheta)} \\
+& \sup_{\vartheta>0 }\sup _{m_0<0, m_0\in  \mathbb{Z}} 2^{-m_0\Gamma u(1+\vartheta)} \epsilon  ^\theta \sum_{i=0}^{m_0-1}\left|\lambda_{i}\right|^{q} \sum_{k=j+1}^{\ell } 2^{(i-k)\left(s+\delta+n \delta_{2}-\alpha_{\infty}\right) u(1+\vartheta) / 2} \\
\lesssim & \sup_{\vartheta>0 }\sup _{m_0<0, m_0\in  \mathbb{Z}} 2^{-m_0\Gamma u(1+\vartheta)} \epsilon  ^\theta \sum_{i=-\infty}^{-1}\left|\lambda_{i}\right|^{u(1+\vartheta)}+\sup_{\vartheta>0 }\sup _{m_0<0, m_0\in  \mathbb{Z}} 2^{-m_0\Gamma u(1+\vartheta)} \epsilon  ^\theta \sum_{i=0}^{m_0-1}\left|\lambda_{i}\right|^{u(1+\vartheta)}\\
\lesssim & \Lambda .
\end{align*}

Thus we have obtained our desired results. 
\end{proof}

\section{BMO estimate for the grand Herz-Morrey  Hardy spaces}
Let $\mathrm{S}^{n-1}$ denote the unit sphere in $\mathbb{R}^{n}(n \geq 2)$ equipped with normalized Lebesgue measure, and let $\Omega \in L^{s}\left(\mathrm{~S}^{n-1}\right), s \geq 1$, be a homogeneous function of degree zero. For $0<\sigma<n$, the homogeneous fractional integral operator $T_{\Omega, \sigma}$ is defined by

$$
T_{\Omega, \sigma} g(x)=\int_{\mathbb{R}^{n}} \frac{\Omega(x-y)}{|x-y|^{n-\sigma}} g(y) d y.
$$

Let $b$ be a locally integrable function, the commutator of a homogeneous fractional integral operator $T_{\Omega, \sigma}$, generated by $b$ and denoted by $\left[b, T_{\Omega, \sigma}\right]$, is defined by

$$
\left[b, T_{\Omega, \sigma}\right] g(x)=\int_{\mathbb{R}^{n}} \frac{\Omega(x-y)}{|x-y|^{n-\sigma}}(b(x)-b(y)) g(y) d y
$$

In this paper we study the boundedness of the commutator $\left[b, T_{\Omega, \sigma}\right]$ for a homogeneous fractional integral operator $T_{\Omega, \sigma}$ on the grand Herz-Morrey  Hardy spaces.
 
Let us first recall that the space $\operatorname{BMO}\left(\mathbb{R}^{n}\right)$ consists of all locally integrable functions $f$ such that

$$
\left\| g\right\|_{*}=\sup _{Q} \frac{1}{|Q|} \int_{Q}\left|g(x)-g_{Q}\right| d x<\infty
$$
where $f_{Q}=|Q|^{-1} \int_{Q} g(y) d y$, the supremum is taken over all cubes $Q \subset \mathbb{R}^{n}$ with sides parallel to the coordinate axes, and $|Q|$ denotes the Lebesgue measure of $Q$.

A nonnegative locally integrable function $\omega(x)$ on $\mathbb{R}^{n}$ is said to belong to the class $A(p, q)(1<p, q<\infty)$, if there is a constant $C>0$ such that

$$
\sup _{Q}\left(\frac{1}{|Q|} \int_{Q} \omega(x)^{q} d x\right)^{1 / q}\left(\frac{1}{|Q|} \int_{Q} \omega(x)^{-p^{\prime}} d x\right)^{1 / p^{\prime}} \leq C<\infty
$$

where $p^{\prime}=p /(p-1)$. Also, we will say that $\omega \in A_{1}$ if $M \omega(x) \leq C \omega(x)$ for a.e. $x$.

\begin{lem}\cite{ding}
     Suppose that $0<\sigma<n, s^{\prime}<p<n / \sigma, 1 / q=1 / p-\sigma / n, \Omega \in L^{s}\left(\mathrm{~S}^{n-1}\right)$, and $\omega^{s^{\prime}} \in$ $A\left(p / s^{\prime}, q / s^{\prime}\right)$. Then for $b \in \operatorname{BMO}\left(\mathbb{R}^{n}\right)$ and $m \in \mathbb{Z}$ there is a constant $C$, independent of $f$, such that

$$
\left(\int_{\mathbb{R}^{n}}\left|\left[b, T_{\Omega, \sigma}\right]^{m} g(x) \omega(x)\right|^{q} d x\right)^{1 / q} \leq C\left(\int_{\mathbb{R}^{n}}|g(x) \omega(x)|^{p} d x\right)^{1 / p}.
$$
\end{lem}

\begin{lem} \cite{uribe}. Given a family $\mathcal{F}$ and an open set $E \subset \mathbb{R}^{n}$. Assume that for some $p_{0}$ and $q_{0}$, $0<p_{0} \leq q_{0}<\infty$, and every weight $\omega \in A_{1}$,

$$
\left(\int_{E} g(x)^{q_{0}} \omega(x) d x\right)^{1 / q_{0}} \leq C_{0}\left(\int_{E} g(x)^{p_{0}} \omega(x)^{p_{0} / q_{0}} d x\right)^{1 / p_{0}}, \quad(f, g) \in \mathcal{F}.
$$
\end{lem}

Using Lemmas 2.1 and 2.2, and arguments similar to those applied in the proof of Theorem 1.9 of [15], we can prove the $\left(L^{p(\cdot)}\left(\mathbb{R}^{n}\right), L^{q(\cdot)}\left(\mathbb{R}^{n}\right)\right)$-boundedness of the commutator $\left[b, T_{\Omega, \sigma}\right]$.

Next we  recall the definition of the $L^{s}$-Dini condition. We say that a function $\Omega$ satisfies the $L^{s}$-Dini condition if $\Omega \in L^{s}\left(\mathrm{~S}^{n-1}\right)$ with $s \geq 1$ is homogeneous of degree zero on $\mathbb{R}^{n}$, and $\int_{0}^{1} \omega_{s}(\delta) / \delta d \delta<\infty$, where $\omega_{s}(\delta)$ denotes the integral modulus of continuity of order $s$ of $\Omega$ defined by

$$
\omega_{s}(\delta)=\sup _{|\rho|<\delta}\left(\int_{\mathrm{S}^{n-1}}\left|\Omega\left(\rho x^{\prime}\right)-\Omega\left(\rho x^{\prime}\right)\right|^{s} d x^{\prime}\right)^{1 / s}
$$

and $\rho$ is a rotation in $\mathbb{R}^{n}$ with $|\rho|=\|\rho-I\|$.

\begin{lem}\cite{tan} Let $p(\cdot) \in \mathcal{B}\left(\mathbb{R}^{n}\right)$, $k$ be a positive integer and $B$ be a ball in $\mathbb{R}^{n}$. Then for all $b \in \operatorname{BMO}\left(\mathbb{R}^{n}\right)$ and all $j, i \in \mathbb{Z}$ with $j>i$, we have

$$
\begin{gathered}
\frac{1}{C}\|b\|_{*}^{t_0} \leq \sup _{B} \frac{1}{\left\|\chi_{B}\right\|_{L^{p \cdot \cdot}\left(\mathbb{R}^{n}\right)}}\left\|\left(b-b_{B}\right)^{t_0} \chi_{B}\right\|_{L^{p(\cdot)}\left(\mathbb{R}^{n}\right)} \leq C\|b\|_{*}^{t_0} \\
\left\|\left(b-b_{B_{i}}\right)^{t_0} \chi_{B_{j}}\right\|_{L^{p(\cdot)}\left(\mathbb{R}^{n}\right)} \leq C(j-i)^{t_0}\|b\|_{*}^{t_0}\left\|\chi_{B_{j}}\right\|_{L^{p(\cdot)}\left(\mathbb{R}^{n}\right)}
\end{gathered}
$$

where $B_{i}=\left\{x \in \mathbb{R}^{n}:|x| \leq 2^{i}\right\}$ and $B_{j}=\left\{x \in \mathbb{R}^{n}:|x| \leq 2^{j}\right\}$.
\end{lem}

\begin{lem} \cite{uribe1} Given a set $E$ and a function $p(\cdot) \in \mathcal{P}(E)$, let $f: E \times E \rightarrow \mathbb{R}$ be a measurable function (with respect to product measure), such that $g(\cdot, y) \in L^{p(\cdot)}(E)$ for almost every $y \in E$. Then the following inequality holds:

$$
\left\|\int_{E} g(\cdot, y) d y\right\|_{L^{p(\cdot)}(E)} \leq C \int_{E}\|g(\cdot, y)\|_{L^{p(\cdot)}(E)} d y
$$
\end{lem}

\begin{lem}\cite{uribe1} Let $\tilde{q}(\cdot)$ be a variable exponent defined by $\frac{1}{p(x)}=\frac{1}{\tilde{q}(x)}+\frac{1}{q} \quad\left(x \in \mathbb{R}^{n}\right)$. Then for all measurable functions $f$ and $g$ we have

$$
\|f g\|_{L^{p(\cdot)}} \leq C\left\| g\right\|_{L^{\bar{q}(\cdot)}}\|f g\|_{L^{q}}
$$
\end{lem}

\begin{lem}\cite{ref6} Let $p(\cdot) \in \mathcal{P}\left(\mathbb{R}^{n}\right)$ satisfy the conditions (1.1) and (1.2) of Lemma 1.1. Then for every cube (or ball) $Q \subset \mathbb{R}^{n}$ we have

$$
\left\|\chi_{Q}\right\|_{L^{p(\cdot)}} \approx \begin{cases}|Q|^{\frac{1}{p(x)}} & \text { if }|Q| \leq 2^{n} \text { and } x \in Q \\ |Q|^{\frac{1}{p(\infty)}} & \text { if }|Q| \geq 1\end{cases}
$$

where $p(\infty)=\lim _{x \rightarrow \infty} p(x)$.
\end{lem}
\begin{lem} \cite{ding1} Suppose that $0<\sigma<n, s>1$, and $\Omega$ satisfies the $L^{s}$-Dini condition. If there is a constant $a_{0}>0$ such that $|y|<a_{0} R$, then the following inequality holds:

$$
\left(\int_{R<|x|<2 R}\left|\frac{\Omega(x-y)}{|x-y|^{n-\sigma}}-\frac{\Omega(x)}{|x|^{n-\sigma}}\right|^{s} d x\right)^{1 / s} \leq C R^{\frac{n}{s}-(n-\sigma)}\left\{\frac{|y|}{R}+\int_{|y| / 2 R<\delta<|y| / R} \frac{\omega_{s}(\delta)}{\delta} d \delta\right\}
$$
\end{lem}

\begin{theorem}

Suppose that  $0 \leq \Gamma<\infty$, $b \in \operatorname{BMO}\left(\mathbb{R}^{n}\right), 0<\beta \leq 1,0<\sigma<n-\beta$,  $\alpha(\cdot) \in L^{\infty}\left(\mathbb{R}^{n}\right)$ be log-H\"older continuous both at the origin and infinity, and $p_{1}(\cdot) \in \mathcal{P}\left(\mathbb{R}^{n}\right)$ satisfies the conditions \eqref{r5} and \eqref{r6} with $p_{1}^{+}<n / \sigma$ and $1 / p_{1}(x)-1 / p_{2}(x)=\sigma / n$.
Let $\Omega \in L^{s}\left(\mathrm{~S}^{n-1}\right)\left(s>p_{2}^{+}\right)$with $1 \leq s^{\prime}<p_{1}^{-}$ and satisfy $\int_{0}^{1} \frac{\omega_{s}(\delta)}{\delta^{1+\beta}} d \delta<\infty$.
Let $0<u_{1} \leq u_{2}<\infty$ and  $\alpha(\cdot)\geq 2\lambda$, with $n \delta_{2} \leq  \alpha(0), \alpha_{\infty}<n \delta_{2} +\beta $, then $\left[b, T_{\Omega, \sigma}\right]$ is bounded from $H M\dot{K} ^{\alpha(\cdot),u_1),\theta}_{ \Gamma, p_1(\cdot)}(\mathbb{R}^n)$ to $ M\dot{K} ^{\alpha(\cdot),u_2),\theta}_{ \Gamma, p_2(\cdot)}(\mathbb{R}^n)$.  
\end{theorem}

\begin{proof}
  
 Suppose that  $g\in   H {M \dot{K} ^{\alpha(\cdot),u_1),\theta}_{ \Gamma, p_1(\cdot)}(\mathbb{R}^n)}$. By using Theorem 2.3, $f=\sum_{j=-\infty}^{\infty} \lambda_{j} b_{j}$ converges in $\mathcal{S}^{\prime}\left(\mathbb{R}^{n}\right)$, where each $b_{j}$ is a central $(\alpha(\cdot), p(\cdot))$-atom with support contained in $B_{j}$ and

\begin{equation*}
\left\| g\right\|_ {H {M \dot{K} ^{\alpha(\cdot),u_1),\theta}_{ \Gamma, p_1(\cdot)}(\mathbb{R}^n)}} \approx \inf \left(\sup_{\psi>0}\sup _{L \in \mathbb{Z}} 2^{-L \Gamma}\left(\psi^\theta\sum_{j=-\infty}^{L}\left|\lambda_{j}\right|^{u_1(1+\psi)}\right)^{\frac{1}{u_1(1+\psi)}}\right) . 
\end{equation*}

For simplicity, we denote $\Lambda=\sup_{\psi>0}\sup _{L \in \mathbb{Z}} 2^{-L \Gamma u_1(1+\psi)}\psi^\theta \sum_{j=-\infty}^{L}\left|\lambda_{j}\right|^{u_1(1+\psi)}$. By Proposition $1.6$, we have

\begin{align*}
&\left \|\left[b, T_{\Omega, \sigma}\right] f\right\|_{M \dot{K} ^{\alpha(\cdot),u_2),\theta}_{ \Gamma, p_2(\cdot)}(\mathbb{R}^n)}^{u_1(1+\psi)} \\
& \approx \max \left\{\sup_{\psi>0}\sup _{L <  0, L \in \mathbb{Z}} 2^{-L \Gamma u_1(1+\psi)} \psi^\theta \left(\sum_{t_0=-\infty}^{L} 2^{t_0\alpha (0) u_1(1+\psi)}\left\|\left(\left[b, T_{\Omega, \sigma}\right] f\right) \chi_{t_0}\right\|_{L^{p_2(\cdot)}}^{u_1(1+\psi)}\right),\right. \\
& \sup_{\psi>0}\sup _{L \geq  0, L \in \mathbb{Z}} 2^{-L \Gamma u_1(1+\psi)} \psi^\theta \left(\sum_{t_0=-\infty}^{-1} 2^{t_0\alpha (0) u_1(1+\psi)}\left\|(\left[b, T_{\Omega, \sigma}\right] f) \chi_{t_0}\right\|_{L^{p_2(\cdot)}}^{u_1(1+\psi)}\left.\quad+\sum_{t_0=0}^{L} 2^{t_0 \alpha_\infty u_1(1+\psi)}\left\|\left(\left[b, T_{\Omega, \sigma}\right] f\right) \chi_{t_0}\right\|_{L^{p_2(\cdot)}}^{u_1(1+\psi)}\right)\right\} \\
& \lesssim \max \{I, I I+I I I\}. 
\end{align*}

We will find the estimated for $I$ and $III$ and estimate of $II$ can be obtained similarly.

To complete our proof, we only need show that there exists a positive constant $C$ such that $I, I I, I I I \leq  C \Lambda$.

Firstly, we will find the  estimate of $I$ :

\begin{align*}
    I=&\sup_{\psi>0}\sup _{L <  0, L \in \mathbb{Z}} 2^{-L \Gamma u_1(1+\psi)} \psi^\theta \left(\sum_{t_0=-\infty}^{L} 2^{t_0\alpha (0) u_1(1+\psi)}\left\|\left( \left[b, T_{\Omega, \sigma}\right] f\right) \chi_{t_0}\right\|_{L^{p_2(\cdot)}}^{u_1(1+\psi)}\right)\\
    \leq & \sup_{\psi>0}\sup _{L <  0, L \in \mathbb{Z}} 2^{-L \Gamma u_1(1+\psi)} \psi^\theta \sum_{t_0=-\infty}^{L} 2^{t_0\alpha (0) u_1(1+\psi)}\left(\sum\limits_{j=-\infty}^{t_0-1} \left| \lambda_j\right|\left\|\left( \left[b, T_{\Omega, \sigma}\right](a_j) \right) \chi_{t_0}\right\|_{L^{p_2(\cdot)}}\right)^{u_1(1+\psi)}\\
    + & \sup_{\psi>0}\sup _{L <  0, L \in \mathbb{Z}} 2^{-L \Gamma u_1(1+\psi)} \psi^\theta \sum_{t_0=-\infty}^{L} 2^{t_0\alpha (0) u_1(1+\psi)}\left(\sum\limits_{j=k}^{\infty} \left| \lambda_j\right|\left\|\left( \left[b, T_{\Omega, \sigma}\right](a_j) \right) \chi_{t_0}\right\|_{L^{p_2(\cdot)}}\right)^{u_1(1+\psi)}\\
    :=& I_1+I_2. 
\end{align*}

We first estimate $I_{1}$. For each $k \in \mathbb{Z}, j \leq t_0-1$ and a.e. $x \in F_{t_0}$, using Lemma 2.4, the Minkowski's inequality and the vanishing moments of $a_{j}$, we get

$$
\begin{aligned}
\left\|\left[b, T_{\Omega, \sigma}\right]\left(a_{j}\right) \chi_{t_0}\right\|_{L^{p_{2}(\cdot)}} & \lesssim\int_{B_{j}}\left\|\left|\frac{\Omega(\cdot-y)}{|\cdot-y|^{n-\sigma}}-\frac{\Omega(\cdot)}{|\cdot|^ {n- \sigma}}\right|(b(\cdot)-b(y)) \chi_{t_0}(\cdot)\right\|_{L^{p_{2}(\cdot)}}\left|a_{j}(y)\right| d y \\
& \lesssim \int_{B_{j}}\left\|\left|\frac{\Omega(\cdot-y)}{|\cdot-y|^{n-\sigma}}-\frac{\Omega(\cdot)}{|\cdot|^{n-\sigma}}\right|\left|b(\cdot)-b_{B_{j}}\right| \chi_{t_0}(\cdot)\right\|_{L^{p_{2}(\cdot)}}\left|a_{j}(y)\right| d y \\
& \lesssim \int_{B_{j}}\left\|\left|\frac{\Omega(\cdot-y)}{|\cdot-y|^{n-\sigma}}-\frac{\Omega(\cdot)}{|\cdot|^{n-\sigma}}\right| \chi_{t_0}(\cdot)\right\|_{L^{p_{2}(\cdot)}}\left|b_{B_{j}}-b(y)\right|\left|a_{j}(y)\right| d y \\
& =: I_{11}+I_{12}. 
\end{aligned}
$$

To estimate $I_{11}$, we note that $s>p_{2}^{+}$, and denote $\tilde{p}_{2}(\cdot)>1$ and $\frac{1}{p_{2}(x)}=\frac{1}{\tilde{p}_{2}(x)}+\frac{1}{s}$. Then by Lemmas $3.3$ and $3.5$ we have

$$
\begin{aligned}
& \left\|\left|\frac{\Omega(\cdot-y)}{\left| \cdot-y\right|^{n-\sigma}}-\frac{\Omega(\cdot)}{|\cdot|^{n-\sigma}}\right|\left|b(\cdot)-b_{B_{j}}\right| \chi_{t_0}(\cdot)\right\|_{L^{p_{2}(\cdot)}} \\
& \lesssim\left\|\left|\frac{\Omega(\cdot-y)}{\left| \cdot-y\right|^{n-\sigma}}-\frac{\Omega(\cdot)}{|\cdot| ^{n-\sigma}}\right| \chi_{t_0}(\cdot)\right\|_{L^{s}}\left\|\left|b(\cdot)-b_{B_{j}}\right| \chi_{t_0}(\cdot) \right\|_{L^{\tilde{p}_{2}(\cdot)}} \\
& \lesssim \left\|\left|\frac{\Omega(\cdot-y)}{\left| \cdot-y\right|^{n-\sigma}}-\frac{\Omega(\cdot)}{|\cdot| ^{n-\sigma}}\right| \chi_{t_0}(\cdot)\right\|_{L^{s}}(t_0-j)\left\| b\right\|_{*}\left\| \chi_{B_{t_0}} \right\|_{L^{\tilde{p}_{2}(\cdot)}} .
\end{aligned}
$$

When $\left|B_{t_0}\right| \lesssim 2^{n}$ and $x_{t_0} \in B_{t_0}$, by Lemma $3.6$ we have

$$
\left\|\chi_{B_{t_0}}\right\|_{L^{\tilde{p}_{2}(\cdot)}} \approx\left|B_{t_0}\right|^{\frac{1}{p_{2}\left(x_{t_0}\right)}} \approx\left\|\chi_{B_{t_0}}\right\|_{L^{p_{1}(\cdot)}}\left|B_{t_0}\right|^{-\frac{1}{s}-\frac{\sigma}{n}}.
$$

When $\left|B_{t_0}\right| \geq 1$ we have

$$
\left\|\chi_{B_{t_0}}\right\|_{L^{\tilde{p}_{2}(\cdot)}} \approx\left|B_{t_0} \right|^{\frac{1}{p_{2}(\infty)}} \approx\left\|\chi_{B_{t_0}}\right\|_{L^{p_{1}(\cdot)}}\left| B_{t_0}\right|^{-\frac{1}{s}-\frac{\sigma}{n}}.
$$

So, we obtain $\left\|\chi_{B_{t_0}}\right\|_{L^{\tilde{p}_{2}(\cdot)}} \approx\left\|\chi_{B_{t_0}}\right\|_{L^{p_{1}(\cdot)}}\left|B_{t_0}\right|^{-\frac{1}{s}-\frac{\sigma}{n}}$.

Meanwhile, by Lemma $3.7$ we have

$$
\begin{aligned}
&\left\|\left|\frac{\Omega(\cdot-y)}{\left| \cdot-y\right|^{n-\sigma}}-\frac{\Omega(\cdot)}{|\cdot| ^{n-\sigma}}\right| \chi_{t_0}(\cdot)\right\|_{L^{s}} \lesssim 2^{(t_0-1)\left(\frac{n}{s}-(n-\sigma)\right)}\left\{\frac{|y|}{2^{t_0}}+\int_{|y| / 2^{t_0}}^{|y| 2^{t_0-1}} \frac{\omega_{s}(\delta)}{\delta} d \delta\right\} \\
& \lesssim 2^{(t_0-1)\left(\frac{n}{s}-(n-\sigma)\right)}\left(2^{j-t_0+1}+2^{(j-t_0+1) \beta} \int_{0}^{1} \frac{\omega_{s}(\delta)}{\delta} d \delta\right) \lesssim  2^{(t_0-1)\left(\frac{n}{s}-(n-\sigma)\right)} 2^{(j-t_0) \beta}.
\end{aligned}
$$

So, using the generalized Hölder inequality, we obtain the following estimate for $I_{11}$ :

\begin{align*}
I_{11}=&\int_{B_{j}}\left\|\left|\frac{\Omega(\cdot-y)}{|\cdot-y|^{n-\sigma}}-\frac{\Omega(\cdot)}{|\cdot|^{n-\sigma}}\right|\left|b(\cdot)-b_{B_{j}}\right| \chi_{t_0}(\cdot)\right\|_{L^{p_{2}(\cdot)}}\left|a_{j}(y)\right| d y \\
\lesssim & (t_0-j)\|b\|_{*} 2^{(t_0-1)\left(\frac{n}{s}-(n-\sigma)\right)} 2^{(j-t_0) \beta}\left\|\chi_{B_{t_0}}\right\|_{L^{p_{1}(\cdot)}}\left|B_{t_0}\right|^{-\frac{1}{s}-\frac{\sigma}{n}} \int_{B_{j}}\left|a_{j}(y)\right| d y  \\
\lesssim & (t_0-j)\|b\|_{*} 2^{-t_0n+(j-t_0) \beta}\left\|\chi_{B_{t_0}}\right\|_{L^{p_{1}(\cdot)}}\left\|a_{j}\right\|_{L^{p_{1}(\cdot)}}\left\|\chi_{B_{j}}\right\|_{L^{p_{1}^{\prime}(\cdot)}}.
\end{align*}

To estimate $I_{12}$, we use arguments similar to those applied for $I_{11}$, to obtain

$$
\begin{aligned}
 \left\|\left|\frac{\Omega(\cdot-y)}{|\cdot-y|^{n-\sigma}}-\frac{\Omega(\cdot)}{|\cdot|^{n-\sigma}}\right| \chi_{t_0}(\cdot)\right\|_{L^{p_{2}(\cdot)}} &\lesssim\left\|\left|\frac{\Omega(\cdot-y)}{|\cdot-y|^{n-\sigma}}-\frac{\Omega(\cdot)}{|\cdot|^{n-\sigma}}\right| \chi_{t_0}(\cdot)\right\|_{L^{s}} \left\|\chi_{t_0}(\cdot) \right\|_{L^{\tilde{p}_{2}(\cdot)}} \\
& \lesssim\left\|\left|\frac{\Omega(\cdot-y)}{|\cdot-y|^{n-\sigma}}-\frac{\Omega(\cdot)}{|\cdot|^{n-\sigma}}\right| \chi_{t_0}(\cdot)\right\|_{L^{s}} \left\|\chi_{B_k} \right\|_{L^{\tilde{p}_{2}(\cdot)}} \\
& \lesssim  2^{(t_0-1)\left(\frac{n}{s}-(n-\sigma)\right)} 2^{(j-t_0) \beta}\left\|\chi_{B_{t_0}}\right\|_{L^{\tilde{p}_{2}(\cdot)}} \lesssim  2^{-t_0n+(j-t_0) \beta}\left\|\chi_{B_{t_0}}\right\|_{L^{p_{1}(\cdot)}}.
\end{aligned}
$$

So, by Lemma $3.3$ and the generalized Hölder inequality, we have

\begin{align*}
I_{12}&=  \int_{B_{j}}\left\|\left|\frac{\Omega(\cdot-y)}{|\cdot-y|^{n-\sigma}}-\frac{\Omega(\cdot)}{|\cdot|{ }^{n-\sigma}}\right| \chi_{t_0}(\cdot)\right\|_{L^{p_{2}(\cdot)}}\left|b_{B_{j}}-b(y) \| a_{j}(y)\right| d y \\
& \lesssim  2^{-t_0n+(j-t_0) \beta}\left\|\chi_{B_{t_0}}\right\|_{L^{p_{1}(\cdot)}} \int_{B_{j}}\left|b_{B_{j}}-b(y) \| a_{j}(y)\right| d y \\
&\lesssim   2^{-t_0n+(j-t_0) \beta}\left\|\chi_{B_{t_0}}\right\|_{L^{p_{1}(\cdot)}}\left\|\left(b_{B_{j}}-b\right) \chi_{B_{j}}\right\|_{L^{p_{1}^{\prime}(\cdot)}}\left\|a_{j}\right\|_{L^{p_{1}(\cdot)}} \\
&\lesssim  \|b\|_{*} 2^{-t_0n+(j-t_0) \beta}\left\|\chi_{B_{t_0}}\right\|_{L^{p_{1}(\cdot)}}\left\|a_{j}\right\|_{L^{p_{1}(\cdot)}}\left\|\chi_{B_{j}}\right\|_{L^{p_{1}^{\prime}(\cdot)}}.
\end{align*}

Next, using estimates for $I_{11}$,  and  $I_{12}$, we can write

\begin{align*}
    \left\|\left[b, T_{\Omega, \sigma}\right]\left(a_{j}\right) \chi_{t_0}\right\|_{L^{p_{2}(\cdot)}}\lesssim & (t_0-j)\|b\|_{*} 2^{-t_0n+(j-t_0) \beta}\left\|\chi_{B_{t_0}}\right\|_{L^{p_{1}(\cdot)}}\left\|a_{j}\right\|_{L^{p_{1}(\cdot)}}\left\|\chi_{B_{j}}\right\|_{L^{p_{1}^{\prime}(\cdot)}}\\
    \lesssim & (t_0-j)\|b\|_{*}2^{(j-t_0) \beta}\left\|a_{j}\right\|_{L^{p_{1}(\cdot)}}\frac{\left\|\chi_{B_{j}}\right\|_{L^{p_{1}^{\prime}(\cdot)}}}{\left\|\chi_{B_{t_0}}\right\|_{L^{p_{1}^{\prime}(\cdot)}}}\\
    \lesssim &  (t_0-j)\|b\|_{*} 2^{-j \alpha_j +(j-t_0)(\beta+n\delta_2)}.
\end{align*}

So we have

\begin{align*}
I_{1} \lesssim &\sup_{\psi>0}\sup _{L <  0, L \in \mathbb{Z}} 2^{-L \Gamma u_1(1+\psi)} \psi^\theta \sum_{t_0=-\infty}^{L} 2^{t_0\alpha (0) u_1(1+\psi)}\left(\sum\limits_{j=-\infty}^{t_0-1} \left| \lambda_j\right|\left\|\left( \left[b, T_{\Omega, \sigma}\right](a_j) \right) \chi_{t_0}\right\|_{L^{p_2(\cdot)}}\right)^{u_1(1+\psi)} \\
\lesssim & \sup_{\psi>0}\sup _{L <  0, L \in \mathbb{Z}} 2^{-L \Gamma u_1(1+\psi)} \psi^\theta \sum_{t_0=-\infty}^{L} 2^{t_0\alpha (0) u_1(1+\psi)}\left(\sum\limits_{j=-\infty}^{t_0-1} \left| \lambda_j\right|(t_0-j)\|b\|_{*} 2^{-j \alpha_j +(j-t_0)(\beta+n\delta_2)}\right)^{u_1(1+\psi)} \\
\lesssim &\sup_{\psi>0}\sup _{L <  0, L \in \mathbb{Z}} 2^{-L \Gamma u_1(1+\psi)} \psi^\theta \sum_{t_0=-\infty}^{L} \left(\sum\limits_{j=-\infty}^{t_0-1} \left| \lambda_j\right|(t_0-j)\|b\|_{*} 2^{(j-t_0)(\beta+n\delta_2-\alpha(0))}\right)^{u_1(1+\psi)} .
\end{align*}

When $1<u_{1}<\infty$, take $1 / u_{1}+1 / u_{1}^{\prime}=1$. Since $\beta+n \delta_{2}-\alpha(0)>0$, we can use the Hölder inequality to obtain

\begin{align*}
I_{1} \lesssim& \|b\|_{*}^{u_1(1+\psi)} \sup_{\psi>0}\sup _{L <  0, L \in \mathbb{Z}} 2^{-L \Gamma u_1(1+\psi)} \psi^\theta \sum_{t_0=-\infty}^{L} \left(\sum\limits_{j=-\infty}^{t_0-1} \left| \lambda_j\right|^{u_1(1+\psi)}2^{(j-t_0)(\beta+n\delta_2-\alpha(0)){u_1(1+\psi)/2}}\right)\\
\times &\left(\sum\limits_{j=-\infty}^{t_0-1} (t_0-j)^{(u_1(1+\psi))^\prime}2^{(j-t_0)(\beta+n\delta_2-\alpha(0)){(u_1(1+\psi))^\prime/2}}\right)^{\frac{u_1(1+\psi)}{(u_1(1+\psi))^\prime}}\\
\lesssim& \|b\|_{*} ^{u_1(1+\psi)}\sup_{\psi>0}\sup _{L <  0, L \in \mathbb{Z}} 2^{-L \Gamma u_1(1+\psi)} \psi^\theta \sum_{t_0=-\infty}^{L} \left(\sum\limits_{j=-\infty}^{t_0-1} \left| \lambda_j\right|^{u_1(1+\psi)}2^{(j-t_0)(\beta+n\delta_2-\alpha(0)){u_1(1+\psi)/2}}\right)\\
\lesssim& \|b\|_{*} ^{u_1(1+\psi)}\sup_{\psi>0}\sup _{L <  0, L \in \mathbb{Z}} 2^{-L \Gamma u_1(1+\psi)} \psi^\theta \sum_{j=-\infty}^{L}\left| \lambda_j\right|^{u_1(1+\psi)} \left(\sum\limits_{t_0=j+1}^{-1} 2^{(j-t_0)(\beta+n\delta_2-\alpha(0)){u_1(1+\psi)/2}}\right)\\
\lesssim& \|b\|_{*} ^{u_1(1+\psi)}\Lambda. 
\end{align*}

When $0<u_{1} \leq 1$, we have

\begin{align*}
I_{1} = & \sup_{\psi>0}\sup _{L <  0, L \in \mathbb{Z}} 2^{-L \Gamma u_1(1+\psi)} \psi^\theta \sum_{t_0=-\infty}^{L} \left(\sum\limits_{j=-\infty}^{t_0-1} \left| \lambda_j\right|(t_0-j)\|b\|_{*} 2^{(j-t_0)(\beta+n\delta_2-\alpha(0))}\right)^{u_1(1+\psi)}\\
 \lesssim &\|b\|_{*} ^{u_1(1+\psi)}\sup_{\psi>0}\sup _{L <  0, L \in \mathbb{Z}} 2^{-L \Gamma u_1(1+\psi)} \psi^\theta \sum_{t_0=-\infty}^{L} \left(\sum\limits_{j=-\infty}^{t_0-1} \left| \lambda_j\right|^{u_1(1+\psi)}(t_0-j)^{u_1(1+\psi)} 2^{(j-t_0)(\beta+n\delta_2-\alpha(0)){u_1(1+\psi)}}\right)\\
 \lesssim &\|b\|_{*} ^{u_1(1+\psi)}\sup_{\psi>0}\sup _{L <  0, L \in \mathbb{Z}} 2^{-L \Gamma u_1(1+\psi)} \psi^\theta \sum_{j=-\infty}^{L} \left| \lambda_j\right|^{u_1(1+\psi)} \left(\sum\limits_{t_0=j+1}^{-1} (t_0-j)^{u_1(1+\psi)} 2^{(j-t_0)(\beta+n\delta_2-\alpha(0)){u_1(1+\psi)}}\right)\\
 \lesssim &\|b\|_{*} ^{u_1(1+\psi)}\sup_{\psi>0}\sup _{L <  0, L \in \mathbb{Z}} 2^{-L \Gamma u_1(1+\psi)} \psi^\theta \sum_{j=-\infty}^{L} \left| \lambda_j\right|^{u_1(1+\psi)} \\
 \lesssim &  \|b\|_{*} ^{u_1(1+\psi)} \Lambda. 
\end{align*}

Next we will find the estimate of $I_2$. By using the boundedess of $\left[b, T_{\Omega, \sigma}\right]$ on variable Lebesgue spaces, we have

\begin{align*}
    I_2 = &  \sup_{\psi>0}\sup _{L <  0, L \in \mathbb{Z}} 2^{-L \Gamma u_1(1+\psi)} \psi^\theta \sum_{t_0=-\infty}^{L} 2^{t_0\alpha (0) u_1(1+\psi)}\left(\sum\limits_{j=k}^{\infty} \left| \lambda_j\right|\left\|\left( \left[b, T_{\Omega, \sigma}\right](a_j) \right) \chi_{t_0}\right\|_{L^{p_2(\cdot)}}\right)^{u_1(1+\psi)}\\
    \lesssim &   \sup_{\psi>0}\sup _{L <  0, L \in \mathbb{Z}} 2^{-L \Gamma u_1(1+\psi)} \psi^\theta \sum_{t_0=-\infty}^{L} 2^{t_0\alpha (0) u_1(1+\psi)}\left(\sum\limits_{j=k}^{\infty} \left| \lambda_j\right|\left\|a_j\right\|_{L^{p_2(\cdot)}}\right)^{u_1(1+\psi)}\\
    \lesssim &   \sup_{\psi>0}\sup _{L <  0, L \in \mathbb{Z}} 2^{-L \Gamma u_1(1+\psi)} \psi^\theta \sum_{t_0=-\infty}^{L} 2^{t_0\alpha (0) u_1(1+\psi)}\left(\sum\limits_{j=k}^{\infty} \left| \lambda_j\right|2^{-j\alpha_j}\right)^{u_1(1+\psi)}\\
    \lesssim &   \sup_{\psi>0}\sup _{L <  0, L \in \mathbb{Z}} 2^{-L \Gamma u_1(1+\psi)} \psi^\theta \sum_{t_0=-\infty}^{L} 2^{t_0\alpha (0) u_1(1+\psi)}\left(\sum\limits_{j=k}^{-1} \left| \lambda_j\right|^{u_1(1+\psi)}2^{-j\alpha(0){u_1(1+\psi)}}+\sum\limits_{j=0}^{\infty} \left| \lambda_j\right|^{u_1(1+\psi)}2^{-j\alpha_\infty{u_1(1+\psi)}}\right)\\
    \lesssim &   \sup_{\psi>0}\sup _{L <  0, L \in \mathbb{Z}} 2^{-L \Gamma u_1(1+\psi)} \psi^\theta \sum_{t_0=-\infty}^{L} \sum\limits_{j=k}^{-1} \left| \lambda_j\right|^{u_1(1+\psi)}2^{\alpha(0)(t_0-j){u_1(1+\psi)}}\\
    & + \sup_{\psi>0}\sup _{L <  0, L \in \mathbb{Z}} 2^{-L \Gamma u_1(1+\psi)} \psi^\theta \sum_{t_0=-\infty}^{L} 2^{t_0\alpha (0) u_1(1+\psi)}\sum\limits_{j=0}^{\infty} \left| \lambda_j\right|^{u_1(1+\psi)}2^{-\alpha_\infty j{u_1(1+\psi)}}\\
    \lesssim &   \sup_{\psi>0}\sup _{L <  0, L \in \mathbb{Z}} 2^{-L \Gamma u_1(1+\psi)} \psi^\theta \sum_{j=-\infty}^{-1} \left| \lambda_j\right|^{u_1(1+\psi)}\sum\limits_{t_0=-\infty}^{j} 2^{\alpha(0)(t_0-j){u_1(1+\psi)}}\\
    & + \sup_{\psi>0}\sup _{L <  0, L \in \mathbb{Z}}  \psi^\theta \sum_{j=0}^\infty 2^{-j \Gamma u_1(1+\psi)} \left| \lambda_j\right|^{u_1(1+\psi)} 2^{(\Gamma-\alpha_\infty)j(u_1(1+\psi))}2^{-L  \Gamma u_1(1+\psi)}\sum\limits_{t_0=-\infty}^{L} 2^{\alpha(0)k {u_1(1+\psi)}}\\
    \lesssim &   \sup_{\psi>0}\sup _{L <  0, L \in \mathbb{Z}} 2^{-L \Gamma u_1(1+\psi)} \psi^\theta \sum_{j=-\infty}^{L}  \left| \lambda_j\right|^{u_1(1+\psi)}+\sup_{\psi>0}\sup _{L <  0, L \in \mathbb{Z}} 2^{-L \Gamma u_1(1+\psi)} \psi^\theta \sum_{j=L}^{-1}  \left| \lambda_j\right|^{u_1(1+\psi)}\sum\limits_{t_0=-\infty}^{j} 2^{\alpha(0)(t_0-j){u_1(1+\psi)}}\\
     & + \Lambda  \sup_{\psi>0}\sup _{L <  0, L \in \mathbb{Z}}  \psi^\theta \sum_{j=0}^\infty 2^{(\Gamma-\alpha_\infty)j(u_1(1+\psi))}\sum\limits_{t_0=-\infty}^{L} 2^{(\alpha(0)t_0-l\Gamma) {u_1(1+\psi)}}\\
     \lesssim & \Lambda + \sup_{\psi>0}\sup _{L <  0, L \in \mathbb{Z}}  \psi^\theta \sum_{j=L}^{-1}2^{-j \Gamma u_1(1+\psi)}  \left| \lambda_j\right|^{u_1(1+\psi)}2^{ (j-\Gamma) u_1(1+\psi)}\sum\limits_{t_0=-\infty}^{j} 2^{\alpha(0)(t_0-j){u_1(1+\psi)}}+\Lambda\\
     \lesssim & \Lambda. 
\end{align*}

Now we will find estimate for $III$. 
\begin{align*}
    III=&  \sup_{\psi>0}\sup _{L \geq  0, L \in \mathbb{Z}} 2^{-L \Gamma u_1(1+\psi)} \psi^\theta \sum_{t_0=0}^{L} 2^{t_0 \alpha_\infty u_1(1+\psi)}\left\|\left(\left[b, T_{\Omega, \sigma}\right] f\right) \chi_{t_0}\right\|_{L^{p_2(\cdot)}}^{u_1(1+\psi)}\\
    \lesssim & \sup_{\psi>0}\sup _{L \geq  0, L \in \mathbb{Z}} 2^{-L \Gamma u_1(1+\psi)} \psi^\theta \sum_{t_0=0}^{L} 2^{t_0 \alpha_\infty u_1(1+\psi)}\left(\sum\limits_{j=k}^\infty\left| \lambda_j\right|\left\|\left(\left[b, T_{\Omega, \sigma}\right] a_j\right) \chi_{t_0}\right\|_{L^{p_2(\cdot)}}\right)^{u_1(1+\psi)}\\
    & +  \sup_{\psi>0}\sup _{L \geq  0, L \in \mathbb{Z}} 2^{-L \Gamma u_1(1+\psi)} \psi^\theta \sum_{t_0=0}^{L} 2^{t_0 \alpha_\infty u_1(1+\psi)}\left(\sum\limits_{j=-\infty}^{t_0-1}\left| \lambda_j\right|\left\|\left(\left[b, T_{\Omega, \sigma}\right] a_j\right) \chi_{t_0}\right\|_{L^{p_2(\cdot)}}\right)^{u_1(1+\psi)}\\
    :=& III_1+III_2.
\end{align*}

When $0<u_1\leq 1$, by using the boundedness of $\left[b, T_{\Omega, \sigma}\right]$ in $L^{p_2(\cdot)}$, we have

\begin{align*}
    III_1= & \sup_{\psi>0}\sup _{L \geq  0, L \in \mathbb{Z}} 2^{-L\Gamma  u_1(1+\psi)} \psi^\theta \sum_{t_0=0}^{L} 2^{t_0 \alpha_\infty u_1(1+\psi)}\left(\sum\limits_{j=k}^\infty\left| \lambda_j\right|\left\|\left(\left[b, T_{\Omega, \sigma}\right] a_j\right) \chi_{t_0}\right\|_{L^{p_2(\cdot)}}\right)^{u_1(1+\psi)}\\
    \lesssim & \sup_{\psi>0}\sup _{L \geq  0, L \in \mathbb{Z}} 2^{-L\Gamma  u_1(1+\psi)} \psi^\theta \sum_{t_0=0}^{L} 2^{t_0 \alpha_\infty u_1(1+\psi)}\left(\sum\limits_{j=k}^\infty\left| \lambda_j\right|^{u_1(1+\psi)}\left\|\left(\left[b, T_{\Omega, \sigma}\right] a_j\right) \chi_{t_0}\right\|^{u_1(1+\psi)}_{L^{p_2(\cdot)}}\right)\\
    \lesssim & \sup_{\psi>0}\sup _{L \geq  0, L \in \mathbb{Z}} 2^{-L\Gamma  u_1(1+\psi)} \psi^\theta \sum_{t_0=0}^{L} 2^{t_0 \alpha_\infty u_1(1+\psi)}\sum\limits_{j=k}^\infty\left| \lambda_j\right|^{u_1(1+\psi)}2^{-\alpha_j ju_1(1+\psi)}\\
    \lesssim & \sup_{\psi>0}\sup _{L \geq  0, L \in \mathbb{Z}} 2^{-L\Gamma  u_1(1+\psi)} \psi^\theta \sum_{t_0=0}^{L} 2^{t_0 \alpha_\infty u_1(1+\psi)}\sum\limits_{j=k}^\infty\left| \lambda_j\right|^{u_1(1+\psi)}2^{-\alpha_\infty ju_1(1+\psi)}\\
     \lesssim & \sup_{\psi>0}\sup _{L \geq  0, L \in \mathbb{Z}} 2^{-L\Gamma  u_1(1+\psi)} \psi^\theta \sum\limits_{j=0}^L\left| \lambda_j\right|^{u_1(1+\psi)}\sum\limits_{t_0=0}^j2^{(t_0-j)\alpha_\infty u_1(1+\psi)}+  \sup_{\psi>0}\sup _{L \geq  0, L \in \mathbb{Z}} 2^{-L\Gamma  u_1(1+\psi)} \psi^\theta \sum\limits_{j=L}^\infty\left| \lambda_j\right|^{u_1(1+\psi)}\sum\limits_{t_0=0}^L2^{(t_0-j)\alpha_\infty u_1(1+\psi)}\\
      \lesssim & \sup_{\psi>0}\sup _{L \geq  0, L \in \mathbb{Z}} 2^{-L\Gamma  u_1(1+\psi)} \psi^\theta \sum\limits_{j=0}^L\left| \lambda_j\right|^{u_1(1+\psi)}+  \sup_{\psi>0}\sup _{L \geq  0, L \in \mathbb{Z}}  \psi^\theta \sum\limits_{j=L}^\infty 2^{(j\Gamma-L\Gamma)  u_1(1+\psi)}2^{-j \Gamma u_1(1+\psi)}\sum\limits_{i=-\infty}^j\left| \lambda_i\right|^{u_1(1+\psi)}\sum\limits_{t_0=0}^L2^{(t_0-j)\alpha_\infty u_1(1+\psi)}\\
      \lesssim & \Lambda + \Lambda \sup_{\psi>0}\sup _{L \geq  0, L \in \mathbb{Z}}  \psi^\theta \sum\limits_{j=L}^\infty 2^{(j-L)  u_1(1+\psi)(\Gamma-\alpha_\infty)} \\  
      \lesssim & \Lambda. 
\end{align*}
When $1<u_1\leq \infty$, by using the boundedness of $\left[b, T_{\Omega, \sigma}\right]$ in $L^{p(\cdot)}$, we have

\begin{align*}
I I I_{1}&=  \sup_{\psi>0}\sup _{L <  0, L \in \mathbb{Z}} 2^{-L \Gamma u_1(1+\psi)} \psi^\theta \sum_{t_0=0}^{L} 2^{t_0\alpha _{\infty}u_1(1+\psi)}\left(\sum_{j=k}^{\infty}\left|\lambda_{j}\right|\left\|\left[b, T_{\Omega, \sigma}\right] \chi_{t_0}\right\|_{L^{p_2(\cdot)}}\right)^{u_1(1+\psi)} \\
& \lesssim \sup_{\psi>0}\sup _{L <  0, L \in \mathbb{Z}} 2^{-L \Gamma u_1(1+\psi)} \psi^\theta \sum_{t_0=0}^{L} 2^{\alpha_{\infty} t_0u _1(1+\psi)}\left(\sum_{j=k}^{\infty}\left|\lambda_{j}\right|^{u_1(1+\psi)}\left\|\left[b, T_{\Omega, \sigma}\right] \chi_{t_0}\right\|_{L^{p_2(\cdot)}}^{u_1(1+\psi) / 2}\right) \\
& \times\left(\sum_{j=k}^{\infty}\left\|\left[b, T_{\Omega, \sigma}\right] \chi_{t_0}\right\|_{L^{p_2(\cdot)}}^{(u_1(1+\psi))^{\prime} / 2}\right)^{u_1(1+\psi) / (u_1(1+\psi))^{\prime}} \\
& \lesssim  \sup_{\psi>0}\sup _{L <  0, L \in \mathbb{Z}} 2^{-L \Gamma u_1(1+\psi)} \psi^\theta \sum_{t_0=0}^{L} 2^{\alpha_{\infty} t_0u _1(1+\psi)}\left(\sum_{j=k}^{\infty}\left|\lambda_{j}\right|^{u_1(1+\psi)}\left\|b_{j}\right\|_{L^{p_1(\cdot)}}^{u_1(1+\psi) / 2}\right) \left(\sum_{j=k}^{\infty}\left\|b_{j}\right\|_{L^{p_1(\cdot)}}^{(u_1(1+\psi))^{\prime} / 2}\right)^{u_1(1+\psi) / (u_1(1+\psi))^{\prime}}\\
& \lesssim  \sup_{\psi>0}\sup _{L <  0, L \in \mathbb{Z}} 2^{-L \Gamma u_1(1+\psi)} \psi^\theta \sum_{t_0=0}^{L} 2^{\alpha_{\infty} t_0u _1(1+\psi)}\left(\sum_{j=k}^{\infty}\left|\lambda_{j}\right|^{u_1(1+\psi)}\left|B_{j}\right|^{-\alpha_{j} u_1(1+\psi) /(2 n)}\right) \left(\sum_{j=k}^{\infty}\left|B_{j}\right|^{-\alpha_{j} (u_1(1+\psi))^{\prime} /(2 n)}\right)^{u_1(1+\psi) / (u_1(1+\psi))^{\prime}} \\
& \lesssim  \sup_{\psi>0}\sup _{L <  0, L \in \mathbb{Z}} 2^{-L \Gamma u_1(1+\psi)} \psi^\theta \sum_{t_0=0}^{L} 2^{\alpha_{\infty} t_0u _1(1+\psi)}\left(\sum_{j=k}^{\infty}\left|\lambda_{j}\right|^{u_1(1+\psi)}\left|B_{j}\right|^{-\alpha_{j} u_1(1+\psi) /(2 n)}\right)\\
& = \sup_{\psi>0}\sup _{L <  0, L \in \mathbb{Z}} 2^{-L \Gamma u_1(1+\psi)} \psi^\theta \sum_{j=0}^{L}\left|\lambda_{j}\right|^{u_1(1+\psi)} \sum_{t_0=0}^{j} 2^{(t_0-j) \alpha_{\infty} u_1(1+\psi) / 2}  +\sup_{\psi>0}\sup _{L <  0, L \in \mathbb{Z}} 2^{-L \Gamma u_1(1+\psi)} \psi^\theta \sum_{j=L}^{\infty}\left|\lambda_{j}\right|^{u_1(1+\psi)} \sum_{t_0=0}^{L} 2^{(t_0-j) \alpha_{\infty} u_1(1+\psi) / 2} \\
& \lesssim  \sup_{\psi>0}\sup _{L <  0, L \in \mathbb{Z}} 2^{-L \Gamma u_1(1+\psi)} \psi^\theta \sum_{j=0}^{L}\left|\lambda_{j}\right|^{u_1(1+\psi)}  +\sup_{\psi>0}\sup _{L <  0, L \in \mathbb{Z}}  \psi^\theta \sum_{j=L}^{\infty} 2^{(j \Gamma -L \Gamma) u_1(1+\psi)} 2^{-j \Gamma u_1(1+\psi)} \sum_{i=-\infty}^{j}\left|\lambda_{i}\right|^{u_1(1+\psi)} \sum_{t_0=0}^{L} 2^{(t_0-j) \alpha_{\infty} u_1(1+\psi) / 2} \\
& \lesssim \Lambda+\Lambda \sup_{\psi>0}\sup _{L <  0, L \in \mathbb{Z}}  \psi^\theta \sum_{j=L}^{\infty} 2^{(j-L) \Gamma u_1(1+\psi)} 2^{(L-j) \alpha_{\infty} u_1(1+\psi) / 2} \\
& \lesssim \Lambda+\Lambda \sup _{L <  0, L \in \mathbb{Z}}  \psi^\theta \sum_{j=L}^{\infty} 2^{(j-L) u_1(1+\psi)\left(\Gamma-\alpha_{\infty} / 2\right)} \\
& \lesssim \Lambda .
\end{align*}

Now we will find the estimate for $III_2$. When $0<u_1 \leq  1$,  we get

\begin{align*}
I I I_{2}= & \sup_{\psi>0}\sup _{L <  0, L \in \mathbb{Z}} 2^{-L \Gamma u_1(1+\psi)} \psi^\theta \sum_{t_0=0}^{L} 2^{\alpha_{\infty} t_0u _1(1+\psi)}\left(\sum_{j=-\infty}^{t_0-1}\left|\lambda_{j}\right|\left\|\left[b, T_{\Omega, \sigma}\right] \chi_{t_0}\right\|_{L^{p_2(\cdot)}}\right)^{u_1(1+\psi)} \\
\lesssim & \sup_{\psi>0}\sup _{L <  0, L \in \mathbb{Z}} 2^{-L \Gamma u_1(1+\psi)} \psi^\theta  \sum_{t_0=0}^{L} 2^{\alpha_{\infty} t_0u _1(1+\psi)}\left(\sum_{j=-\infty}^{t_0-1}\left|\lambda_{j}\right|^{u_1(1+\psi)} 2^{[-j \alpha_j +(j-t_0)(\beta+n\delta_2)] u_1(1+\psi)}\right) \\
= & \sup_{\psi>0}\sup _{L <  0, L \in \mathbb{Z}} 2^{-L \Gamma u_1(1+\psi)} \psi^\theta\\
& \times \sum_{t_0=0}^{L} 2^{\alpha_{\infty} t_0u _1(1+\psi)}\left(\sum_{j=-\infty}^{-1}\left|\lambda_{j}\right|^{u_1(1+\psi)} 2^{\left[(\beta+n\delta_2)(j-t_0)-j \alpha(0)\right] u_1(1+\psi)}\right) \\
+& \sup_{\psi>0}\sup _{L <  0, L \in \mathbb{Z}} 2^{-L \Gamma u_1(1+\psi)} \psi^\theta \\
& \times \sum_{t_0=0}^{L} 2^{\alpha_{\infty} t_0u _1(1+\psi)}\left(\sum_{j=0}^{t_0-1}\left|\lambda_{j}\right|^{u_1(1+\psi)} 2^{\left[(\beta+n\delta_2)(j-t_0)-j \alpha_{\infty}\right] u_1(1+\psi)}\right)\\
\lesssim & \sup_{\psi>0}\sup _{L <  0, L \in \mathbb{Z}} 2^{-L \Gamma u_1(1+\psi)} \psi^\theta \sum_{t_0=0}^{L} 2^{t_0\left[\alpha_{\infty}-(\beta+n\delta_2)\right]  u_1(1+\psi)} \\
& \times \sum_{j=-\infty}^{-1}\left|\lambda_{j}\right|^{u_1(1+\psi)} 2^{\left(\beta+n\delta_2-\alpha(0)\right) j u_1(1+\psi)} \\
+& \sup_{\psi>0}\sup _{L <  0, L \in \mathbb{Z}} 2^{-L \Gamma u_1(1+\psi)} \psi^\theta \sum_{t_0=0}^{L}\left|\lambda_{j}\right|^{u_1(1+\psi)} \sum_{t_0=j+1}^{\infty} 2^{(j-t_0)\left(\beta+n\delta_2-\alpha_{\infty}\right) u_1(1+\psi)} \\
\lesssim & \sup_{\psi>0}\sup _{L <  0, L \in \mathbb{Z}} 2^{-L \Gamma u_1(1+\psi)} \psi^\theta \sum_{j=-\infty}^{-1}\left|\lambda_{j}\right|^{u_1(1+\psi)}+\sup_{\psi>0}\sup _{L <  0, L \in \mathbb{Z}} 2^{-L \Gamma u_1(1+\psi)} \psi^\theta \sum_{j=0}^{L-1}\left|\lambda_{j}\right|^{u_1(1+\psi)} \\
\lesssim & \Lambda .
\end{align*}

When $1<u_1(1+\psi)<\infty$ and since $n \delta_{2} \leq  \alpha(0)$, $\alpha_{\infty}<\beta+n\delta_2$, by H\"older's inequality, we have

\begin{align*}
I I I_{2}=& \sup_{\psi>0}\sup _{L <  0, L \in \mathbb{Z}} 2^{-L \Gamma u_1(1+\psi)} \psi^\theta \sum_{t_0=0}^{L} 2^{\alpha_{\infty} t_0u _1(1+\psi)}\left(\sum_{j=-\infty}^{t_0-1}\left|\lambda_{j}\right|\left\|\left[b, T_{\Omega, \sigma}\right] \chi_{t_0}\right\|_{L^{p_2(\cdot)}}\right)^{u_1(1+\psi)} \\
 \lesssim & \sup_{\psi>0}\sup _{L <  0, L \in \mathbb{Z}} 2^{-L \Gamma u_1(1+\psi)} \psi^\theta \sum_{t_0=0}^{L} 2^{\alpha_{\infty} t_0u _1(1+\psi)}\left(\sum_{j=-\infty}^{t_0-1}\left|\lambda_{j}\right| 2^{\left[(\beta+n\delta_2)(j-t_0)-j \alpha_{j}\right]}\right)^{u_1(1+\psi)} \\
\lesssim & \sup_{\psi>0}\sup _{L <  0, L \in \mathbb{Z}} 2^{-L \Gamma u_1(1+\psi)} \psi^\theta \sum_{t_0=0}^{L} 2^{\alpha_{\infty} t_0u _1(1+\psi)}\left(\sum_{j=-\infty}^{-1}\left|\lambda_{j}\right| 2^{\left[(\beta+n\delta_2)(j-t_0)-j \alpha(0)\right]}\right)^{u_1(1+\psi)} \\
 +&\sup_{\psi>0}\sup _{L <  0, L \in \mathbb{Z}} 2^{-L \Gamma u_1(1+\psi)} \psi^\theta \sum_{t_0=0}^{L} 2^{\alpha_{\infty} t_0u _1(1+\psi)}\left(\sum_{j=0}^{t_0-1}\left|\lambda_{j}\right| 2^{\left[(\beta+n\delta_2)(j-t_0)-j \alpha_{\infty}\right]}\right)^{u_1(1+\psi)} \\
\lesssim & \sup_{\psi>0}\sup _{L <  0, L \in \mathbb{Z}} 2^{-L \Gamma u_1(1+\psi)} \psi^\theta \sum_{t_0=0}^{L} 2^{\left[\alpha_{\infty}-(\beta+n\delta_2)\right] t_0u _1(1+\psi)} \\
 &\times \left(\sum_{j=-\infty}^{-1}\left|\lambda_{j}\right| 2^{\left(\beta+n\delta_2-\alpha(0)\right) j}\right)^{u_1(1+\psi)} \\
+& \sup_{\psi>0}\sup _{L <  0, L \in \mathbb{Z}} 2^{-L \Gamma u_1(1+\psi)} \psi^\theta \sum_{t_0=0}^{L}\left(\sum_{j=0}^{t_0-1}\left|\lambda_{j}\right| 2^{(j-t_0)\left(\beta+n\delta_2-\alpha_{\infty}\right)}\right)^{u_1(1+\psi)} \\
\lesssim & \left(\sup_{\psi>0}\sup _{L <  0, L \in \mathbb{Z}} 2^{-L \Gamma u_1(1+\psi)} \psi^\theta\sum_{j=-\infty}^{-1}\left|\lambda_{j}\right|^{u_1(1+\psi)} 2^{\left(\beta+n\delta_2-\alpha(0)\right) j u_1(1+\psi) / 2}\right) \\
 &\times\left(\sum_{j=-\infty}^{-1} 2^{\left(\beta+n\delta_2-\alpha(0)\right) j (u_1(1+\psi))^{\prime} / 2}\right)^{u_1(1+\psi) / (u_1(1+\psi))^{\prime}} \\
\lesssim & \sup_{\psi>0}\sup _{L <  0, L \in \mathbb{Z}} 2^{-L \Gamma u_1(1+\psi)} \psi^\theta \sum_{t_0=0}^{L}\left(\sum_{j=0}^{t_0-1}\left|\lambda_{j}\right|^{u_1(1+\psi)} 2^{(j-t_0)\left(\beta+n\delta_2-\alpha_{\infty}\right) u_1(1+\psi) / 2}\right) \\
& \times\left(\sum_{j=0}^{t_0-1} 2^{(j-t_0)\left(\beta+n\delta_2-\alpha_{\infty}\right) (u_1(1+\psi))^{\prime} / 2}\right)^{u_1(1+\psi) / (u_1(1+\psi))^{\prime}}\\
\lesssim & \sup_{\psi>0}\sup _{L <  0, L \in \mathbb{Z}} 2^{-L \Gamma u_1(1+\psi)} \psi^\theta \sum_{j=-\infty}^{-1}\left|\lambda_{j}\right|^{u_1(1+\psi)} 2^{\left(\beta+n\delta_2-\alpha(0)\right) j u_1(1+\psi) / 2} \\
& +\sup_{\psi>0}\sup _{L <  0, L \in \mathbb{Z}} 2^{-L \Gamma u_1(1+\psi)} \psi^\theta \sum_{t_0=0}^{L} \sum_{j=0}^{t_0-1}\left|\lambda_{j}\right|^{u_1(1+\psi)} 2^{(j-t_0)\left(\beta+n\delta_2-\alpha_{\infty}\right) u_1(1+\psi) / 2} \\
\lesssim & \sup_{\psi>0}\sup _{L <  0, L \in \mathbb{Z}} 2^{-L \Gamma u_1(1+\psi)} \psi^\theta \sum_{j=-\infty}^{-1}\left|\lambda_{j}\right|^{u_1(1+\psi)} \\
+& \sup_{\psi>0}\sup _{L <  0, L \in \mathbb{Z}} 2^{-L \Gamma u_1(1+\psi)} \psi^\theta \sum_{j=0}^{L-1}\left|\lambda_{j}\right|^{q} \sum_{t_0=j+1}^{L} 2^{(j-t_0)\left(\beta+n\delta_2-\alpha_{\infty}\right) u_1(1+\psi) / 2} \\
\lesssim & \sup_{\psi>0}\sup _{L <  0, L \in \mathbb{Z}} 2^{-L \Gamma u_1(1+\psi)} \psi^\theta \sum_{j=-\infty}^{-1}\left|\lambda_{j}\right|^{u_1(1+\psi)}+\sup_{\psi>0}\sup _{L <  0, L \in \mathbb{Z}} 2^{-L \Gamma u_1(1+\psi)} \psi^\theta \sum_{j=0}^{L-1}\left|\lambda_{j}\right|^{u_1(1+\psi)}\\
\lesssim & \Lambda .
\end{align*}

Thus we have obtained our desired results. 

\end{proof}

\section{Lipschitz estimate for the grand Herz-Morrey  Hardy spaces}

For $0<\gamma \leq 1$, the Lipschitz space $\operatorname{Lip}_{\gamma}\left(\mathbb{R}^{n}\right)$ is defined as follows:

$$
\operatorname{Lip}_{\gamma}\left(\mathbb{R}^{n}\right)=\left\{f:\left\| g\right\|_{\operatorname{Lip}_{\gamma}}=\sup _{x, y \in \mathbb{R}^{n} ; x \neq y} \frac{|g(x)-g(y)|}{|x-y|^{\gamma}}<\infty\right\}
$$

Let $b \in \operatorname{Lip}_{\gamma}\left(\mathbb{R}^{n}\right)$. It is easy to see that $\left|\left[b, T_{\Omega, \sigma}\right]\right| \leq C\|b\|_{\operatorname{Lip}_{\gamma}}\left|T_{\Omega, \sigma+\gamma}\right|$. In [15], the authors proved that the operator $T_{\Omega, \sigma}$ is bounded from $L^{p_{1}(\cdot)}\left(\mathbb{R}^{n}\right)$ to $L^{p_{2}(\cdot)}\left(\mathbb{R}^{n}\right)$ for $1 / p_{1}(x)-1 / p_{2}(x)=\sigma / n$ and $p_{1}(\cdot) \in \mathcal{P}\left(\mathbb{R}^{n}\right)$ satisfying the conditions \eqref{r5} and \eqref{r6}  with $p_{1}^{+}<n / \sigma$. So, we can state the following theorem.

\begin{theorem}
    Suppose that $b \in \operatorname{Lip}_{\gamma}\left(\mathbb{R}^{n}\right)$ with $0<\gamma \leq 1$ and $0<\sigma<n-\gamma$. Let $p_{1}(\cdot) \in \mathcal{P}\left(\mathbb{R}^{n}\right)$ satisfy the conditions \eqref{r5} and \eqref{r6} with $p_{1}^{+}<n /(\sigma+\gamma), 1 / p_{1}(x)-1 / p_{2}(x)=(\sigma+$ $\gamma) / n$, and let $\Omega \in L^{s}\left(\mathrm{~S}^{n-1}\right)\left(s>p_{2}^{+}\right)$with $1 \leq s^{\prime}<p_{1}^{-}$. Then the commutator $\left[b, T_{\Omega, \sigma}\right]$ is bounded from $L^{p_{1}}(\cdot)\left(\mathbb{R}^{n}\right)$ to $L^{p_{2}(\cdot)}\left(\mathbb{R}^{n}\right)$.
\end{theorem}

\begin{theorem}

Suppose that  $0 \leq \Gamma<\infty$, $b \in \operatorname{Lip}_\gamma\left(\mathbb{R}^{n}\right), 0<\gamma \leq 1,0<\sigma<n-\gamma$,  $\alpha(\cdot) \in L^{\infty}\left(\mathbb{R}^{n}\right)$ be log-H\"older continuous both at the origin and infinity, and $p_{1}(\cdot) \in \mathcal{P}\left(\mathbb{R}^{n}\right)$ satisfies the conditions \eqref{r5} and \eqref{r6} with $p_{1}^{+}<n /( \sigma+\gamma)$ and $1 / p_{1}(x)-1 / p_{2}(x)=(\sigma +\gamma)/ n$.
Let $\Omega \in L^{s}\left(\mathrm{~S}^{n-1}\right)\left(s>p_{2}^{+}\right)$with $1 \leq s^{\prime}<p_{1}^{-}$ and satisfy $\int_{0}^{1} \frac{\omega_{s}(\delta)}{\delta^{1+\gamma}} d \delta<\infty$.
Let $0<u_{1} \leq u_{2}<\infty$, with $n \delta_{2} \leq  \alpha(0), \alpha_{\infty}<n \delta_{2} +\gamma $, then $\left[b, T_{\Omega, \sigma}\right]$ is bounded from $H M\dot{K} ^{\alpha(\cdot),u_1),\theta}_{ \Gamma, p_1(\cdot)}(\mathbb{R}^n)$ to $ M\dot{K} ^{\alpha(\cdot),u_2),\theta}_{ \Gamma, p_2(\cdot)}(\mathbb{R}^n)$.  
\end{theorem}

\begin{proof}
  
 Suppose that  $g\in   H {M \dot{K} ^{\alpha(\cdot),u_1),\theta}_{ \Gamma, p_1(\cdot)}(\mathbb{R}^n)}$. By using Theorem 2.3, $f=\sum_{j=-\infty}^{\infty} \lambda_{j} b_{j}$ converges in $\mathcal{S}^{\prime}\left(\mathbb{R}^{n}\right)$, where each $b_{j}$ is a central $(\alpha(\cdot), p(\cdot))$-atom with support contained in $B_{j}$ and

\begin{equation*}
\left\| g\right\|_ {H {M \dot{K} ^{\alpha(\cdot),u_1),\theta}_{ \Gamma, p_1(\cdot)}(\mathbb{R}^n)}} \approx \inf \left(\sup_{\psi>0}\sup _{L \in \mathbb{Z}} 2^{-L \Gamma}\left(\psi^\theta\sum_{j=-\infty}^{L}\left|\lambda_{j}\right|^{u_1(1+\psi)}\right)^{\frac{1}{u_1(1+\psi)}}\right) . 
\end{equation*}

For simplicity, we denote $\Lambda=\sup_{\psi>0}\sup _{L \in \mathbb{Z}} 2^{-L \Gamma u_1(1+\psi)}\psi^\theta \sum_{j=-\infty}^{L}\left|\lambda_{j}\right|^{u_1(1+\psi)}$. By Proposition $1.6$, we have

\begin{align*}
&\left \|\left[b, T_{\Omega, \sigma}\right] f\right\|_{M \dot{K} ^{\alpha(\cdot),u_2),\theta}_{ \Gamma, p_2(\cdot)}(\mathbb{R}^n)}^{u_1(1+\psi)} \\
& \approx \max \left\{\sup_{\psi>0}\sup _{L <  0, L \in \mathbb{Z}} 2^{-L \Gamma u_1(1+\psi)} \psi^\theta \left(\sum_{t_0=-\infty}^{L} 2^{t_0\alpha (0) u_1(1+\psi)}\left\|\left(\left[b, T_{\Omega, \sigma}\right] f\right) \chi_{t_0}\right\|_{L^{p_2(\cdot)}}^{u_1(1+\psi)}\right),\right. \\
& \sup_{\psi>0}\sup _{L \geq  0, L \in \mathbb{Z}} 2^{-L \Gamma u_1(1+\psi)} \psi^\theta \left(\sum_{t_0=-\infty}^{-1} 2^{t_0\alpha (0) u_1(1+\psi)}\left\|(\left[b, T_{\Omega, \sigma}\right] f) \chi_{t_0}\right\|_{L^{p_2(\cdot)}}^{u_1(1+\psi)}\left.\quad+\sum_{t_0=0}^{L} 2^{t_0 \alpha_\infty u_1(1+\psi)}\left\|\left(\left[b, T_{\Omega, \sigma}\right] f\right) \chi_{t_0}\right\|_{L^{p_2(\cdot)}}^{u_1(1+\psi)}\right)\right\} \\
& \lesssim \max \{J, JJ+JJJ\}. 
\end{align*}

We will find the estimated for $J$ and $JJJ$ and estimate of $J$ can be obtained similarly.

To complete our proof, we only need show that there exists a positive constant $C$ such that $J, JJ, JJJ \leq  C \Lambda$.

Firstly, we will find the  estimate of $J$ :

\begin{align*}
    J=&\sup_{\psi>0}\sup _{L <  0, L \in \mathbb{Z}} 2^{-L \Gamma u_1(1+\psi)} \psi^\theta \left(\sum_{t_0=-\infty}^{L} 2^{t_0\alpha (0) u_1(1+\psi)}\left\|\left( \left[b, T_{\Omega, \sigma}\right] f\right) \chi_{t_0}\right\|_{L^{p_2(\cdot)}}^{u_1(1+\psi)}\right)\\
    \leq & \sup_{\psi>0}\sup _{L <  0, L \in \mathbb{Z}} 2^{-L \Gamma u_1(1+\psi)} \psi^\theta \sum_{t_0=-\infty}^{L} 2^{t_0\alpha (0) u_1(1+\psi)}\left(\sum\limits_{j=-\infty}^{t_0-1} \left| \lambda_j\right|\left\|\left( \left[b, T_{\Omega, \sigma}\right](a_j) \right) \chi_{t_0}\right\|_{L^{p_2(\cdot)}}\right)^{u_1(1+\psi)}\\
    + & \sup_{\psi>0}\sup _{L <  0, L \in \mathbb{Z}} 2^{-L \Gamma u_1(1+\psi)} \psi^\theta \sum_{t_0=-\infty}^{L} 2^{t_0\alpha (0) u_1(1+\psi)}\left(\sum\limits_{j=k}^{\infty} \left| \lambda_j\right|\left\|\left( \left[b, T_{\Omega, \sigma}\right](a_j) \right) \chi_{t_0}\right\|_{L^{p_2(\cdot)}}\right)^{u_1(1+\psi)}\\
    :=& J_1+J_2. 
\end{align*}

We first estimate $J_{1}$. For each $k \in \mathbb{Z}, j \leq t_0-1$ and a.e. $x \in F_{t_0}$, using Lemma 2.4, the Minkowski's inequality and the vanishing moments of $a_{j}$, we get

$$
\begin{aligned}
\left\|\left[b, T_{\Omega, \sigma}\right]\left(a_{j}\right) \chi_{t_0}\right\|_{L^{p_{2}(\cdot)}} & \lesssim\int_{B_{j}}\left\|\left|\frac{\Omega(\cdot-y)}{|\cdot-y|^{n-\sigma}}-\frac{\Omega(\cdot)}{|\cdot|^ {n- \sigma}}\right|(b(\cdot)-b(y)) \chi_{t_0}(\cdot)\right\|_{L^{p_{2}(\cdot)}}\left|a_{j}(y)\right| d y \\
& \lesssim \int_{B_{j}}\left\|\left|\frac{\Omega(\cdot-y)}{|\cdot-y|^{n-\sigma}}-\frac{\Omega(\cdot)}{|\cdot|^{n-\sigma}}\right|\left|b(\cdot)-b(0)\right| \chi_{t_0}(\cdot)\right\|_{L^{p_{2}(\cdot)}}\left|a_{j}(y)\right| d y \\
& \lesssim \int_{B_{j}}\left\|\left|\frac{\Omega(\cdot-y)}{|\cdot-y|^{n-\sigma}}-\frac{\Omega(\cdot)}{|\cdot|^{n-\sigma}}\right| \chi_{t_0}(\cdot)\right\|_{L^{p_{2}(\cdot)}}\left|b(0)-b(y)\right|\left|a_{j}(y)\right| d y \\
& =: J_{11}+J_{12}. 
\end{aligned}
$$

To estimate $J_{11}$, we note that $s>p_{2}^{+}$, and denote $\tilde{p}_{2}(\cdot)>1$ and $\frac{1}{p_{2}(x)}=\frac{1}{\tilde{p}_{2}(x)}+\frac{1}{s}$. Then by Lemmas $3.3$ and $3.5$ we have

$$
\begin{aligned}
& \left\|\left|\frac{\Omega(\cdot-y)}{\left| \cdot-y\right|^{n-\sigma}}-\frac{\Omega(\cdot)}{|\cdot|^{n-\sigma}}\right|\left|b(\cdot)-b(0)\right| \chi_{t_0}(\cdot)\right\|_{L^{p_{2}(\cdot)}} \\
& \lesssim\left\|\left|\frac{\Omega(\cdot-y)}{\left| \cdot-y\right|^{n-\sigma}}-\frac{\Omega(\cdot)}{|\cdot| ^{n-\sigma}}\right| \chi_{t_0}(\cdot)\right\|_{L^{s}}\left\|\left|b(\cdot)-b(0)\right| \chi_{t_0}(\cdot) \right\|_{L^{\tilde{p}_{2}(\cdot)}} \\
& \lesssim \left\|\left|\frac{\Omega(\cdot-y)}{\left| \cdot-y\right|^{n-\sigma}}-\frac{\Omega(\cdot)}{|\cdot| ^{n-\sigma}}\right| \chi_{t_0}(\cdot)\right\|_{L^{s}}2^{k\gamma}\left\| b\right\|_{Lip_\gamma}\left\| \chi_{B_{t_0}} \right\|_{L^{\tilde{p}_{2}(\cdot)}} .
\end{aligned}
$$

When $\left|B_{t_0}\right| \lesssim 2^{n}$ and $x_{t_0} \in B_{t_0}$, by Lemma $3.6$ we have

$$
\left\|\chi_{B_{t_0}}\right\|_{L^{\tilde{p}_{2}(\cdot)}} \approx\left|B_{t_0}\right|^{\frac{1}{p_{2}\left(x_{t_0}\right)}} \approx\left\|\chi_{B_{t_0}}\right\|_{L^{p_{1}(\cdot)}}\left|B_{t_0}\right|^{-\frac{1}{s}-\frac{\sigma+\gamma}{n}}.
$$

When $\left|B_{t_0}\right| \geq 1$ we have

$$
\left\|\chi_{B_{t_0}}\right\|_{L^{\tilde{p}_{2}(\cdot)}} \approx\left|B_{t_0} \right|^{\frac{1}{p_{2}(\infty)}} \approx\left\|\chi_{B_{t_0}}\right\|_{L^{p_{1}(\cdot)}}\left| B_{t_0}\right|^{-\frac{1}{s}-\frac{\sigma+\gamma}{n}}.
$$

So, we obtain $\left\|\chi_{B_{t_0}}\right\|_{L^{\tilde{p}_{2}(\cdot)}} \approx\left\|\chi_{B_{t_0}}\right\|_{L^{p_{1}(\cdot)}}\left|B_{t_0}\right|^{-\frac{1}{s}-\frac{\sigma+\gamma}{n}}$.

Meanwhile, by Lemma $3.7$ we have

$$
\begin{aligned}
&\left\|\left|\frac{\Omega(\cdot-y)}{\left| \cdot-y\right|^{n-\sigma}}-\frac{\Omega(\cdot)}{|\cdot| ^{n-\sigma}}\right| \chi_{t_0}(\cdot)\right\|_{L^{s}} \lesssim 2^{(t_0-1)\left(\frac{n}{s}-(n-\sigma)\right)}\left\{\frac{|y|}{2^{t_0}}+\int_{|y| / 2^{t_0}}^{|y| 2^{t_0-1}} \frac{\omega_{s}(\delta)}{\delta} d \delta\right\} \\
& \lesssim 2^{(t_0-1)\left(\frac{n}{s}-(n-\sigma)\right)}\left(2^{j-t_0+1}+2^{(j-t_0+1) \gamma} \int_{0}^{1} \frac{\omega_{s}(\delta)}{\delta} d \delta\right) \lesssim  2^{(t_0-1)\left(\frac{n}{s}-(n-\sigma)\right)} 2^{(j-t_0) \gamma}.
\end{aligned}
$$

So, using the generalized Hölder inequality, we obtain the following estimate for $J_{11}$ :

\begin{align*}
J_{11}=&\int_{B_{j}}\left\|\left|\frac{\Omega(\cdot-y)}{|\cdot-y|^{n-\sigma}}-\frac{\Omega(\cdot)}{|\cdot|^{n-\sigma}}\right|\left|b(\cdot)-b(0)\right| \chi_{t_0}(\cdot)\right\|_{L^{p_{2}(\cdot)}}\left|a_{j}(y)\right| d y \\
\lesssim & 2^{k\gamma}\|b\|_{Lip_\gamma} 2^{(t_0-1)\left(\frac{n}{s}-(n-\sigma)\right)} 2^{(j-t_0) \gamma}\left\|\chi_{B_{t_0}}\right\|_{L^{p_{1}(\cdot)}}\left|B_{t_0}\right|^{-\frac{1}{s}-\frac{\sigma+\gamma}{n}} \int_{B_{j}}\left|a_{j}(y)\right| d y  \\
\lesssim &\|b\|_{Lip_\gamma} 2^{-t_0n+(j-t_0) \gamma}\left\|\chi_{B_{t_0}}\right\|_{L^{p_{1}(\cdot)}}\left\|a_{j}\right\|_{L^{p_{1}(\cdot)}}\left\|\chi_{B_{j}}\right\|_{L^{p_{1}^{\prime}(\cdot)}}.
\end{align*}

To estimate $J_{12}$, we use arguments similar to those applied for $J_{11}$, to obtain

$$
\begin{aligned}
 \left\|\left|\frac{\Omega(\cdot-y)}{|\cdot-y|^{n-\sigma}}-\frac{\Omega(\cdot)}{|\cdot|^{n-\sigma}}\right| \chi_{t_0}(\cdot)\right\|_{L^{p_{2}(\cdot)}} &\lesssim\left\|\left|\frac{\Omega(\cdot-y)}{|\cdot-y|^{n-\sigma}}-\frac{\Omega(\cdot)}{|\cdot|^{n-\sigma}}\right| \chi_{t_0}(\cdot)\right\|_{L^{s}} \left\|\chi_{t_0}(\cdot) \right\|_{L^{\tilde{p}_{2}(\cdot)}} \\
& \lesssim\left\|\left|\frac{\Omega(\cdot-y)}{|\cdot-y|^{n-\sigma}}-\frac{\Omega(\cdot)}{|\cdot|^{n-\sigma}}\right| \chi_{t_0}(\cdot)\right\|_{L^{s}} \left\|\chi_{B_k} \right\|_{L^{\tilde{p}_{2}(\cdot)}} \\
& \lesssim  2^{(t_0-1)\left(\frac{n}{s}-(n-\sigma)\right)} 2^{(j-t_0) \gamma}\left\|\chi_{B_{t_0}}\right\|_{L^{\tilde{p}_{2}(\cdot)}} \lesssim  2^{-t_0n+(j-t_0) \gamma-k\gamma}\left\|\chi_{B_{t_0}}\right\|_{L^{p_{1}(\cdot)}}.
\end{aligned}
$$

So, by Lemma $3.3$ and the generalized Hölder inequality, we have

\begin{align*}
J_{12}&=  \int_{B_{j}}\left\|\left|\frac{\Omega(\cdot-y)}{|\cdot-y|^{n-\sigma}}-\frac{\Omega(\cdot)}{|\cdot|{ }^{n-\sigma}}\right| \chi_{t_0}(\cdot)\right\|_{L^{p_{2}(\cdot)}}\left|b(0)-b(y) \| a_{j}(y)\right| d y \\
& \lesssim  2^{-t_0n+(j-t_0) \gamma-k\gamma}\left\|\chi_{B_{t_0}}\right\|_{L^{p_{1}(\cdot)}} \int_{B_{j}}\left|b(0)-b(y) \| a_{j}(y)\right| d y \\
&\lesssim   2^{-t_0n+(j-t_0) \gamma}\left\|\chi_{B_{t_0}}\right\|_{L^{p_{1}(\cdot)}}\left\|\left(b(0)-b\right) \chi_{B_{j}}\right\|_{L^{p_{1}^{\prime}(\cdot)}}\left\|a_{j}\right\|_{L^{p_{1}(\cdot)}} \\
&\lesssim  \|b\|_{Lip_\gamma} 2^{-t_0n+(j-t_0) \gamma}\left\|\chi_{B_{t_0}}\right\|_{L^{p_{1}(\cdot)}}\left\|a_{j}\right\|_{L^{p_{1}(\cdot)}}\left\|\chi_{B_{j}}\right\|_{L^{p_{1}^{\prime}(\cdot)}}.
\end{align*}

Next, using estimates for $J_{11}$,  and  $J_{12}$, we can write

\begin{align*}
    \left\|\left[b, T_{\Omega, \sigma}\right]\left(a_{j}\right) \chi_{t_0}\right\|_{L^{p_{2}(\cdot)}}\lesssim & \|b\|_{Lip_\gamma} 2^{-t_0n+(j-t_0) \gamma}\left\|\chi_{B_{t_0}}\right\|_{L^{p_{1}(\cdot)}}\left\|a_{j}\right\|_{L^{p_{1}(\cdot)}}\left\|\chi_{B_{j}}\right\|_{L^{p_{1}^{\prime}(\cdot)}}\\
    \lesssim & \|b\|_{Lip_\gamma}2^{(j-t_0) \gamma}\left\|a_{j}\right\|_{L^{p_{1}(\cdot)}}\frac{\left\|\chi_{B_{j}}\right\|_{L^{p_{1}^{\prime}(\cdot)}}}{\left\|\chi_{B_{t_0}}\right\|_{L^{p_{1}^{\prime}(\cdot)}}}\\
    \lesssim &  \|b\|_{Lip_\gamma} 2^{-j \alpha_j +(j-t_0)(\gamma+n\delta_2)}.
\end{align*}
So we have
\begin{align*}
J_{1} \lesssim &\sup_{\psi>0}\sup _{L <  0, L \in \mathbb{Z}} 2^{-L \Gamma u_1(1+\psi)} \psi^\theta \sum_{t_0=-\infty}^{L} 2^{t_0\alpha (0) u_1(1+\psi)}\left(\sum\limits_{j=-\infty}^{t_0-1} \left| \lambda_j\right|\left\|\left( \left[b, T_{\Omega, \sigma}\right](a_j) \right) \chi_{t_0}\right\|_{L^{p_2(\cdot)}}\right)^{u_1(1+\psi)} \\
\lesssim & \sup_{\psi>0}\sup _{L <  0, L \in \mathbb{Z}} 2^{-L \Gamma u_1(1+\psi)} \psi^\theta \sum_{t_0=-\infty}^{L} 2^{t_0\alpha (0) u_1(1+\psi)}\left(\sum\limits_{j=-\infty}^{t_0-1} \left| \lambda_j\right|\|b\|_{Lip_\gamma} 2^{-j \alpha_j +(j-t_0)(\gamma+n\delta_2)}\right)^{u_1(1+\psi)} \\
\lesssim &\sup_{\psi>0}\sup _{L <  0, L \in \mathbb{Z}} 2^{-L \Gamma u_1(1+\psi)} \psi^\theta \sum_{t_0=-\infty}^{L} \left(\sum\limits_{j=-\infty}^{t_0-1} \left| \lambda_j\right|\|b\|_{Lip_\gamma} 2^{(j-t_0)(\gamma+n\delta_2-\alpha(0))}\right)^{u_1(1+\psi)} .
\end{align*}

When $1<u_{1}<\infty$, take $1 / u_{1}+1 / u_{1}^{\prime}=1$. Since $\gamma+n \delta_{2}-\alpha(0)>0$, we can use the Hölder inequality to obtain

\begin{align*}
J_{1} \lesssim& \|b\|_{Lip_\gamma}^{u_1(1+\psi)} \sup_{\psi>0}\sup _{L <  0, L \in \mathbb{Z}} 2^{-L \Gamma u_1(1+\psi)} \psi^\theta \sum_{t_0=-\infty}^{L} \left(\sum\limits_{j=-\infty}^{t_0-1} \left| \lambda_j\right|^{u_1(1+\psi)}2^{(j-t_0)(\gamma+n\delta_2-\alpha(0)){u_1(1+\psi)/2}}\right)\\
\times &\left(\sum\limits_{j=-\infty}^{t_0-1} 2^{(j-t_0)(\gamma+n\delta_2-\alpha(0)){(u_1(1+\psi))^\prime/2}}\right)^{\frac{u_1(1+\psi)}{(u_1(1+\psi))^\prime}}\\
\lesssim& \|b\|_{*} ^{u_1(1+\psi)}\sup_{\psi>0}\sup _{L <  0, L \in \mathbb{Z}} 2^{-L \Gamma u_1(1+\psi)} \psi^\theta \sum_{t_0=-\infty}^{L} \left(\sum\limits_{j=-\infty}^{t_0-1} \left| \lambda_j\right|^{u_1(1+\psi)}2^{(j-t_0)(\gamma+n\delta_2-\alpha(0)){u_1(1+\psi)/2}}\right)\\
\lesssim& \|b\|_{Lip_\gamma} ^{u_1(1+\psi)}\sup_{\psi>0}\sup _{L <  0, L \in \mathbb{Z}} 2^{-L \Gamma u_1(1+\psi)} \psi^\theta \sum_{j=-\infty}^{L}\left| \lambda_j\right|^{u_1(1+\psi)} \left(\sum\limits_{t_0=j+1}^{-1} 2^{(j-t_0)(\gamma+n\delta_2-\alpha(0)){u_1(1+\psi)/2}}\right)\\
\lesssim& \|b\|_{Lip_\gamma} ^{u_1(1+\psi)}\Lambda. 
\end{align*}

When $0<u_{1} \leq 1$, we have

\begin{align*}
J_{1} = & \sup_{\psi>0}\sup _{L <  0, L \in \mathbb{Z}} 2^{-L \Gamma u_1(1+\psi)} \psi^\theta \sum_{t_0=-\infty}^{L} \left(\sum\limits_{j=-\infty}^{t_0-1} \left| \lambda_j\right|\|b\|_{Lip_\gamma} 2^{(j-t_0)(\gamma+n\delta_2-\alpha(0))}\right)^{u_1(1+\psi)}\\
 \lesssim &\|b\|_{Lip_\gamma} ^{u_1(1+\psi)}\sup_{\psi>0}\sup _{L <  0, L \in \mathbb{Z}} 2^{-L \Gamma u_1(1+\psi)} \psi^\theta \sum_{t_0=-\infty}^{L} \left(\sum\limits_{j=-\infty}^{t_0-1} \left| \lambda_j\right|^{u_1(1+\psi)} 2^{(j-t_0)(\gamma+n\delta_2-\alpha(0)){u_1(1+\psi)}}\right)\\
 \lesssim &\|b\|_{Lip_\gamma} ^{u_1(1+\psi)}\sup_{\psi>0}\sup _{L <  0, L \in \mathbb{Z}} 2^{-L \Gamma u_1(1+\psi)} \psi^\theta \sum_{j=-\infty}^{L} \left| \lambda_j\right|^{u_1(1+\psi)} \left(\sum\limits_{t_0=j+1}^{-1}  2^{(j-t_0)(\gamma+n\delta_2-\alpha(0)){u_1(1+\psi)}}\right)\\
 \lesssim &\|b\|_{Lip_\gamma} ^{u_1(1+\psi)}\sup_{\psi>0}\sup _{L <  0, L \in \mathbb{Z}} 2^{-L \Gamma u_1(1+\psi)} \psi^\theta \sum_{j=-\infty}^{L} \left| \lambda_j\right|^{u_1(1+\psi)} \\
 \lesssim &  \|b\|_{Lip_\gamma} ^{u_1(1+\psi)} \Lambda. 
\end{align*}

Next we will find the estimate of $I_2$. By using the boundedess of $\left[b, T_{\Omega, \sigma}\right]$ on variable Lebesgue spaces, we have

\begin{align*}
    J_2 = &  \sup_{\psi>0}\sup _{L <  0, L \in \mathbb{Z}} 2^{-L \Gamma u_1(1+\psi)} \psi^\theta \sum_{t_0=-\infty}^{L} 2^{t_0\alpha (0) u_1(1+\psi)}\left(\sum\limits_{j=k}^{\infty} \left| \lambda_j\right|\left\|\left( \left[b, T_{\Omega, \sigma}\right](a_j) \right) \chi_{t_0}\right\|_{L^{p_2(\cdot)}}\right)^{u_1(1+\psi)}\\
    \lesssim &   \sup_{\psi>0}\sup _{L <  0, L \in \mathbb{Z}} 2^{-L \Gamma u_1(1+\psi)} \psi^\theta \sum_{t_0=-\infty}^{L} 2^{t_0\alpha (0) u_1(1+\psi)}\left(\sum\limits_{j=k}^{\infty} \left| \lambda_j\right|\left\|a_j\right\|_{L^{p_2(\cdot)}}\right)^{u_1(1+\psi)}\\
    \lesssim &   \sup_{\psi>0}\sup _{L <  0, L \in \mathbb{Z}} 2^{-L \Gamma u_1(1+\psi)} \psi^\theta \sum_{t_0=-\infty}^{L} 2^{t_0\alpha (0) u_1(1+\psi)}\left(\sum\limits_{j=k}^{\infty} \left| \lambda_j\right|2^{-j\alpha_j}\right)^{u_1(1+\psi)}\\
    \lesssim &   \sup_{\psi>0}\sup _{L <  0, L \in \mathbb{Z}} 2^{-L \Gamma u_1(1+\psi)} \psi^\theta \sum_{t_0=-\infty}^{L} 2^{t_0\alpha (0) u_1(1+\psi)}\left(\sum\limits_{j=k}^{-1} \left| \lambda_j\right|^{u_1(1+\psi)}2^{-j\alpha(0){u_1(1+\psi)}}+\sum\limits_{j=0}^{\infty} \left| \lambda_j\right|^{u_1(1+\psi)}2^{-j\alpha_\infty{u_1(1+\psi)}}\right)\\
    \lesssim &   \sup_{\psi>0}\sup _{L <  0, L \in \mathbb{Z}} 2^{-L \Gamma u_1(1+\psi)} \psi^\theta \sum_{t_0=-\infty}^{L} \sum\limits_{j=k}^{-1} \left| \lambda_j\right|^{u_1(1+\psi)}2^{\alpha(0)(t_0-j){u_1(1+\psi)}}\\
    & + \sup_{\psi>0}\sup _{L <  0, L \in \mathbb{Z}} 2^{-L \Gamma u_1(1+\psi)} \psi^\theta \sum_{t_0=-\infty}^{L} 2^{t_0\alpha (0) u_1(1+\psi)}\sum\limits_{j=0}^{\infty} \left| \lambda_j\right|^{u_1(1+\psi)}2^{-\alpha_\infty j{u_1(1+\psi)}}\\
    \lesssim &   \sup_{\psi>0}\sup _{L <  0, L \in \mathbb{Z}} 2^{-L \Gamma u_1(1+\psi)} \psi^\theta \sum_{j=-\infty}^{-1} \left| \lambda_j\right|^{u_1(1+\psi)}\sum\limits_{t_0=-\infty}^{j} 2^{\alpha(0)(t_0-j){u_1(1+\psi)}}\\
    & + \sup_{\psi>0}\sup _{L <  0, L \in \mathbb{Z}}  \psi^\theta \sum_{j=0}^\infty 2^{-j \Gamma u_1(1+\psi)} \left| \lambda_j\right|^{u_1(1+\psi)} 2^{(\Gamma-\alpha_\infty)j(u_1(1+\psi))}2^{-L  \Gamma u_1(1+\psi)}\sum\limits_{t_0=-\infty}^{L} 2^{\alpha(0)k {u_1(1+\psi)}}\\
    \lesssim &   \sup_{\psi>0}\sup _{L <  0, L \in \mathbb{Z}} 2^{-L \Gamma u_1(1+\psi)} \psi^\theta \sum_{j=-\infty}^{L}  \left| \lambda_j\right|^{u_1(1+\psi)}+\sup_{\psi>0}\sup _{L <  0, L \in \mathbb{Z}} 2^{-L \Gamma u_1(1+\psi)} \psi^\theta \sum_{j=L}^{-1}  \left| \lambda_j\right|^{u_1(1+\psi)}\sum\limits_{t_0=-\infty}^{j} 2^{\alpha(0)(t_0-j){u_1(1+\psi)}}\\
     & + \Lambda  \sup_{\psi>0}\sup _{L <  0, L \in \mathbb{Z}}  \psi^\theta \sum_{j=0}^\infty 2^{(\Gamma-\alpha_\infty)j(u_1(1+\psi))}\sum\limits_{t_0=-\infty}^{L} 2^{(\alpha(0)t_0-l\Gamma) {u_1(1+\psi)}}\\
     \lesssim & \Lambda + \sup_{\psi>0}\sup _{L <  0, L \in \mathbb{Z}}  \psi^\theta \sum_{j=L}^{-1}2^{-j \Gamma u_1(1+\psi)}  \left| \lambda_j\right|^{u_1(1+\psi)}2^{ (j-\Gamma) u_1(1+\psi)}\sum\limits_{t_0=-\infty}^{j} 2^{\alpha(0)(t_0-j){u_1(1+\psi)}}+\Lambda\\
     \lesssim & \Lambda. 
\end{align*}

Now we will find estimate for $III$. 
\begin{align*}
    III=&  \sup_{\psi>0}\sup _{L \geq  0, L \in \mathbb{Z}} 2^{-L \Gamma u_1(1+\psi)} \psi^\theta \sum_{t_0=0}^{L} 2^{t_0 \alpha_\infty u_1(1+\psi)}\left\|\left(\left[b, T_{\Omega, \sigma}\right] f\right) \chi_{t_0}\right\|_{L^{p_2(\cdot)}}^{u_1(1+\psi)}\\
    \lesssim & \sup_{\psi>0}\sup _{L \geq  0, L \in \mathbb{Z}} 2^{-L \Gamma u_1(1+\psi)} \psi^\theta \sum_{t_0=0}^{L} 2^{t_0 \alpha_\infty u_1(1+\psi)}\left(\sum\limits_{j=k}^\infty\left| \lambda_j\right|\left\|\left(\left[b, T_{\Omega, \sigma}\right] a_j\right) \chi_{t_0}\right\|_{L^{p_2(\cdot)}}\right)^{u_1(1+\psi)}\\
    & +  \sup_{\psi>0}\sup _{L \geq  0, L \in \mathbb{Z}} 2^{-L \Gamma u_1(1+\psi)} \psi^\theta \sum_{t_0=0}^{L} 2^{t_0 \alpha_\infty u_1(1+\psi)}\left(\sum\limits_{j=-\infty}^{t_0-1}\left| \lambda_j\right|\left\|\left(\left[b, T_{\Omega, \sigma}\right] a_j\right) \chi_{t_0}\right\|_{L^{p_2(\cdot)}}\right)^{u_1(1+\psi)}\\
    :=& III_1+III_2.
\end{align*}

When $0<u_1\leq 1$, by using the boundedness of $\left[b, T_{\Omega, \sigma}\right]$ in $L^{p_2(\cdot)}$, we have

\begin{align*}
    III_1= & \sup_{\psi>0}\sup _{L \geq  0, L \in \mathbb{Z}} 2^{-L\Gamma  u_1(1+\psi)} \psi^\theta \sum_{t_0=0}^{L} 2^{t_0 \alpha_\infty u_1(1+\psi)}\left(\sum\limits_{j=k}^\infty\left| \lambda_j\right|\left\|\left(\left[b, T_{\Omega, \sigma}\right] a_j\right) \chi_{t_0}\right\|_{L^{p_2(\cdot)}}\right)^{u_1(1+\psi)}\\
    \lesssim & \sup_{\psi>0}\sup _{L \geq  0, L \in \mathbb{Z}} 2^{-L\Gamma  u_1(1+\psi)} \psi^\theta \sum_{t_0=0}^{L} 2^{t_0 \alpha_\infty u_1(1+\psi)}\left(\sum\limits_{j=k}^\infty\left| \lambda_j\right|^{u_1(1+\psi)}\left\|\left(\left[b, T_{\Omega, \sigma}\right] a_j\right) \chi_{t_0}\right\|^{u_1(1+\psi)}_{L^{p_2(\cdot)}}\right)\\
    \lesssim & \sup_{\psi>0}\sup _{L \geq  0, L \in \mathbb{Z}} 2^{-L\Gamma  u_1(1+\psi)} \psi^\theta \sum_{t_0=0}^{L} 2^{t_0 \alpha_\infty u_1(1+\psi)}\sum\limits_{j=k}^\infty\left| \lambda_j\right|^{u_1(1+\psi)}2^{-\alpha_j ju_1(1+\psi)}\\
    \lesssim & \sup_{\psi>0}\sup _{L \geq  0, L \in \mathbb{Z}} 2^{-L\Gamma  u_1(1+\psi)} \psi^\theta \sum_{t_0=0}^{L} 2^{t_0 \alpha_\infty u_1(1+\psi)}\sum\limits_{j=k}^\infty\left| \lambda_j\right|^{u_1(1+\psi)}2^{-\alpha_\infty ju_1(1+\psi)}\\
     \lesssim & \sup_{\psi>0}\sup _{L \geq  0, L \in \mathbb{Z}} 2^{-L\Gamma  u_1(1+\psi)} \psi^\theta \sum\limits_{j=0}^L\left| \lambda_j\right|^{u_1(1+\psi)}\sum\limits_{t_0=0}^j2^{(t_0-j)\alpha_\infty u_1(1+\psi)}+  \sup_{\psi>0}\sup _{L \geq  0, L \in \mathbb{Z}} 2^{-L\Gamma  u_1(1+\psi)} \psi^\theta \sum\limits_{j=L}^\infty\left| \lambda_j\right|^{u_1(1+\psi)}\sum\limits_{t_0=0}^L2^{(t_0-j)\alpha_\infty u_1(1+\psi)}\\
      \lesssim & \sup_{\psi>0}\sup _{L \geq  0, L \in \mathbb{Z}} 2^{-L\Gamma  u_1(1+\psi)} \psi^\theta \sum\limits_{j=0}^L\left| \lambda_j\right|^{u_1(1+\psi)}+  \sup_{\psi>0}\sup _{L \geq  0, L \in \mathbb{Z}}  \psi^\theta \sum\limits_{j=L}^\infty 2^{(j\Gamma-L\Gamma)  u_1(1+\psi)}2^{-j \Gamma u_1(1+\psi)}\sum\limits_{i=-\infty}^j\left| \lambda_i\right|^{u_1(1+\psi)}\sum\limits_{t_0=0}^L2^{(t_0-j)\alpha_\infty u_1(1+\psi)}\\
      \lesssim & \Lambda + \Lambda \sup_{\psi>0}\sup _{L \geq  0, L \in \mathbb{Z}}  \psi^\theta \sum\limits_{j=L}^\infty 2^{(j-L)  u_1(1+\psi)(\Gamma-\alpha_\infty)} \\  
      \lesssim & \Lambda. 
\end{align*}
When $1<u_1\leq \infty$, by using the boundedness of $\left[b, T_{\Omega, \sigma}\right]$ in $L^{p(\cdot)}$, we have

\begin{align*}
I I I_{1}&=  \sup_{\psi>0}\sup _{L <  0, L \in \mathbb{Z}} 2^{-L \Gamma u_1(1+\psi)} \psi^\theta \sum_{t_0=0}^{L} 2^{t_0\alpha _{\infty}u_1(1+\psi)}\left(\sum_{j=k}^{\infty}\left|\lambda_{j}\right|\left\|\left[b, T_{\Omega, \sigma}\right] \chi_{t_0}\right\|_{L^{p_2(\cdot)}}\right)^{u_1(1+\psi)} \\
& \lesssim \sup_{\psi>0}\sup _{L <  0, L \in \mathbb{Z}} 2^{-L \Gamma u_1(1+\psi)} \psi^\theta \sum_{t_0=0}^{L} 2^{\alpha_{\infty} t_0u _1(1+\psi)}\left(\sum_{j=k}^{\infty}\left|\lambda_{j}\right|^{u_1(1+\psi)}\left\|\left[b, T_{\Omega, \sigma}\right] \chi_{t_0}\right\|_{L^{p_2(\cdot)}}^{u_1(1+\psi) / 2}\right) \\
& \times\left(\sum_{j=k}^{\infty}\left\|\left[b, T_{\Omega, \sigma}\right] \chi_{t_0}\right\|_{L^{p_2(\cdot)}}^{(u_1(1+\psi))^{\prime} / 2}\right)^{u_1(1+\psi) / (u_1(1+\psi))^{\prime}} \\
& \lesssim  \sup_{\psi>0}\sup _{L <  0, L \in \mathbb{Z}} 2^{-L \Gamma u_1(1+\psi)} \psi^\theta \sum_{t_0=0}^{L} 2^{\alpha_{\infty} t_0u _1(1+\psi)}\left(\sum_{j=k}^{\infty}\left|\lambda_{j}\right|^{u_1(1+\psi)}\left\|b_{j}\right\|_{L^{p_1(\cdot)}}^{u_1(1+\psi) / 2}\right) \left(\sum_{j=k}^{\infty}\left\|b_{j}\right\|_{L^{p_1(\cdot)}}^{(u_1(1+\psi))^{\prime} / 2}\right)^{u_1(1+\psi) / (u_1(1+\psi))^{\prime}}\\
& \lesssim  \sup_{\psi>0}\sup _{L <  0, L \in \mathbb{Z}} 2^{-L \Gamma u_1(1+\psi)} \psi^\theta \sum_{t_0=0}^{L} 2^{\alpha_{\infty} t_0u _1(1+\psi)}\left(\sum_{j=k}^{\infty}\left|\lambda_{j}\right|^{u_1(1+\psi)}\left|B_{j}\right|^{-\alpha_{j} u_1(1+\psi) /(2 n)}\right) \left(\sum_{j=k}^{\infty}\left|B_{j}\right|^{-\alpha_{j} (u_1(1+\psi))^{\prime} /(2 n)}\right)^{u_1(1+\psi) / (u_1(1+\psi))^{\prime}} \\
& \lesssim  \sup_{\psi>0}\sup _{L <  0, L \in \mathbb{Z}} 2^{-L \Gamma u_1(1+\psi)} \psi^\theta \sum_{t_0=0}^{L} 2^{\alpha_{\infty} t_0u _1(1+\psi)}\left(\sum_{j=k}^{\infty}\left|\lambda_{j}\right|^{u_1(1+\psi)}\left|B_{j}\right|^{-\alpha_{j} u_1(1+\psi) /(2 n)}\right)\\
& = \sup_{\psi>0}\sup _{L <  0, L \in \mathbb{Z}} 2^{-L \Gamma u_1(1+\psi)} \psi^\theta \sum_{j=0}^{L}\left|\lambda_{j}\right|^{u_1(1+\psi)} \sum_{t_0=0}^{j} 2^{(t_0-j) \alpha_{\infty} u_1(1+\psi) / 2}  +\sup_{\psi>0}\sup _{L <  0, L \in \mathbb{Z}} 2^{-L \Gamma u_1(1+\psi)} \psi^\theta \sum_{j=L}^{\infty}\left|\lambda_{j}\right|^{u_1(1+\psi)} \sum_{t_0=0}^{L} 2^{(t_0-j) \alpha_{\infty} u_1(1+\psi) / 2} \\
& \lesssim  \sup_{\psi>0}\sup _{L <  0, L \in \mathbb{Z}} 2^{-L \Gamma u_1(1+\psi)} \psi^\theta \sum_{j=0}^{L}\left|\lambda_{j}\right|^{u_1(1+\psi)}  +\sup_{\psi>0}\sup _{L <  0, L \in \mathbb{Z}}  \psi^\theta \sum_{j=L}^{\infty} 2^{(j \Gamma -L \Gamma) u_1(1+\psi)} 2^{-j \Gamma u_1(1+\psi)} \sum_{i=-\infty}^{j}\left|\lambda_{i}\right|^{u_1(1+\psi)} \sum_{t_0=0}^{L} 2^{(t_0-j) \alpha_{\infty} u_1(1+\psi) / 2} \\
& \lesssim \Lambda+\Lambda \sup_{\psi>0}\sup _{L <  0, L \in \mathbb{Z}}  \psi^\theta \sum_{j=L}^{\infty} 2^{(j-L) \Gamma u_1(1+\psi)} 2^{(L-j) \alpha_{\infty} u_1(1+\psi) / 2} \\
& \lesssim \Lambda+\Lambda \sup _{L <  0, L \in \mathbb{Z}}  \psi^\theta \sum_{j=L}^{\infty} 2^{(j-L) u_1(1+\psi)\left(\Gamma-\alpha_{\infty} / 2\right)} \\
& \lesssim \Lambda .
\end{align*}

Now we will find the estimate for $III_2$. When $0<u_1 \leq  1$,  we get

\begin{align*}
I I I_{2}= & \sup_{\psi>0}\sup _{L <  0, L \in \mathbb{Z}} 2^{-L \Gamma u_1(1+\psi)} \psi^\theta \sum_{t_0=0}^{L} 2^{\alpha_{\infty} t_0u _1(1+\psi)}\left(\sum_{j=-\infty}^{t_0-1}\left|\lambda_{j}\right|\left\|\left[b, T_{\Omega, \sigma}\right] \chi_{t_0}\right\|_{L^{p_2(\cdot)}}\right)^{u_1(1+\psi)} \\
\lesssim & \sup_{\psi>0}\sup _{L <  0, L \in \mathbb{Z}} 2^{-L \Gamma u_1(1+\psi)} \psi^\theta  \sum_{t_0=0}^{L} 2^{\alpha_{\infty} t_0u _1(1+\psi)}\left(\sum_{j=-\infty}^{t_0-1}\left|\lambda_{j}\right|^{u_1(1+\psi)} 2^{[-j \alpha_j +(j-t_0)(\gamma+n\delta_2)] u_1(1+\psi)}\right) \\
= & \sup_{\psi>0}\sup _{L <  0, L \in \mathbb{Z}} 2^{-L \Gamma u_1(1+\psi)} \psi^\theta\\
& \times \sum_{t_0=0}^{L} 2^{\alpha_{\infty} t_0u _1(1+\psi)}\left(\sum_{j=-\infty}^{-1}\left|\lambda_{j}\right|^{u_1(1+\psi)} 2^{\left[(\gamma+n\delta_2)(j-t_0)-j \alpha(0)\right] u_1(1+\psi)}\right) \\
+& \sup_{\psi>0}\sup _{L <  0, L \in \mathbb{Z}} 2^{-L \Gamma u_1(1+\psi)} \psi^\theta \\
& \times \sum_{t_0=0}^{L} 2^{\alpha_{\infty} t_0u _1(1+\psi)}\left(\sum_{j=0}^{t_0-1}\left|\lambda_{j}\right|^{u_1(1+\psi)} 2^{\left[(\gamma+n\delta_2)(j-t_0)-j \alpha_{\infty}\right] u_1(1+\psi)}\right)\\
\lesssim & \sup_{\psi>0}\sup _{L <  0, L \in \mathbb{Z}} 2^{-L \Gamma u_1(1+\psi)} \psi^\theta \sum_{t_0=0}^{L} 2^{t_0\left[\alpha_{\infty}-(\gamma+n\delta_2)\right]  u_1(1+\psi)} \\
& \times \sum_{j=-\infty}^{-1}\left|\lambda_{j}\right|^{u_1(1+\psi)} 2^{\left(\gamma+n\delta_2-\alpha(0)\right) j u_1(1+\psi)} \\
+& \sup_{\psi>0}\sup _{L <  0, L \in \mathbb{Z}} 2^{-L \Gamma u_1(1+\psi)} \psi^\theta \sum_{t_0=0}^{L}\left|\lambda_{j}\right|^{u_1(1+\psi)} \sum_{t_0=j+1}^{\infty} 2^{(j-t_0)\left(\gamma+n\delta_2-\alpha_{\infty}\right) u_1(1+\psi)} \\
\lesssim & \sup_{\psi>0}\sup _{L <  0, L \in \mathbb{Z}} 2^{-L \Gamma u_1(1+\psi)} \psi^\theta \sum_{j=-\infty}^{-1}\left|\lambda_{j}\right|^{u_1(1+\psi)}+\sup_{\psi>0}\sup _{L <  0, L \in \mathbb{Z}} 2^{-L \Gamma u_1(1+\psi)} \psi^\theta \sum_{j=0}^{L-1}\left|\lambda_{j}\right|^{u_1(1+\psi)} \\
\lesssim & \Lambda .
\end{align*}

When $1<u_1(1+\psi)<\infty$ and since $n \delta_{2} \leq  \alpha(0)$, $\alpha_{\infty}<\gamma+n\delta_2$, by H\"older's inequality, we have

\begin{align*}
I I I_{2}=& \sup_{\psi>0}\sup _{L <  0, L \in \mathbb{Z}} 2^{-L \Gamma u_1(1+\psi)} \psi^\theta \sum_{t_0=0}^{L} 2^{\alpha_{\infty} t_0u _1(1+\psi)}\left(\sum_{j=-\infty}^{t_0-1}\left|\lambda_{j}\right|\left\|\left[b, T_{\Omega, \sigma}\right] \chi_{t_0}\right\|_{L^{p_2(\cdot)}}\right)^{u_1(1+\psi)} \\
 \lesssim & \sup_{\psi>0}\sup _{L <  0, L \in \mathbb{Z}} 2^{-L \Gamma u_1(1+\psi)} \psi^\theta \sum_{t_0=0}^{L} 2^{\alpha_{\infty} t_0u _1(1+\psi)}\left(\sum_{j=-\infty}^{t_0-1}\left|\lambda_{j}\right| 2^{\left[(\gamma+n\delta_2)(j-t_0)-j \alpha_{j}\right]}\right)^{u_1(1+\psi)} \\
\lesssim & \sup_{\psi>0}\sup _{L <  0, L \in \mathbb{Z}} 2^{-L \Gamma u_1(1+\psi)} \psi^\theta \sum_{t_0=0}^{L} 2^{\alpha_{\infty} t_0u _1(1+\psi)}\left(\sum_{j=-\infty}^{-1}\left|\lambda_{j}\right| 2^{\left[(\gamma+n\delta_2)(j-t_0)-j \alpha(0)\right]}\right)^{u_1(1+\psi)} \\
 +&\sup_{\psi>0}\sup _{L <  0, L \in \mathbb{Z}} 2^{-L \Gamma u_1(1+\psi)} \psi^\theta \sum_{t_0=0}^{L} 2^{\alpha_{\infty} t_0u _1(1+\psi)}\left(\sum_{j=0}^{t_0-1}\left|\lambda_{j}\right| 2^{\left[(\gamma+n\delta_2)(j-t_0)-j \alpha_{\infty}\right]}\right)^{u_1(1+\psi)} \\
\lesssim & \sup_{\psi>0}\sup _{L <  0, L \in \mathbb{Z}} 2^{-L \Gamma u_1(1+\psi)} \psi^\theta \sum_{t_0=0}^{L} 2^{\left[\alpha_{\infty}-(\gamma+n\delta_2)\right] t_0u _1(1+\psi)} \\
 &\times \left(\sum_{j=-\infty}^{-1}\left|\lambda_{j}\right| 2^{\left(\gamma+n\delta_2-\alpha(0)\right) j}\right)^{u_1(1+\psi)} \\
+& \sup_{\psi>0}\sup _{L <  0, L \in \mathbb{Z}} 2^{-L \Gamma u_1(1+\psi)} \psi^\theta \sum_{t_0=0}^{L}\left(\sum_{j=0}^{t_0-1}\left|\lambda_{j}\right| 2^{(j-t_0)\left(\gamma+n\delta_2-\alpha_{\infty}\right)}\right)^{u_1(1+\psi)} \\
\lesssim & \left(\sup_{\psi>0}\sup _{L <  0, L \in \mathbb{Z}} 2^{-L \Gamma u_1(1+\psi)} \psi^\theta\sum_{j=-\infty}^{-1}\left|\lambda_{j}\right|^{u_1(1+\psi)} 2^{\left(\gamma+n\delta_2-\alpha(0)\right) j u_1(1+\psi) / 2}\right) \\
 &\times\left(\sum_{j=-\infty}^{-1} 2^{\left(\gamma+n\delta_2-\alpha(0)\right) j (u_1(1+\psi))^{\prime} / 2}\right)^{u_1(1+\psi) / (u_1(1+\psi))^{\prime}} \\
\lesssim & \sup_{\psi>0}\sup _{L <  0, L \in \mathbb{Z}} 2^{-L \Gamma u_1(1+\psi)} \psi^\theta \sum_{t_0=0}^{L}\left(\sum_{j=0}^{t_0-1}\left|\lambda_{j}\right|^{u_1(1+\psi)} 2^{(j-t_0)\left(\gamma+n\delta_2-\alpha_{\infty}\right) u_1(1+\psi) / 2}\right) \\
& \times\left(\sum_{j=0}^{t_0-1} 2^{(j-t_0)\left(\gamma+n\delta_2-\alpha_{\infty}\right) (u_1(1+\psi))^{\prime} / 2}\right)^{u_1(1+\psi) / (u_1(1+\psi))^{\prime}}\\
\lesssim & \sup_{\psi>0}\sup _{L <  0, L \in \mathbb{Z}} 2^{-L \Gamma u_1(1+\psi)} \psi^\theta \sum_{j=-\infty}^{-1}\left|\lambda_{j}\right|^{u_1(1+\psi)} 2^{\left(\gamma+n\delta_2-\alpha(0)\right) j u_1(1+\psi) / 2} \\
& +\sup_{\psi>0}\sup _{L <  0, L \in \mathbb{Z}} 2^{-L \Gamma u_1(1+\psi)} \psi^\theta \sum_{t_0=0}^{L} \sum_{j=0}^{t_0-1}\left|\lambda_{j}\right|^{u_1(1+\psi)} 2^{(j-t_0)\left(\gamma+n\delta_2-\alpha_{\infty}\right) u_1(1+\psi) / 2} \\
\lesssim & \sup_{\psi>0}\sup _{L <  0, L \in \mathbb{Z}} 2^{-L \Gamma u_1(1+\psi)} \psi^\theta \sum_{j=-\infty}^{-1}\left|\lambda_{j}\right|^{u_1(1+\psi)} \\
+& \sup_{\psi>0}\sup _{L <  0, L \in \mathbb{Z}} 2^{-L \Gamma u_1(1+\psi)} \psi^\theta \sum_{j=0}^{L-1}\left|\lambda_{j}\right|^{q} \sum_{t_0=j+1}^{L} 2^{(j-t_0)\left(\gamma+n\delta_2-\alpha_{\infty}\right) u_1(1+\psi) / 2} \\
\lesssim & \sup_{\psi>0}\sup _{L <  0, L \in \mathbb{Z}} 2^{-L \Gamma u_1(1+\psi)} \psi^\theta \sum_{j=-\infty}^{-1}\left|\lambda_{j}\right|^{u_1(1+\psi)}+\sup_{\psi>0}\sup _{L <  0, L \in \mathbb{Z}} 2^{-L \Gamma u_1(1+\psi)} \psi^\theta \sum_{j=0}^{L-1}\left|\lambda_{j}\right|^{u_1(1+\psi)}\\
\lesssim & \Lambda .
\end{align*}

Thus we have obtained our desired results. 

\end{proof}

\section*{Conflict of Interests}
The authors declare that there is no conflict of interests regarding the publication of this paper.

\section*{Data Availability}
	No data were used to support this study.

\bibliographystyle{amsplain}

\begin{thebibliography}{99}

\bibitem{grafakos2}
A. Beurling, Construction and analysis of some convolution algebras, \textit{Ann. Inst. Fourier Grenoble.} { \bf 14} (1964), 1-32. MR 31:321

\bibitem{grafakos1}
A.Baernstein II and E. T. Sawyer, Embedding and Multiplier Theorems for $H^p(\mathbb{R}^n),$ \textit{Memoirs of the Amer. Math. Soc.}  { \bf 59} ( 318) (1985) 

\bibitem{uribe}
D. Cruz-Uribe, A. Fiorenza, J.M.Martell and C. Pe rez, The boundedness of classical operators on variable
$L^p$ spaces, Ann.Acad. Sci. Fen.Math., 31 (2006), 239-264.

\bibitem{uribe1}
D. Cruz-Uribe and A. Fiorenza, Variable Lebesgue Spaces: Foundations and Harmonic Analysis(
Applied and Numerical Harmonic Analysis) (Springer, Heidelberg, 2013).


\bibitem{ref6} L. Diening, P. Harjulehto, P. H\"ast\"o and M. Ruzicka, Lebesgue and Sobolev Spaces with Variable Exponents, Springer, Berlin, 2011.

\bibitem{ding1}
Y. Ding, S. Lu, Homogeneous fractional integrals on Hardy spaces, Tohoku Math. J., 52 (2000), 153-162.

\bibitem{grafakos15}
C. Herz, Lipschitz spaces and Bernstein’s theorem on absolutely convergent Fourier transforms, \textit{J.Math.Mech.} {\bf 18}  (1968), 283-324.

\bibitem{grafakos3}
Y. Z. Chen and K. S. Lau, Some new classes of Hardy spaces, \textit{J. Funct. Anal.} {\bf 84} (1989), 255-278.

\bibitem{grafakos10}
J. Garc\'ia-Cuerva, Hardy spaces and Beurling algebras, \textit{J. London Math. Soc.} {\bf 2} (39) (1989), 499-513. 

\bibitem{grafakos18}
S. Lu and D. Yang, The Littlewood-Paley function and $\phi$-transform characterizations of a new Hardy space $HK_2$ associated with the Herz space, \textit{Studia Math.} {\bf 101} (1992), 285-298.  

\bibitem{grafakos19}
 S. Lu and D. Yang, Some new Hardy spaces associated with the Herz spaces and their applications (in Chinese), \textit{J. of Beijing Normal Univ. Natur. Sci.} {\bf 29} (1993), 10-19.
 
 
 \bibitem{grafakos20}
  S. Lu and D. Yang, The local versions of $H^p(\mathbb{R}^n)$ spaces at the origin, \textit{Studia Math.} {\bf 116} (1995), 103-131.
  
  \bibitem{grafakos21}
   S. Lu and D. Yang, The decomposition of weighted Herz space on $R^n$ and its application, \textit{Science in China (Scientia Sinica) Series A} {\bf 38} (1995), 147-158. 
   
   \bibitem{grafakos22}
    S. Lu and D. Yang, The weighted Herz-type Hardy spaces and its applications, \textit{Chinese Sci, Abst. Series A.} {\bf 38} (1995), 662-673. 
    
  


\bibitem{grafakos11}
 J. Garc\'ia-Cuerva and M.-J. L. Herrero, A theory of Hardy spaces associated to the Herz spaces, \textit{Proc. London Math. Soc.} {\bf  69} (3) (1994), 605-628.


\bibitem{chai20}
T. Anh, J. Cao, L. D. Ky, D. Yang, and S. Yang, Weighted Hardy spaces associated with operators satisfying reinforced off-diagonal estimates,\textit{ Taiwan. J. Math.}  {\bf 17} (4) (2013), 1127–1166.

\bibitem{chai23}
X. Fu, H. Lin, D. Yang, and D. Yang, Hardy spaces $H^p$ over non-homogeneous metric measure spaces and their applications, \textit{Sci. China Math.}  {\bf 58} (2) (2015), 309–388.

\bibitem{chai24}
R. Gong, J. Li, and L. Yan, A local version of Hardy spaces associated with operators on metric spaces, \textit{Sci. China Math.} {\bf 56} (2) ( 2013) 315–330.

\bibitem{chai28}
D. Yang and S. Yang, Local Hardy spaces of Musielak-Orlicz type and their applications, \textit{Sci. China Math.} {\bf 55} (8) (2012) 1677–1720.

\bibitem{chai29}
K. Zhao and Y. Han, Boundedness of operators on Hardy spaces, \textit{Taiwan. J. Math.} {\bf 14} ( 2) (2010), 319–327.

\bibitem{chai30}
J. Dong, J. Huang, and H. Liu, Boundedness of singular integrals on Hardy type spaces associated with Schr\"odinger operators, \textit{J. Funct. Spaces.} {\bf 2015} (2015),  1–11.




\bibitem{babargh2}
B. Sultan, M. Sultan, Boundedness of commutators of rough Hardy operators on grand variable Herz spaces, \textit{Forum Math.} (2023), https://doi .org /10 .1515 /forum -2023 -0152.



\bibitem{babargh4}
B. Sultan, M. Sultan, I. Khan, On Sobolev theorem for higher commutators of fractional integrals in grand variable Herz spaces, \textit{Commun. Nonlinear Sci. Numer. Simul.} 126 (2023).

\bibitem{ghm1}
B. Sultan, F. Azmi, M. Sultan, M. Mehmood, N. Mlaiki, Boundedness of Riesz potential operator on grand Herz-Morrey spaces, \textit{Axioms.} {\bf 11}(11) (2022) 583.

\bibitem{ghm2}
M. Sultan, B. Sultan, A. Khan, T. Abdeljawad, Boundedness of Marcinkiewicz integral operator of variable order in grand Herz-Morrey spaces, \textit{AIMS Math.} {\bf 8}(9) (2023) 22338-22353.


\bibitem{zliu}
Zongguang Liu. Hongbin Wang. Some characterizations of Herz-type Hardy spaces with variable exponent. \textit{Ann. Funct. Anal.} {\bf6 }(2) (2015) 224 - 233.




\bibitem{ghm4}
B. Sultan, M. Sultan, Boundedness of higher order commutators of Hardy operators on grand Herz-Morrey spaces, \textit{Bull. Sci.math.} {\bf 190}(2024) 103373.


\bibitem{ghm5}
M. Sultan, B. Sultan, A. Hussain, Grand Herz–Morrey Spaces with variable exponent, \textit{Math. Notes.}  {\bf 114} (5) (2023),  957–977.

\bibitem{hmh}
J. Xu, X. Yang, Herz-Morrey-Hardy Spaces with variable vxponents and their applications, \textit{J. Funct. Spaces.}  {\bf 2015} (2015),  19 pages. 

\bibitem{n17}
M. Izuki, Boundedness of sublinear operators on Herz spaces with variable exponent and application to wavelet characterization, \textit{Anal. Math.} {\bf36} (2010), 33-50.

\bibitem{n19}
O. Kov\'acik, J. R\'akosn\'ik,  On spaces $L^{p(x)}$
and $W^{k,p(x)}$, \textit{Czechoslov. Math. J.} {\bf 41} (1991), 592-618.

\bibitem{yanchai}
Y. Chai, Y. Han, K. Zhao, Herz-Type Hardy Spaces Associated with Operators, \textit{J. Funct. Spaces.}  {\bf 2018} (2018), 10 pages.


\bibitem{ding}
Y. Ding, Weighted boundedness for commutators of integral operators of fractional order with rough kernel
(in Chinese), Beijing Shifan Daxue Xuebao, 32 (1996, 157-161 .



\bibitem{xu7} M. Ruzicka, Electrorheological Fluids: Modeling and Mathematical Theory, vol. 1748 of Lecture Notes in Mathematics, Springer, Berlin, Germany, 2000.

\bibitem{xu8} P. Harjulehto, P. H\"ast\"o, \'Ut. V. L\'e, and M. Nuortio, Overview of differential equations with non-standard growth, \textit{Nonlinear Anal. Theory Methods Appl.} {\bf 72}( 12)(2010),  4551-4574.

\bibitem{xu9} Y. Chen, S. Levine, and M. Rao, Variable exponent, linear growth functionals in image restoration, \textit{SIAM J. Appl. Math.}  {\bf66} (4) (2006),  1383-1406.

\bibitem{xu10} F. Li, Z. Li, and L. Pi, Variable exponent functionals in image restoration, \textit{Appl. Math. Comput.}  {\bf 216}(3) (2010) 870-882.

\bibitem{xu11} P. Harjulehto, P. H\"ast\"o, V. Latvala, and O. Toivanen, Critical variable exponent functionals in image restoration, \textit{Appl. Math. Lett.} {\bf 26} (1) (2013) 56-60.


\bibitem{xu13} A. Almeida and P. H\"ast\"o, Besov spaces with variable smoothness and integrability, \textit{J. Funct. Anal.} {\bf 258} (5) (2010)  1628-1655.

\bibitem{xu14} L. Diening, P. Häst\"o, and S. Roudenko, Function spaces of variable smoothness and integrability, \textit{J. Funct. Anal.} {\bf 256}(6) (2009),  1731-1768.


\bibitem{xu20} H. Kempka, 2-microlocal Besov and Triebel-Lizorkin spaces of variable integrability, \textit{Rev. mat. complut.} {\bf 22} (1) (2009),227-251.

\bibitem{xu21} H. Kempka, Atomic, molecular and wavelet decomposition of generalized 2-microlocal Besov spaces, \textit{J. Funct. spaces and app.} {\bf 8} (2) (2010), 129-165.

\bibitem{xu22} E. Nakai and Y. Sawano, Hardy spaces with variable exponents and generalized Campanato spaces, \textit{J. Funct. Anal.} {\bf 262} (9) (2012), 3665-3748.



\bibitem{xu24} J.-S. Xu, An atomic decomposition of variable Besov and Triebel-Lizorkin spaces, \textit{Armen. J. Math.} {\bf 2} (1) (2009),  1-12.

\bibitem{xu25} T. Noi, Duality of variable exponent Triebel-Lizorkin and Besov spaces, \textit{J. Funct. Spaces and App.} {\bf  2012} (2012),  19 pages..

\bibitem{xu26} Y. Sawano, Atomic decompositions of Hardy spaces with variable exponents and its application to bounded linear operators, \textit{Integral Equ. Oper. Theory.}  {\bf 77}(1) (2013), 123-148.

\bibitem{xu31}
H. Wang and Z. Liu, The Herz-type Hardy spaces with
variable exponent and their applications, {Taiwan. J. Math.} {\bf 16} (4) (2012),  1363–1389.

\bibitem{cald15}
A. Almeida and D. Drihem, Maximal, potential and singular type operators on Herz spaces with variable exponents, \textit{J. Math. Anal. Appl.}  {\bf 394} (2) (2012), 781-795.


\bibitem{48}
S. Lu and D. Yang, The local versions of $H^p(\mathbb{R}^n)$ spaces at the origin, \textit{Stud. Math.} {\bf 116} (1995),  103–131.





  \bibitem{grafakos23}
S. Lu and D. Yang, Some characterizations of weighted Herz-type Hardy spaces and its applications, \textit{ Acta Math. Sinica.} {\bf 13} (1997), 45–58. 

\bibitem{tan}
J. Tan and Z. Liu, Some boundedness of homogeneous fractional integrals on variable exponent function
spaces, Acta Math. Sinica (Chin. Ser.), 58 (2015), 309-320.

\end{thebibliography}

\end{document}